\documentclass[reqno,11pt]{amsart}
\numberwithin{equation}{section}
\usepackage{amsmath}
\usepackage{amsthm}
\usepackage{epsfig}
\usepackage{psfrag}
\usepackage{graphicx}
\usepackage{graphpap,latexsym,epsf}
\usepackage{color}
\usepackage{amssymb,mathrsfs,enumerate}

\newcommand{\R}{\Rz}
\newcommand{\N}{\Nz}
\newcommand{\Rz}{\mathbb{R}}
\newcommand{\Nz}{\mathbb{N}}
\newcommand{\up}{\to}

\newcommand{\eps}{\varepsilon}
\newcommand{\U}{U}
\newcommand{\ls}{|\partial \phi|}       
\newcommand{\rls}{|\partial^- \phi|}    
\newcommand{\rsls}{|\partial^+ \phi|}  
\newcommand{\weak}{\buildrel \sigma \over \to} 
\newcommand{\weakto}{\rightharpoonup} 
\newcommand{\wweak}{\buildrel \sigma\times \sigma \over \to}     

\newcommand{\du}{d_\phi}
\newcommand{\ds}{d_\sigma}
\newcommand{\loc}{\text{\rm loc}}
\newcommand{\forae}{\text{for a.a. }}

\newcommand{\xX}{{\mbox{\boldmath$X$}}}
\newcommand{\Flow}{\xX}
\newcommand{\boldpartial}{{\boldsymbol \partial}}

\newcommand{\sfl}{G}
\newcommand{\geneset}{\sfl}
\newcommand{\att}{A}
\newcommand{\cx}{\mathcal{X}}
\newcommand{\dcx}{d_\cx}
\newcommand{\rest}{Z(\sfl)}

\newcommand{\ttau}{{\boldsymbol \tau}}
\newcommand{\down}{\downarrow}
\newcommand{\weaksto}{\rightharpoonup^*}
\newcommand{\limitingSubdifferential}{\partial_{\ell}}
\newcommand{\lmsbd}{\limitingSubdifferential}
\newcommand{\stronglimitingSubderivative}{\partial_{s}}
\newcommand{\slmsbd}{\stronglimitingSubderivative}
\newcommand{\FrechetSubdifferential}{\partial}
\newcommand{\frsbd}{\FrechetSubdifferential}
\newcommand{\argmin}{\mathop{\rm Argmin}}
\newcommand{\AC}{\mathop{\rm AC}}

\newcommand{\Ene}{\mathcal{E}}
\newcommand{\Banach}{{\mathcal B}}
\newcommand{\Hilbert}{{\mathcal H}}
\newcommand{\dualoperator}

  \def\rmC{{\mathrm C}}

\def\rmJ{{\mathrm J}}

\newcommand{\dive}{\mathop{\rm div}}

\def\dd{\;\!\mathrm{d}} 
\newcommand{\ddd}{\mathrm{d}} 

\newcommand{\pairing}[4]{ \sideset{_{ #1 }}{_{ #2 }}  {\mathop{\langle #3 , #4
\rangle}}}
\newcommand{\Diff}{{\mathsf{D}}}

\newcommand{\teta}{\vartheta}
\newcommand{\zzeta}{\mathbf{\zeta}}
\newcommand{\nn}{\mathbf{n}}
\newcommand{\nnu}{\mathbf{\nu}}
\newcommand{\nchi}{{\raise.2ex\hbox{$\chi$}}}
\newcommand{\phims}{\phi_{\mathrm{MS}}}
\newcommand{\phiste}{\phi_{\mathrm{Ste}}}

\definecolor{ddcyan}{rgb}{0,0.1,0.9}
\definecolor{ddmagenta}{rgb}{0.8,0,0.8}
\definecolor{orange}{rgb}{0.6,0.2,0}

\newcommand{\Utau}{U_{\ttau}}
\newcommand{\piecewiseConstant}[2]{\overline{#1}_{\kern-1pt#2}}

 \def\trait #1 #2 #3 {\vrule width #1pt height #2pt depth #3pt}

 \def\fin{\hfill
         \trait .3 5 0
         \trait 5 .3 0
         \kern-5pt
         \trait 5 5 -4.7
         \trait 0.3 5 0
 \medskip}
\newtheorem{theorem}{Theorem}[section]
\newtheorem{remark}[theorem]{Remark}
\newtheorem{corollary}[theorem]{Corollary}
\newtheorem{definition}[theorem]{Definition}
\newtheorem{proposition}[theorem]{Proposition}
\newtheorem{problem}[theorem]{Problem}
\newtheorem{lemma}[theorem]{Lemma}
\newtheorem{example}[theorem]{Example}
\begin{document}

\title[Global attractors for gradient flows in metric spaces]{Global
attractors
for  gradient flows \\ in metric spaces}

\author{Riccarda Rossi}
  \address{Dipartimento di Matematica, Universit\`a di Brescia, via Valotti 9, I--25133 Brescia, Italy}
\email{\tt riccarda.rossi@ing.unibs.it}
\urladdr{http://dm.ing.unibs.it/~riccarda.rossi/}


\author{Antonio Segatti}
\address{Dipartimento di Matematica ``F. Casorati'', via Ferrata 1, I--27100 Pavia, Italy}
\email{antonio.segatti@unipv.it}
\urladdr{http://www-dimat.unipv.it/segatti/ita/res.php}

\author{Ulisse Stefanelli}
\address{IMATI -- CNR, via Ferrata 1, I--27100 Pavia, Italy}
\email{ulisse.stefanelli@imati.cnr.it}
\urladdr{http://www.imati.cnr.it/ulisse/}

\date{Oct. 16, 2009}

\thanks{R. R. and U. S. have been partially supported by a MIUR-PRIN'06 grant for the project
 \emph{Variational methods in optimal mass transportation and in geometric
 measure theory}. U. S. is also supported by the FP7-IDEAS-ERC-StG grant
\#200497 (BioSMA). A. S. has been partially
 supported by the Deutsche Forschungsgemeinschaft (DFG), within the Priority
Program SPP $1095$ \emph{Analysis, Modelling and Simulation of
Multiscale Problems}.}


\begin{abstract}
We develop the long-time analysis  for
 gradient flow equations  in metric spaces. In particular, we
consider two notions of solutions  for {\it metric  gradient flows},
 namely  \emph{energy} and \emph{generalized solutions}. While the
former concept    coincides with the notion of   \emph{curves of
maximal slope} of~\cite{Ambrosio05}, we introduce the latter to
include limits of time-incremental approximations constructed via
the \emph{Minimizing Movements} approach \cite{DeGiorgi93,Ambrosio95}.

 For both  notions of solutions we prove
 the existence of the global attractor. Since the
 evolutionary
problems we consider may lack uniqueness, we rely on the theory of
\emph{generalized semiflows} introduced in~\cite{Ball97}.

The notions of generalized and energy solutions are quite flexible and can
be used to address gradient flows in a variety of contexts,
 ranging from Banach spaces to Wasserstein
spaces of probability measures.
 We  present applications of our abstract results by proving
  the existence of the global
attractor for the \emph{energy solutions} both  of abstract doubly
nonlinear evolution equations in reflexive Banach spaces, and of  a
class of evolution equations in Wasserstein spaces, as well as  for
the \emph{generalized solutions} of some  phase-change
 evolutions driven by mean curvature.
\end{abstract}

\maketitle




\section{Introduction} \label{s:1} \setcounter{equation}{0}
\noindent  Gradient flows in metric spaces
 have   recently been   the object of
intensive studies.  We especially refer to the monograph by
\textsc{Ambrosio, Gigli \& Savar\'e}~\cite{Ambrosio05} for a
systematic analysis of {\it metric gradient flows}
 from the viewpoint of existence and approximation of solutions,
as well as their uniqueness. Moreover, some decay to equilibrium result has also been developed.

The main aim of this paper is to progress further in the analysis of the long-time dynamics for metric gradient flows. In particular, we focus on the existence of the global attractor for the
\emph{generalized semiflow}~\cite{Ball97} associated with suitable solution notions. Moreover, we present new convergence-to-equilibrium results and apply the abstract theory to doubly nonlinear equations in Banach spaces, to evolutions in Wasserstein spaces, and to some phase-change problems driven by mean curvature.

\paragraph{\bf Energy and generalized solutions.}
The possibility of discussing gradient flow dynamics in metric spaces relies on a suitable {\it scalar} formulation of the evolution problem. Let us develop some heuristics by starting from the classical setting of the
Euclidean  space $\R^d$ and from a proper, smooth functional $\phi$
on $\R^d$,  driving the gradient flow equation
\begin{equation}
\label{e:1.1} u'(t) + \nabla \phi (u(t)) =0  \qquad \text{in $\R^d$
} \ \ \forae\, t \in (0,T).
\end{equation}
Now, testing \eqref{e:1.1} by $u'(t)$ and taking into account the
\emph{chain rule}
\begin{equation}
\label{e:cr-1} \frac{\ddd}{\ddd t} (\phi \circ u) (t) = \langle
\nabla \phi (u(t)), u'(t) \rangle \qquad \forae\, t \in (0,T)\,
\end{equation}
 it is straightforward to check that \eqref{e:1.1} may be
 equivalently reformulated as 
\begin{equation}
\label{e:1.2}   \frac{\ddd}{\ddd t} (\phi \circ u) (t)= -\frac12
|u'(t)|^2 - \frac12 |\nabla \phi (u(t))|^2
 \qquad \forae\, t \in
(0,T).
\end{equation}

Equation \eqref{e:1.1} bears no sense in a metric space where a linear structure is missing. On the other hand, relation \eqref{e:1.2} is {\it scalar} and can serve as a possible notion of {\it metric gradient flow} evolution, as soon as one provides some possible metric surrogate for the {\it norm} of the time-derivative and the {\it norm} of a gradient (and not for the {\it full} time-derivative or the {\it full} gradient). These are provided by the scalar {\it metric derivative} $t \mapsto |u'|(t) $ and {\it slope} $t \mapsto |\partial \phi|(u(t))$, which are suitably defined for the metric-space-valued trajectory $t \mapsto u(t)\in (U,d)$ (see Section 2 for their definitions and properties). In particular, the {\it metric} (and doubly nonlinear) reinterpretation of \eqref{e:1.2} reads $(1<p<\infty, \ 1/p + 1/p'=1)$

\begin{equation}
\label{eq:formu-p}
\begin{gathered}
\text{the map $t \in (0,T) \mapsto \phi(u(t))$ is absolutely
continuous and}
\\
\frac{\ddd}{\ddd t} (\phi \circ u) (t)=  -\frac1p |u'|^{p'}(t)
-\frac1{p'}\ls^{q}(u(t))  \quad \forae \ t \in (0,T).
\end{gathered}
\end{equation}
Our stronger notion of solvability for metric gradient flows is that of {\it energy solutions}: namely, curves $t \mapsto u(t)$ fulfilling \eqref{eq:formu-p}, where $|\partial \phi|$ is replaced by a suitable relaxation. A weaker notion of solution we aim to consider is that of {\it generalized solutions}. The latter are intimately tailored to the concept of {\it Minimizing Movements} \cite{DeGiorgi93, Ambrosio95}, namely a natural limiting object arising in connection with time-discretization procedures.

We know from ~\cite{Ambrosio95, Ambrosio05} that, if
\begin{equation}
\label{COMP} \tag{\textrm{COMP}}
 \text{$\phi$ has compact sublevels in $\U$,}
\end{equation}
(and hence $\phi$ is lower semicontinuous), if
some conditional continuity holds  (see \eqref{cond-cont} in Section 2 below)  and (some relaxation of) $\ls$ fulfills a suitable {\it metric chain-rule inequality} (see \eqref{e:2.3} below),
 then the  time-incremental approximations  constructed via the {Minimizing Movements} scheme converge to an energy solution.

 Moreover, the \emph{sole
 compactness}
 \eqref{COMP} is sufficient to guarantee that
 approximate solutions converge to some curve $t \mapsto u(t)$ which
 satisfies a specific energy inequality (see \eqref{1}). The
 latter keeps track of the limit $t \mapsto \varphi(t)$ of the energy $\phi$ along the
 approximate solutions, which   does not, in general, coincide with
 $\phi \circ u$. 
  We shall  call the pair $(u,\varphi)$ a \emph{generalized
 solution} of the metric gradient flow. Note that energy solutions $t \mapsto u(t)$
give rise to generalized solutions $t \mapsto (u(t),\phi(u(t)))$.

\paragraph{\bf Main results.}
In~\cite[Chap.~4]{Ambrosio05} the large-time behavior \emph{energy solutions} was analyzed in case $p=2$, by requiring that the functional
\begin{equation}
\label{convex}
 \text{$\phi$ is $\lambda$-geodesically convex for
some $\lambda>0$.}
\end{equation}
Such condition extends the usual notion of  $\lambda$-convexity
to the metric framework (see \eqref{def:l-geod-convex}).
Under \eqref{convex}, it was proved that
 every trajectory $t \mapsto  u(t) $ exponentially
 converges  as $t \to +\infty$ to the
unique minimum point  $\bar{u}$ of $\phi$.    Moreover,
 the exponential decay of the energy $\phi(u(t))$ to the equilibrium energy
$\phi(\bar{u})$ was  obtained (see
also~\cite{Carrillo-McCann-Villani03, Carrilloal}).
 In addition, it was shown that,
 under some structural, convexity-type condition on the ambient
metric space,  the Cauchy problem
for~\eqref{eq:formu-p}, in the case $p=2$,
 admits a unique solution,
generating a $\lambda$-contracting semigroup.

In this paper,
 we shall complement the long-time behavior results
of~\cite{Ambrosio05} in two
 directions,
  namely the construction of the global attractor and the
 investigation of the convergence of single trajectories to
 stationary states.

  Firstly,  we shall prove that, under   very general  hypotheses on
 $\phi$, no structural assumption on the ambient space, no geodesic convexity on $\phi$,
   and for  general
 $1<p<\infty$,
 both the set of energy  and the set of generalized solutions admit
 a \emph{global attractor},  namely a maximal, compact, invariant set attracting all bounded sets in the phase space.
Our choice of the phase space is dictated by the energy $\phi$.
Indeed, the functional  $
\phi$ decreases along trajectories and is thus a {\it Lyapunov
function} for the system. In view of this,
 as  already pointed out  in~\cite{Rocca-Schimperna04, Segatti06},
as well as  in~\cite{Rossi-Segatti-Stefanelli08} (where the
long-time behavior
 of gradient flows of nonconvex functionals in Hilbert spaces was
investigated), it is significant to  set our long-time analysis in
the
 metric phase space
 \begin{equation}
 \label{p-space}
(\mathcal{U}, d_{\mathcal{U}}), \quad \mathcal{U}:=D(\phi), \quad d_{\mathcal{U}}(u,v):= d(u,v) +
|\phi(u)-\phi(v)|
\end{equation}
 for energy solutions $u$ and in
an analogously defined,
 \emph{augmented},  phase space for genera\-lized solutions
$(u,\varphi)$.

Due to the possible nonconvexity of the functional $\phi$ and to the
\emph{doubly nonlinear charac\-ter} of \eqref{eq:formu-p} for $p\neq
2$, uniqueness of energy  (and, a fortiori, of generalized)
solutions may genuinely fail.  In recent years, several approaches
have been developed  to extend the well-established theory of
attractors for semigroups (see, e.g., \cite{Temam88}), to
differential problems without uniqueness. In this connection,
without claiming completeness we may recall  the results
in~\cite{Sell73, Sell96, Chep-Vish95, Mell-Vall98, Mell-Vall00,
Caraballo03}.  Here, we shall specifically move within the frame
advanced by  \textsc{J.~M. Ball}
 \cite{Ball97,Ball04}, namely with the theory of \emph{generalized
 semiflows} (Appendix \ref{s:a2}).

The two main abstract theorems of this paper run as follows:
\begin{itemize}
  \item
compactness \eqref{COMP} and the boundedness of equilibria entail the existence of the global attractor for the generalized semiflow of {\it generalized solutions} (Thm. \ref{teor:1}).\\
 \item
 compactness \eqref{COMP}, conditional continuity \eqref{cond-cont}. the metric chain-rule inequality \eqref{e:2.3}, and the boundedness of equilibria in \eqref{p-space} yield the existence of the global attractor for the generalized semiflow of {\it energy solutions} (Thm. \ref{thm:ene}).
\end{itemize}
Note that, apart from the boundedness of equilibria, the existence of a global attractor is obtained {\it under the very same conditions} ensuring existence of the corresponding solution notions.

Moreover,  we shall extend the convergence to equilibrium results
of~\cite{Ambrosio05} to the case of energy solutions~\eqref{eq:formu-p} with $p
\neq 2$. As already
 pointed out  in~\cite[Rmk.~2.4.7]{Ambrosio05},  the key condition is  a
suitable  generalization of \eqref{convex}, in which the modulus of
convexity depends on the $p$-power of the distance, i.e.
\begin{equation}
\label{p-convex} \tag{\textrm{CONV}} \text{$\phi$ is
$(\lambda,p)$-geodesically convex for some $\lambda>0$,}
\end{equation}
(see Section~\ref{ss:2.2}).
Hence, in Theorem~\ref{thm:3} we shall prove that, if $\phi$
complies with~\eqref{COMP} and~\eqref{p-convex}, then
every trajectory $t \mapsto u(t) $ of~\eqref{eq:formu-p}
 converges as $t \to +\infty$ to the
unique minimum point  $\bar{u}$ of $\phi$  exponentially  fast,
again with energy $\phi(u(t))$  exponentially  decaying to
$\phi(\bar{u})$. This in particular entails that the global
attractor for energy solutions reduces to the singleton $\{
\bar{u}\}$. Notice however that uniqueness of solutions
to~\eqref{eq:formu-p}, even under \eqref{p-convex} and the  convexity
structural condition on $(\U,d)$ imposed
in~\cite[Chap.~4]{Ambrosio05},  still seems to be an open problem.
%
\paragraph{\bf Application to doubly nonlinear evolutions in Banach spaces.} 
Our first example of energy solutions  (see
Section~\ref{ss:cvx-gflows})  is provided by abstract doubly
nonlinear evolution equations in a reflexive Banach space $\Banach$,
of the type
\begin{equation}
\label{e:dne-intro} \rmJ_p (u'(t)) + \partial \phi (u(t)) \ni 0  \ \
\text{in $\Banach'$}\quad \forae\, t \in (0,T)\,,
 \end{equation}
where  $\rmJ_p : \Banach \rightrightarrows \Banach'$ is the
$p$-duality map (see~\eqref{p-dual} later on), $\phi: \Banach \to
(-\infty,+\infty]$ a proper, lower semicontinuous and convex
functional and $\partial \phi$ its subdifferential in the sense of
convex analysis.

The long-time behavior of  doubly nonlinear equations driven by
\emph{nonconvex} functionals has been recently investigated
in~\cite{Segatti06, Akagi08}, where the  existence of the global
attractor for~\eqref{e:dne-intro} has been obtained in the case
$\phi$ is $\lambda$-convex and, respectively,  when $\partial \phi$
is perturbed by a non-monotone multi-valued operator. We also refer
to~\cite{Schimperna-Segatti-Stefanelli07, Schimperna-Segatti08},
where convergence to equilibrium  for large times of the
trajectories of Allen-Cahn type equations has been proved.

In Section~\ref{s:6}, following the outline
of~\cite{Rossi-Mielke-Savare08}  we  shall apply our metric approach
to the long-time analysis of
\begin{equation}
\label{e:dne-intro-nonconvex} \rmJ_p (u'(t)) + \lmsbd \phi (u(t))
\ni 0 \ \ \text{in $\Banach'$}\quad \forae\, t \in (0,T)\,,
 \end{equation}
the \emph{limiting subdifferential} $\lmsbd \phi $ of  $\phi$ being
a suitably generalized gradient notion, tailored for nonconvex
functionals, cf.~\cite{Mordukhovich84, Mordukhovich06,
Rossi-Savare04}.
 We shall prove that energy solutions in the sense of \eqref{eq:formu-p} yield
solutions to~\eqref{e:dne-intro-nonconvex} and in fact characterize
the solutions to~\eqref{e:dne-intro-nonconvex} arising from the
metric approach. Then, from our abstract results we  shall deduce that the semiflow of the energy
solutions to~\eqref{e:dne-intro-nonconvex} admits a global
attractor, thus extending to the doubly nonlinear framework the
results obtained in~\cite{Rossi-Segatti-Stefanelli08} in a Hilbert
setting, for the gradient flow case $p=2$.
Following~\cite{Rossi-Mielke-Savare08}, we shall further exploit the
flexibility of the metric approach to tackle
\emph{quasivariational} doubly nonlinear evolutions of the type
\[
\partial \Psi_{u'}(u(t),u'(t)) + \lmsbd \phi(u(t)) \ni 0  \qquad
\text{in $\Banach'$} \quad \forae\,t \in (0,T)\,,
\]
in which   the dissipation functional $\Psi $ depends on the unknown
function $u$. 
\paragraph{\bf Application to $p$-gradient flows in Wasserstein spaces.}
 In a series of pioneering papers~\cite{Otto96, Jordan-Kinderlehrer-Otto98,
 Otto98, Otto01}, \textsc{F. Otto} proposed  a novel variational interpretation for a wide class of
 diffusion equations of the form
 \begin{equation}
\label{e:wide-class}
\begin{gathered}
\partial_t \rho - \mathrm{div} \left(\rho \nabla \left( \frac{\delta \mathscr{L}}{\partial \rho}\right)
\right)=0 \qquad \text{in $\R^d \times (0,+\infty)$,}
\\
\text{ with}  \ \ \rho(x,t) \geq 0 \ \ \forae\,(x,t) \in \R^d \times
(0,+\infty),
\\
 \text{and}  \ \  \int_{\R^d} \rho(x,t) \dd x=1, \int_{\R^d}
|x|^2\rho(x,t) \dd x <+\infty \ \ \text{for all $t \in
(0,+\infty)$,}
\end{gathered}
 \end{equation}
where  ${\delta \mathscr{L}}/{\partial \rho}$ is  the first
variation of an integral functional
 \[
\mathscr{L}(\rho)= \int_{\R^d} L(x,\rho(x),\nabla \rho(x)) \dd x
 \]
associated with a smooth Lagrangian $L: \R^d \times [0,+\infty)
\times \R^d \to \R$. Indeed, it was  shown that \eqref{e:wide-class}
can be interpreted as a gradient flow in the Wasserstein space
$\mathscr{P}_2 (\R^d)$ of the probability measures on $\R^d$ with
finite second moment, endowed with the $2$-Wasserstein distance (see
\cite{Villani1, Ambrosio05, Villani2}). In fact, Otto's formalism
paved the way to crucial developments in the study of equations of
the type~\eqref{e:wide-class}: in particular, mass-transportation
techniques  have turned out to be  key  tools for the study of the
asymptotic behavior of solutions.  While referring
to~\cite{Ambrosio05} for a thorough survey of results in this
direction, here we mention~\cite{Carrillo-McCann-Villani03,
Carrilloal}, investigating
 the   decay rates to equilibrium
of the solutions of the following drift-diffusion, nonlocal equation
\begin{equation}
\label{e:drift-intro}
\partial_t \rho-  \mathrm{div} \left(\rho \nabla\left( V + I'(\rho) + W\ast \rho\right)
\right)=0  \qquad \text{in $\R^d \times (0,+\infty)$,}
\end{equation}
where $V: \R^d \to \R$ is a confinement potential, $I: [0,+\infty)
\to \R$ a density of internal energy, $W:\R^d \to \R$ an interaction
potential and $\ast$ denotes the convolution product. On the other hand, the convergence  to
equilibrium for trajectories of the
 {\it doubly nonlinear  variant} of~\eqref{e:drift-intro} without the nonlocal term $W \ast
 \rho$, i.e.
\begin{equation}
\label{e:drift-intro-p-local}
\partial_t \rho-  \mathrm{div} \left(\rho j_{p'} \left(\nabla\left( V + I'(\rho) \right)
\right)\right)=
0 \qquad \text{in $\R^d \times (0,+\infty)$,}
\end{equation}
($j_{p'}(r):= |r|^{p' -2} r$), was studied in~\cite{Agueh2003}. Finally,
 the metric approach to the
analysis of equation~\eqref{e:drift-intro} was
systematically developed in~\cite{Ambrosio05}.

In this direction,
 in
Section~\ref{s:7} we shall tackle the long-time behavior of the
{\it nonlocal doubly nonlinear} drift-diffusion equation
\begin{equation}
\label{e:drift-intro-p-local-nuovo}
\partial_t \rho -\mathrm{div} \left(\rho j_{p'} \left(\nabla\left( V + I'(\rho) + W \ast \rho\right)
\right)\right)= 0 \qquad \text{in $\R^d \times (0,+\infty)$,}
\end{equation}
complementing the results of~\cite{Carrillo-McCann-Villani03,
Carrilloal, Agueh2003}.
 Indeed, reviewing the discussion in~\cite{Ambrosio05},
 we are going to show that
solutions to~\eqref{e:drift-intro-p-local-nuovo} can be obtained
from
 energy
solutions~\eqref{eq:formu-p} in the Wasserstein space
$(\mathscr{P}_p (\R^d), W_p)$ of the probability measures with
finite $p$-moment.  Hence, from our abstract results we shall deduce the existence of the global
attractor for the \emph{metric solutions}
to~\eqref{e:drift-intro-p-local-nuovo}.

\paragraph{\bf Application to  phase-change
 evolutions driven by mean curvature.} In
Sections~\ref{ss:3.2} and~\ref{ss:8.1} we shall
  consider the
Stefan-Gibbs-Thomson problem, modeling solid-liquid phase
transitions obeying the Gibbs-Thomson law for the temperature
evolution at the phase interface.
 The gradient flow structure of this problem was
first revealed in the paper~\cite{Luckhaus90} (see
also~\cite{Visintin96}). Therein, the existence of solutions was
proved by passing to the limit in the Minimizing Movements scheme.
The related energy functional does not fulfill the chain rule, which calls
for the notion of generalized solutions. By relying on
 our abstract theory, in Theorem~\ref{e:coro-rest-points} we shall
prove that the semiflow of the generalized solutions of the
Stefan-Gibbs-Thomson problem possesses a global attractor.

 Finally, in Section~\ref{ss:8.2}  we shall
 obtain some partial results for the generalized solutions
the Mullins-Sekerka flow~\cite{Luckhaus95, Roeger06}. However, we point out that, at the moment, the existence of the related global attractor is still open.

\paragraph{\bf Plan of the paper.}
In Section~\ref{ss:2.1}, we  specify the problem setup and recall
the basic notions on evolution in metric spaces which shall be used
throughout the paper. Hence, in Sections~\ref{ss:2.2}--\ref{ss:2.3} we
define the concepts of \emph{energy} and \emph{generalized}
solutions, and discuss  their relation.  In
Section~\ref{ss:cvx-gflows} we illustrate  energy solutions by addressing abstract doubly nonlinear equations in Banach spaces, whereas in Section~\ref{ss:3.2} we exemplify  the
concept of generalized solution via the Stefan-Gibbs-Thomson and
the Mullins-Sekerka flows. We state our main results
Theorem~\ref{teor:1}, \ref{thm:ene}, and~\ref{thm:3} on the
long-time behavior of generalized and energy solutions in
Section~\ref{s:4}, and  detail all proofs in Section \ref{s:5}.

 As for the application of our abstract results,  in Section~\ref{s:6} we study the connections between energy
solutions and doubly nonlinear equations driven by possibly
nonconvex functionals. In Section~\ref{s:7} we prove the
existence of the global attractor for the energy solutions of a
class of gradient flows in the Wasserstein space $(\mathscr{P}_p
(\R^d), W_p)$, while in Section~\ref{s:8}
generalized solutions are used to investigate the long-time behavior of the
Stefan-Gibbs-Thomson flow.

The Appendix is devoted to a concise presentation of the theory of
generalized semiflows by \textsc{J. M. Ball}, to the alternative
proof of a result in Section~\ref{ss:6.1}, and to some recaps on
maximal functions.

\paragraph{\bf Acknowledgments.} We would like to thank Giuseppe
Savar\'e for  many inspiring  and helpful discussions.


\parskip1mm



\section{Solution notions}
\label{s:2} \noindent In this section, we  introduce the two notions
of gradient flows in metric spaces which we shall consider in the
paper, see Definitions~\ref{def:ly-en} and~\ref{generalized} below.
Preliminarily, for the reader's convenience we
 briefly recall the  tools from analysis in metric spaces on which such
definitions rely,  referring   to~\cite{Ambrosio05}
 for a systematic treatment of
these issues.
\subsection{Problem setup} \setcounter{equation}{0}
\label{ss:2.1}
\paragraph{\bf The ambient space.} Throughout the paper
\begin{equation}
\label{e:setup} \begin{gathered} \text{ $(U,d)$ shall denote a
metric space, $\sigma$ a Hausdorff topology on $U$,  and}
\\
\text{
 $\phi : U \to (-\infty,+ \infty]$
a proper functional.}
\end{gathered}
\end{equation}
The symbol $ \weak $  shall stand for  the convergence in the
topology $ \sigma$,
 the metric $d$-convergence being denoted by $ \to $ instead.
We  shall precisely state the links between the $d$- and the
$\sigma$-topology in Section~\ref{ss:4.1}. However, to fix ideas one
may think that, in a Banach space setting, $d$ is the distance
induced by the norm and $\sigma$ is the  weak/weak$^*$ topology.

\paragraph{\bf Absolutely continuous curves and metric
derivative.} We say that a  curve $u: [0,T]\to U$ belongs to
$\AC^p([0,T];U)$,  $p\in [1,\infty]$,  if there exists $m\in
L^p(0,T)$ such that
\begin{equation}\label{metric_dev}
d(u(s),u(t))\leq \int_s^t m(r)\dd r \quad \text{for all $0 \leq
s\leq t \leq T.$}
\end{equation}
For $p=1$, we simply write $\AC([0,T];U)$ and refer to {\it
absolutely continuous curves}. A remarkable fact is that, for all $u
\in\AC^p([0,T];U) $, the limit
$$|u'|(t) = \lim_{s\to t}\frac{d(u(s),u(t))}{|t-s|}$$
exists for a.a. $t\in (0,T)$. We will refer to it as the {\it metric
derivative} of $v$ at $t$. We have that the map
 $t \mapsto|v'|(t)$  belongs to $ L^p(0,T) $
and it is minimal within the class of functions $m\in L^p(0,T)$
fulfilling \eqref{metric_dev}, see \cite[Sec. 1.1]{Ambrosio05}.

\paragraph{\bf The local and the strong relaxed slope.}
Let
$D(\phi):=\{u \in \U \ : \ \phi(u) < +\infty\}$
denote the effective domain of $ \phi $.
We define the {\it local slope} (see \cite{Ambrosio05,Cheeger99,DeGiorgi80}) of $ \phi $ at $ u \in D(\phi) $ as
$$
\ls(u) := \limsup_{v \to u} \frac{(\phi(u) - \phi(v))^+}{d(u,v)}.$$
\begin{remark}
\label{rmk:surrogate} \upshape  The local slope
 is a surrogate of the norm of $\nabla \phi$,  for  it can be
shown that, if $\U$ is a Banach space $\Banach$ with norm $\| \cdot
\|$ and $\phi: \Banach \to (-\infty,+\infty]$ is (Fr\'echet)
differentiable at $u \in D(\phi)$,  then $\ls(u) = \| {-}\nabla \phi
(u) \|_{*}.$
\end{remark}
The function $u \mapsto \ls(u)$ cannot be expected to be lower
semicontinuous. On the other hand, lower semicontinuity is
 desirable
 within limiting procedures.
For the purposes of the present long-time analysis, we shall
 deal with the following relaxation notion
\begin{equation}
\label{slope1}
 \rsls(u) :=\inf\big\{\liminf_{n\to +\infty}\ls(u_n)\ :  \ u_n \weak u,  \  \phi(u_n) \to \phi(u)
 \},\quad \text{for $u \in D(\phi)$.}
\end{equation}
We refer to $\rsls$ as the {\em strong relaxed slope} and remark that this relaxation is slightly stronger than the corresponding notion considered in \cite[Def. 2.3.1]{Ambrosio05} (where the approximating points $u_n$ are additionally required to belong to a $d$-ball but only have bounded energy $\phi$, possibly non converging.).

\paragraph{\bf A weaker relaxation for the local slope.}
 In the following we
shall make use of a second and weaker relaxation of the local slope.
This second notion  brings into play the limit of the energy $\phi$
along the approximating sequences $\{u_n\}$, and to this aim an
auxiliary variable is introduced. In particular, let the set
$$
{\mathcal X}:= \left\{(u,\varphi) \in D(\phi) \times \R : \ \varphi\geq
\phi(u)\right\}
$$
be given  and  define
\begin{equation}
\label{slope2}
 \rls(u,\varphi) :=\inf\big\{\liminf_{n\to +\infty}\ls(u_n)\ :  \ u_n \weak u, \  \phi(u_n) \to
 \varphi\geq \phi(u)
 \}
\end{equation}
for $(u,\varphi)\in {\mathcal X}$.

As soon as some compactness is assumed (see \eqref{A5} below), $\rls$
turns out to be  lower
semicontinuous with respect to both its arguments, i.e. for any sequence
$\{(u_n,\varphi_n)\} \subset {\mathcal X}$,
\begin{equation}
\label{double-lsc} (u_n \weak u, \quad \varphi_n \to \varphi) \
\Rightarrow \ \liminf_{n\to +\infty}\rls(u_n, \varphi_n) \geq
\rls(u, \varphi).
\end{equation}
Moreover, we clearly have that $\rsls(u) \geq \rls(u,\varphi)$ for
all $(u,\varphi)  \in {\mathcal X},$ and
\begin{align}
\label{disu-di-base} & \rsls(u)=\rls(u,\phi(u)) 
 \quad \forall u \in D(\phi).
\end{align}
\paragraph{\bf Strong upper gradient.}
By slightly strengthening \cite[Def.~1.2.1]{Ambrosio05}, we say that
a function
 $ g : \U \to [0,+\infty]$ is a  {\it strong upper gradient} for
  the functional $ \phi $ if, for every curve $ u \in \AC_\loc([0,+\infty);\U)$,
  the function $ g \circ u $ is Borel and
\begin{equation}
\label{e:2.3} |\phi(u(t)) - \phi(u(s))| \leq \int_s^t g(u(r))
|u'|(r)\dd r \quad \text{for all $ 0 \leq s\leq t$}
\end{equation}
(the original \cite[Def.~1.2.1]{Ambrosio05} requires \eqref{e:2.3}
for $s>0$ only). Let us explicitly observe that, whenever $ (g \circ
u) |u'| \in L^1_\loc([0,+\infty))$,  then $ \phi \circ u  \in
W_\loc^{1,1} ([0,+\infty))$ and
$$
 |(\phi \circ u)'(t)| \leq g(u(t)) |u'|(t) \quad
\forae t \in (0, + \infty). 
$$
 \paragraph{\bf Curves of Maximal Slope.}
 Let $g :\U \to [0,+\infty]$ be a {\em strong upper gradient} and $p \in (1,\infty)$.
 We recall (see~\cite[Def.~1.3.2,~p.32]{Ambrosio05}, following~\cite{DeGiorgi80,Ambrosio95}) that
 a  curve  $ u \in \AC^p_\loc([0,+\infty);\U)$ is said to be a {\it p-curve of
maximal slope} for the functional $ \phi $ with respect to the
strong upper gradient $ g$ if
 \begin{equation}
 \label{e:differential-equality}
   -(\phi\circ u)'(t) = |u'|^p (t)= g^{p'}(u(t))
   \qquad \forae t \in (0,+
   \infty),
\end{equation}
$p'$ denoting the conjugate exponent of $p$. In particular, $\phi
\circ u$ is locally absolutely continuous in $[0,+\infty)$, $g \circ
u \in L^{p'}_\loc([0,+\infty))$, and the energy identity
  \begin{equation}
  \label{abstra-enide}
  \frac1p \int_s^t |u'|^p (r)\dd r + \frac1{p'}\int_s^t g^{p'}(u(r))\dd r + \phi(u(t)) =
  \phi(u(s))  \qquad \text{for all $0 \leq s\leq t,$}
\end{equation} directly follows.

 We are now in the position of
defining the two solution notions we shall be dealing with in the
sequel. Let us start from the strongest one.
 \subsection{Energy solutions}
\label{ss:2.2}
\begin{definition}[Energy solution]
\label{def:ly-en}
 Assume
\[
\text{$\rsls$ is a strong upper gradient.}
\]
 We call
{\em energy solution
 of
the
 metric $p$-gradient flow  of the functional $\phi$} (or,
 simply, \emph{energy solution}),
a $p$-curve of maximal slope for the functional $\phi$, with respect
to the strong upper gradient  $\rsls$.\end{definition} In
particular, an energy solution $u$ is such that $ u \in
\AC^p_\loc([0,+\infty);\U)$, $\phi \circ u $ is locally absolutely
continuous on $[0,+\infty)$, the map $t\mapsto \rsls(u(t))$
 is in $L^{p'}_\loc([0,+\infty))$,  and
\begin{equation}
\label{maximal-slope-weak-tris} (\phi \circ u)'(t) + \frac1p
|u'|^p(t) +\frac1{p'}  \rsls^{p'}(u(t)) =0 \qquad \forae t \in
(0,+\infty).
\end{equation}
\begin{remark}
\upshape \label{rem:max-slope}
We point out that, being $\rsls$ a strong upper
 gradient, \eqref{maximal-slope-weak-tris} is in fact equivalent to
 the (integrated) \emph{energy inequality}
\begin{equation}
\label{e:1.1.3-bis}
 \frac1p \int_s^t |u'|^p (r)\dd r + \frac1{p'}\int_s^t \rsls^{p'}(u(r))\dd r + \phi(u(t))
 \leq
  \phi(u(s))  \qquad \text{for all $0 \leq s\leq t.$}
\end{equation}
\end{remark}
\paragraph{\bf  Geodesically convex functionals.}
A remarkable case in which  the local slope is a strong upper
gradient occurs when (cf.~\cite[Thm.~2.4.9]{Ambrosio05}) the
functional $\phi$ is $\lambda$-geodesically convex for some $\lambda
\in \R$, i.e.
\[
\text{for all $v_0,\, v_1 \in D(\phi)$ there exists a
 constant-speed  geodesic $\gamma:[0,1] \to \U$}
\]
(i.e. satisfying $d(\gamma_s,\gamma_t)= (t-s) d(\gamma_0,\gamma_1)$
for all $0 \leq s \leq t \leq 1$), such that
\begin{equation}
\label{def:l-geod-convex}
\begin{gathered}
\gamma_0=v_0, \ \ \gamma_1=v_1, \ \   \text{and $\phi$ is
$\lambda$-convex on $\gamma$, i.e.}
\\
\phi(\gamma_t) \leq (1-t) \phi(\gamma_0) + t \phi(\gamma_1)
-\frac{\lambda}{2}t (1-t) d^2(\gamma_0,\gamma_1) \ \ \text{for all
$0 \leq t \leq 1$.}
\end{gathered}
\end{equation}
\begin{remark}[$\lambda$-geodesic convexity in Banach spaces]
\upshape \label{rmk:disc2} When $\U$ is a Banach space $\Banach $
with norm $\| \cdot\|$, $\lambda$-geodesic convexity reduces to the
usual notion of $\lambda$-convexity, i.e.
\begin{equation}
\label{e:lambda-convex}
\begin{aligned}
& \exists\lambda \in \R  \quad \forall\, u_0,\,u_1 \in \Banach \quad
\forall\, \theta \in [0,1]\, : \\ & \phi((1-\theta) u_0 +\theta u_1)
\leq (1-\theta)\phi(u_0) + \theta\phi(u_1) -\frac12 \lambda \theta
(1-\theta) \| u_0-u_1\|^2\,,
\end{aligned}
\end{equation}
In particular, in the Hilbertian case $\phi$ is $\lambda$-convex if
and only if the map $v\mapsto \phi(v) - \frac{\lambda}2|v|^2$ is
convex.
\end{remark}

 As pointed out in~\cite[Rmk.~2.4.7]{Ambrosio05} (see also~\cite{Agueh2003}),
  in the case of
$p$-curves of maximal slope one should consider a generalized notion
of geodesic convexity, in which the modulus of convexity depends on
the $p$-power of the distance $d$.
 Also in view of the applications
to $p$-gradient flows in Wasserstein spaces, we thus give the
following
\begin{definition}
\label{definition:l-p-geod-convex} Given $\lambda \in \R$ and $p \in
(1,\infty)$, we say that a functional $\phi: \U \to
(-\infty,+\infty]$ is $(\lambda,p)$-geodesically convex if
\begin{equation}
\label{def:l-p-geod-convex}
\begin{gathered}
 \text{for
all $v_0,\, v_1 \in D(\phi)$ there exists a constant speed geodesic
$\gamma:[0,1] \to \U$ such that} \\
 \gamma_0=v_0, \ \ \gamma_1=v_1, \ \
\text{and $\phi$ is $\lambda$-convex on $\gamma$, i.e.}
\\
\phi(\gamma_t) \leq (1-t) \phi(\gamma_0) + t \phi(\gamma_1)
-\frac{\lambda}{p}t (1-t) d^p(\gamma_0,\gamma_1) \ \ \text{for all
$0 \leq t \leq 1$.}
\end{gathered}
\end{equation}
\end{definition}
\begin{remark}[$\lambda$- versus $(\lambda,p)$-convexity]
\label{rem:discussion} \upshape Clearly, the notion of
$(\lambda,p)$-geodesic convexity reduces to
\emph{$(\lambda,p)$-convexity} in a Banach framework, i.e.
\begin{equation}
\label{p-conv-Ban}
\begin{aligned}
& \exists\lambda \in \R  \quad \forall\, u_0,\,u_1 \in \Banach \quad
\forall\, \theta \in [0,1]\, : \\ & \phi((1-\theta) u_0 +\theta u_1)
\leq (1-\theta)\phi(u_0) + \theta\phi(u_1) -\frac1p \lambda \theta
(1-\theta) \| u_0-u_1\|^p\,,
\end{aligned}
\end{equation}
Hereafter, in a metric framework we shall always speak of
$(\lambda,p)$-geodesic convexity, and refer to
condition~\eqref{def:l-geod-convex} as $(\lambda,2)$-geodesic
convexity. Instead, in a Banach context we shall simply call
condition~\eqref{e:lambda-convex} $\lambda$-convexity, and speak of
$(\lambda,p)$-convexity only for $p \neq 2$.
\end{remark}

The following result  extends~\cite[Cor.~2.4.10, Lemma~2.4.13,
Thm.~2.4.9]{Ambrosio05}  to the case of $(\lambda,p)$-geodesically
convex functionals.

\begin{proposition}
\label{prop:l-p-geod} Let $\phi: \U \to (-\infty,+\infty]$ be
$d$-lower semicontinuous and  $(\lambda,p)$-geo\-de\-si\-cal\-ly
convex for some $\lambda \in \R$ and $p \in (1,\infty)$.
\begin{enumerate}
\item
Then, the local slope $\ls$ is $d$-lower semicontinuous and  admits
the representation
\begin{equation}
\label{repr-loc-slope} \ls(u)= \sup_{v \neq u} \left( \frac{\phi(u)
- \phi(v)}{d(u,v)} + \frac{\lambda}p d^{p-1}(u,v)\right)^+ \quad
\text{for all $u \in D(\phi)$.}
\end{equation}
Furthermore,  $\ls$ is  a strong upper gradient.
\item Suppose further
that
\[
\text{the $(\lambda,p)$-geodesic convexity
condition~\eqref{def:l-p-geod-convex} holds with $\lambda >0$.}
\]
Then, the following estimate holds
\[
\phi(u) - \inf_{\U} \phi
 \leq
\frac{1}{\lambda p'} \rsls^{p'}(u) \leq \frac{1}{\lambda p'}
\ls^{p'}(u)  \quad \text{for all $u \in D(\phi)$.}
 \]
Moreover, if $\bar{u} \in D(\phi)$ is the unique minimizer of
$\phi$, then
\begin{equation}
\label{e:key-estimate} \frac{\lambda}{p} d^p(u,\bar{u}) \leq \phi(u)
- \phi(\bar{u}) \leq \frac{1}{\lambda p'} \rsls^{p'}(u) \leq
\frac{1}{\lambda p'} \ls^{p'}(u)   \quad \text{for all $u \in
D(\phi)$.}
\end{equation}
\end{enumerate}
\end{proposition}
\noindent The \emph{proof}, which we choose not to detail,  follows
from carefully adapting the arguments for~\cite[Cor.~2.4.10,
Lemma~2.4.13, Thm.~2.4.9]{Ambrosio05}
 to the general $(\lambda,p)$-geodesically convex case.
Notice that, since the relaxed slope $\rsls$ is defined in terms of
the $\sigma$-topology, the $d$-lower semicontinuity of $\ls$ is not
sufficient to ensure that the local and strong relaxed slopes
coincide, cf. also with Remark~\ref{rem:yields}.
\subsection{Generalized solutions}
\label{ss:2.3}

 We shall  introduce   {\em generalized solutions}
by highlighting their connections to the notion of {
Minimizing Movements} due {\sc E.~De~Giorgi} \cite{DeGiorgi80,
DeGiorgi93}, cf. Definition~\ref{generalized} later on. In
particular, the discussion developed in the next lines  will show
that every  Generalized Minimizing Movement gives raise to a
generalized solution.
\subsubsection{\bf Heuristics for generalized solutions: the Minimizing
Movements approach} \label{sss:2.3.1}
 The natural way of proving the existence of
solutions to the Cauchy problem for~\eqref{maximal-slope-weak-tris}
is to approximate it  by  time-discretization.  Hence,
 let us consider a   partition of $[0,+\infty)$, which we
identify with the corresponding vector $\ttau=(\tau^1, \tau^2,
\tau^3,\dots) $ of strictly positive  time-steps.
 Note that we indicate with superscripts the elements of a
generic vector. In particular $\tau^j$ represents the $j$-th
component of the vector $\ttau$ (and not the $j$-th power of the
scalar $\tau$). Let $|\ttau|= \sup \tau^i$ be the diameter of the
partition which we ask to be finite, $t^0_\ttau=0$, and define
recursively
\begin{gather}
  t^i_\ttau:= t^{i-1}_\ttau +  \tau^i, \quad I^i_\ttau: =[t^{i-1}_\ttau, t^i_\ttau)
\quad \text{for} \ \  i\geq 1.\nonumber
\end{gather}

We shall now consider the  following  time-incremental minimization problem\linebreak
(see~\cite[Chap.~II]{Ambrosio05}).
\begin{problem}[Variational approximation scheme]
\label{pr:time-incremental} Given $\Utau^0:=u_0,$ find $\Utau^{n} $
such that
\begin{equation}
\label{eq:time-incremental} \Utau^n \in
 \argmin_{v \in U}\left\{
 \frac{d^p (v,\Utau^{n-1})}{p (\tau^i)^{p-1}} + \phi(v)
\right\} \quad \text{for all  $n \geq 1$}.
\end{equation}
\end{problem}
Under suitable lower semicontinuity and coercivity conditions on
$\phi$ (cf. with \eqref{COMP}
the framework of the $\sigma$-topology, see
assumptions~\eqref{A3}--\eqref{A5}  later on), one can verify that,
for all partitions $\ttau $ and $u_0\in D(\phi)$, Problem
\ref{pr:time-incremental} admits at least one
solution~$\{\Utau^n\}_{n \in \N}$. We then construct approximate
solutions by considering the left-continuous  piecewise constant
interpolant $\overline U_\ttau$ of the values $\{\Utau^n\}_{n\in
\N}$, i.e.
\begin{equation}
\label{pwc-interp}
 \overline U_\ttau(t):=\Utau^i \quad \text{for all} \ \  t \in I^i_\ttau \quad
  i \geq 1.
\end{equation}
 We shall also deal with the \emph{right-continuous} piecewise constant
interpolant $\underline U_\ttau$, defined by $\underline
U_\ttau(t):=\Utau^{i-1} $ for all $t \in I^i_\ttau$ and $i \geq 1$.

\begin{definition}[\cite{DeGiorgi93}]
\label{def:dg} We say that a curve $u: [0,+\infty) \to U$ is a
\emph{Generalized Minimizing Movement} for $\phi$ starting from
$u_0$ if there exists a sequence $\{\ttau_k\}$, with $|\ttau_k|\to 0$,
such that
\begin{equation}
\label{convergenze}
 \overline U_{\ttau_{k}}(t) \weak u(t) \qquad \text{as $k \to \infty$}
 \quad \text{for all} \ \  t \geq 0.
\end{equation}
We denote by $\mathrm{GMM}(\phi;u_0)$ the set of Generalized
Minimizing Movements  with initial datum  $u_0$.
\end{definition}
In order to show that $\mathrm{GMM}(\phi;u_0)$ is non-empty, one
needs to prove some a priori estimates  on  the sequence
$\{\overline U_{\ttau_{k}}\} $. The crucial step in this direction
is to observe that approximate solutions satisfy at all nodes
$t^i_\ttau$ the following \emph{discrete energy identity}
\begin{equation}
\label{A}
\begin{aligned}
 \frac1p \int_{t^{i-1}_\ttau}^{t^i_\ttau}
\left(\frac{d\left( \overline U_{\ttau}(t), \underline
U_{\ttau}(t)\right)}{\ttau^i} \right)^p \dd t &  +  \frac1{p'}
\int_{t^{i-1}_\ttau}^{t^i_\ttau} \ls^{p'}\left(\tilde U_{\ttau}(t)
\right)\dd t\\ &  + \phi\left( \overline U_{\ttau}(t)\right)=
\phi\left( \underline U_{\ttau}(t)\right)\,,
\end{aligned}
\end{equation}
where $\underline U_{\ttau}$  and $\tilde U_{\ttau}$ respectively
denote the right-continuous
 piecewise constant and the \emph{De Giorgi variational}
 interpolants of the values
$\{\Utau^n\}_{n \in \N}$, see~\cite[Def.~3.2.1]{Ambrosio05} for the
latter notion. It is immediate to realize that \eqref{A} is the
discrete counterpart to~\eqref{maximal-slope-weak-tris}. Exploiting
 the discrete energy inequality~\eqref{A}  and the coercivity of $\phi$, in~\cite[Secs.~3.3.2
and 3.3.3]{Ambrosio05} suitable a priori estimates are obtained for
the approximate sequences $\{ \overline U_{\ttau}\} $ and $\{ \tilde
U_{\ttau}\}$,  and  convergence is shown   along a suitable
subsequence $\{ \ttau_k\}$ to some limit curve $u: [0,+\infty) \to
U$. Furthermore, \eqref{A} yields that the functions
$\varphi_{\ttau_k}(t):= \phi (\overline U_{\ttau_k}(t))$ form a
non-increasing sequence. A suitable generalization of Helly's
theorem (cf.~\cite[Lemma~3.3.3]{Ambrosio05}) gives that, up to the
extraction of a further subsequence,
\[
\exists\, \varphi(t):= \lim_{k \to \infty}  \phi(\overline
U_{\ttau_{k}}(t)) \geq
 \phi(u(t)) \quad \text{for all} \ \ t \geq 0,
\]
the latter inequality by the lower semicontinuity of $\phi$.
Altogether, passing to the limit by lower semicontinuity
in~\eqref{A} it is proved in~\cite{Ambrosio05} that
$\mathrm{GMM}(\phi;u_0) \neq \emptyset$ and that for every $u \in
\mathrm{GMM}(\phi;u_0) $ there exists a non-increasing function
$\varphi: [0,+\infty) \to \R$ such that
\[
\begin{aligned}
   &
\begin{aligned}
   \frac12 \int_s^t |u'|^2(r) dr +
  \frac1{2}\int_s^t
\rls^{2} (u(r),\varphi(r))dr   +  \varphi(t)  \leq  \varphi(s) \ \
\text{for all $ 0 \leq s \leq t$,}
\end{aligned}
\\
 &\phi(u(t)) \leq \varphi(t)   \quad
 \text{for all $ t\geq 0$},
\end{aligned}
\]
where the weak relaxed slope $\rls$ naturally arises  by
 taking the $\liminf$  of the second integral term on the
left-hand side of~\eqref{A}.
\subsubsection{\bf Definition of generalized solution}
\label{sss:2.3.2} Motivated by the above discussion,  we give the
definition of generalized solution, tailored to include all limits
of the time-incremental approximations constructed
in~\eqref{eq:time-incremental}.
\begin{definition}[Generalized solution]
\label{generalized}
 A pair
 $ (u,\varphi) $, with $u
\in \AC_{\loc}^p([0,+\infty);U) $ and  $ \varphi: [0,+\infty) \to
\R,$ is a \emph{generalized solution  of the
 metric $p$-gradient flow  of  the functional~$\phi$}  (or,
 simply, a \emph{generalized solution}),
 if
\begin{align}
   &
   \label{1}
\begin{aligned}
   \frac1p \int_s^t |u'|^p(r) \dd r +
  \frac1{p'}\int_s^t
\rls^{p'} (u(r),\varphi(r))\dd r   +  \varphi(t)  \leq  \varphi(s) \
\ \text{for all $ 0 \leq s \leq t$,}
\end{aligned}
\\
 \label{2}
 &\phi(u(t)) \leq \varphi(t)   \quad
 \text{for all $ t\geq 0$}.
\end{align}
\end{definition}
\begin{remark}
\upshape \label{rem:prop-gen-sol}
 Notice that, if $ (u,\varphi) $ is a generalized solution,
then $\varphi$ is non-increasing and thus of finite pointwise
variation.
 Moreover, letting $\varphi_-(t)= \varphi(t-)$
  and $\varphi_+(t)= \varphi(t+)$ denote the pointwise left- and right-limits,
   respectively, we have that the pairs $(u,\varphi_-)$
    and $(u,\varphi_+)$ are generalized solutions as well.
\end{remark}
\paragraph{\bf From generalized to  energy solutions via conditional continuity.}
The main step
 to conclude that the
Minimizing Movements approach of Section~\ref{sss:2.3.1}  yields
energy solutions is to identify $\phi(u(t))$  as the limit of the
sequence $\{ \phi(\overline U_{\ttau_k}(t))\}$, i.e. to prove that
\begin{equation}
\label{toward-cont} \varphi(t) = \phi(u(t)) \qquad \text{for all $t
\geq 0$.}
\end{equation}
Once \eqref{toward-cont} is obtained, in view of
 equality~\eqref{disu-di-base}  one can replace  $\rls^{p'}
(u,\varphi)$, in the second integral term in~\eqref{1}, with the
strong relaxed slope $\rsls^{p'}(u)$ (see~\eqref{slope1}, \eqref{disu-di-base}). In  this
way, one obtains
  the energy inequality~\eqref{e:1.1.3-bis}, yielding
that $u$ is an energy solution, see Remark~\ref{rem:max-slope}.

Indeed, a sufficient condition for~\eqref{toward-cont} to hold
 is the following \emph{conditional continuity} requirement
\begin{equation}
\label{cond-cont}{\tag\textsc{CONT}}
 u_n \weak u, \ \ \sup_n\{\phi(u_n), \ls(u_n)\} < +\infty \ \ \Rightarrow \ \ \phi(u_n) \to
 \phi(u).
\end{equation}
\begin{remark}[Links between $(\lambda,p)$-geodesic convexity  and  conditional
continuity.] \label{rem:yields}
 \upshape
In the case $\sigma$ is  the topology induced by $d$,  let $\phi: \U
\to (-\infty,+\infty]$ be a lower semicontinuous functional
complying with the $(\lambda,p)$-geodesic convexity
condition~\eqref{def:l-p-geod-convex} for some  constant $\lambda\in
\R$.  Then, the conditional continuity property~\eqref{cond-cont}
holds. Indeed, let $\{ u_n\}$ be a sequence as in \eqref{cond-cont}.
It follows from~\eqref{repr-loc-slope} that
\begin{equation}
\label{e:conse-repr-ls}
 \phi(u_n) - \phi(u) +\frac{\lambda}{p} d^p
(u_n,u) \leq \ls(u_n) d(u_n,u) \qquad \text{for all $n \in \N$,}
\end{equation}
whence
\[
\begin{aligned}
\phi(u) \leq \liminf_{n \to \infty} \phi(u_n) &\leq \limsup_{n \to
\infty} \phi(u_n) \\ & \leq
 \phi(u) + \limsup_{n \to \infty} \left( \ls(u_n) d(u_n,u) - \frac{\lambda}{p} d^p
(u_n,u)\right) \leq \phi(u)\,,
\end{aligned}
\]
the first inequality by lower semicontinuity, the third one by
\eqref{e:conse-repr-ls}, and the last one by the properties of the
sequence $\{u_n\}$, cf. with \eqref{cond-cont}.

It is immediate to see that the previous argument also works when
the $\sigma$- and $d$-topology do not coincide, provided that $\phi$
has the following property: along sequences with bounded energy,
$\sigma$-convergence implies $d$-convergence.
\end{remark}

\subsubsection{\bf Comparison between generalized and  energy solutions}
\noindent
 Under the conditional
continuity~\eqref{cond-cont}, generalized and energy solutions may
be compared. The first result in this direction is the following
\begin{lemma}
\label{comparison} Assume  that the  {conditional continuity
property}~\eqref{cond-cont} holds.
 Then,
\begin{equation}
\label{crucial}
\begin{gathered}
 \rls(u,\varphi)<+\infty\ \Rightarrow \
\varphi=\phi(u) \ \text{and}
\\
 \rls(u,\varphi)= \rls(u,\phi(u))=\rsls(u)
 \quad \forall (u,\varphi) \in {\mathcal X}.
 \end{gathered}
\end{equation}
\end{lemma}
\begin{proof}
It follows from the definition~\eqref{slope2} of $\rls(u,\varphi)$
that for all $\eps
>0$ there exists a sequence $\{u_n\}$ with $u_n \weak u$, $\phi(u_n)
\to \varphi$, and $\ls (u_n) \leq \rls(u,\varphi)+\eps $.
Hence,~\eqref{cond-cont} yields that $\varphi=\phi(u)$,
and~\eqref{crucial} follows.
\end{proof}
\noindent  An immediate consequence of Lemma \ref{comparison} is the
following
\begin{proposition}[$\varphi=\phi \circ u$]
\label{comp-sol} Let \eqref{cond-cont}
 hold. Then, for any generalized solution $(u,\varphi)$ we have
$\varphi=\phi \circ u $ almost everywhere in $(0,+\infty)$  and the
following energy inequality holds for all $t \geq 0$, for a.a. $ s
\in (0,t)$
\begin{equation}
\label{en-ineq-with-cont}
  \frac1p \int_s^t |u'|^p(r) \dd r +
  \frac1{p'} \int_s^t
\rsls^{p'} (u(r)) \dd r +  \phi(u(t))  \leq  \phi(u(s))\,.
\end{equation}
\end{proposition}
\noindent We thus conclude
\begin{proposition}[Comparison between generalized and energy solutions]\label{comp-sol2}
Given any {energy solution} $u$, the pair $(u, \phi \circ u)$ is a
{generalized solution}.

Conversely, assume  \eqref{cond-cont}, let $\rsls$ be a {strong
upper gradient} for the functional $\phi$, and $(u,\varphi)$ be a
generalized solution. Then, $u$ is an energy solution and $ \varphi
= \phi \circ u$.
\end{proposition}

\begin{proof} The first part of the proposition is immediate.
As for the second one, use Proposition \ref{comp-sol} in order to
get that $\varphi=\phi \circ u $ almost everywhere in $(0,+\infty)$,
and replace the  energy inequality~\eqref{1}  with
$$
  \frac1p \int_s^t |u'|^p(r) \dd r +
  \frac1{p'} \int_s^t
\rsls^{p'} (u(r))\dd r +  \varphi(t)  \leq  \varphi(s) \quad
\text{for all $0 \leq s \leq t$.}
$$
Hence, as $r \mapsto \rsls (u(r))$ is in $L^{p'}_\loc([0,+\infty))$,
we have that $$\rsls(u) |u'| \in L^1_\loc([0,+\infty))$$ and $\phi
\circ u$ is locally absolutely continuous since $|(\phi \circ
u)'|\leq  \rsls(u)|u'| $ almost everywhere. In particular,
$\varphi(s) = \phi(u(s))$ for all $s\geq 0$. Finally,
inequality~\eqref{en-ineq-with-cont}  holds for all $0 \leq s \leq
t$,  and the converse inequality follows again from $|(\phi \circ
u)'|\leq \rsls(u)|u'| $.
\end{proof}

\section{Examples of energy and generalized solutions}
\label{s:3} \noindent
\subsection{Doubly nonlinear equations in Banach spaces: the convex
case}
 \label{ss:cvx-gflows} \upshape Let $(\Banach,\| \cdot\|)$ be a
(separable) reflexive Banach space
 and suppose that
\[ \phi: \Banach \to (-\infty,+\infty] \ \ \text{is a proper, l.s.c.
and convex functional.}
\]
We now show  that   the notion of energy solution (of the
 metric $p$-gradient flow  of $\phi$),
  in the ambient metric space $ \U=\Banach $ (with $d(u,v)= \| v-u\|$ and  $\sigma $ the strong
 topology),
  relates to the doubly nonlinear equation
 \begin{equation}
\label{e:dne-B} \rmJ_p (u'(t)) + \partial \phi(u(t)) \ni 0 \qquad
\text{in $\Banach'$} \quad \forae\, t \in (0,T)\,,
 \end{equation}
where $\rmJ_p : \Banach \rightrightarrows \Banach'$ is the
$p$-duality map,
 defined by
 \begin{equation}
 \label{p-dual}
\xi \in \rmJ_p (v) \ \Leftrightarrow \ \langle \xi,v \rangle =
\|v\|^p=\|\xi\|_*^{p'}= \|v\|\|\xi\|_*\,,
 \end{equation}
 and $\partial \phi$ is the
 Fr\'echet subdifferential of $\phi$, defined at a point $u \in
 D(\phi)$ by
 \begin{equation}
 \label{def:frechet}
\xi \in \partial \phi(u) \ \ \Leftrightarrow \ \ \phi(v) -\phi(u)
\geq \langle \xi, v-u\rangle + o(\|v-u\|)  \ \ \text{as }v \to u.
 \end{equation}
(which,  in the present convex case,  in fact reduces to the
standard subdifferential of $\phi$ in the sense of convex analysis).

 The link  between the formulation of Definition~\ref{def:ly-en}
 and equation~\eqref{e:dne-B} is based on the following key fact: denoting
 \begin{equation}
\label{e:argmin} \begin{cases} \displaystyle &\partial^\circ
\phi(v):= \argmin \left\{\| \xi \|_{*}\, : \ \xi \in \partial
\phi(v) \right\}
\\
& \displaystyle \|\partial^\circ \phi(v)\|_*:= \min \left\{\| \xi
\|_{*}\, : \ \xi \in \partial \phi(v) \right\}
\end{cases}
\qquad \text{for $v \in D(\phi)$,}
 \end{equation}
 it was proved in~\cite[Prop.~1.4.4]{Ambrosio05} that
 \begin{equation}
 \label{e:link-slope}
 \ls(v) = \|\partial^\circ \phi(v)\|_* \qquad \text{for all } v \in D(\phi)\,.
 \end{equation}
  It now follows from \eqref{e:link-slope} and from
   the strong-weak closedness of the graph of $\partial
 \phi$ that the map $v \mapsto \ls(v)$ is lower semicontinuous, so that $\ls \equiv \rsls$.
 Finally, the chain rule for convex functionals
\[
  \begin{gathered}
    \text{if}\quad
    u\in H^1(0,T;\Banach),\ \xi\in L^2(0,T;\Banach'),\
    \xi(t)\in \partial \phi(u(t))\ \forae\ t\in (0,T)\\
    \text{then}\quad \phi\circ u \in \AC([0,T]),\quad
    \tfrac \dd{\dd t}\phi(u(t))=\langle\xi(t),u'(t)\rangle\ \forae \ t \in (0,T).
  \end{gathered}
\]
ensures that $\ls \equiv \rsls$ is a strong upper gradient.
Exploiting~\eqref{e:argmin}, in~\cite[Prop.~1.4.1]{Ambrosio05} it
was shown that a curve $u \in \AC^p ([0,T];\Banach)$ is an energy
solution if and only if it fulfills
\[ \rmJ_p (u'(t)) +
\partial^\circ \phi(u(t)) \ni 0 \qquad \text{in $\Banach'$} \quad
\forae\, t \in (0,T)\,,
 \]
i.e. $\rmJ_p (u'(t))$ complies with the \emph{minimal section}
principle.
\subsection{Generalized solutions of the Stefan-Gibbs-Thomson  and the Mullins-Sekerka flows}
\label{ss:3.2} 
\paragraph{\bf The Stefan-Gibbs-Thomson problem.}
The Stefan problem coupled with the Gibbs-Thomson law describes the
melting and solidification of a  solid-liquid  system, and also
takes into account  surface tension effects by imposing that the
temperature is equal to the mean curvature at the phase interface.
We denote by $\teta$ the relative temperature of the system,
occupying a bounded domain $\Omega \subset \R^d$, and by $\nchi \in
\{-1,1\}$ the phase parameter, so that the phases at time $t \in
(0,T)$ are
\[
E^+(t) = \left\{x \in \Omega \, : \ \nchi(x,t) =1 \right\}\,, \qquad
E^-(t) = \left\{x \in \Omega \, : \ \nchi(x,t) =-1 \right\}\,,
\]
and the phase interface  $S(t) $ is their common essential  boundary.
 Neglecting external heat sources, the related PDE system is given by the energy balance
\begin{equation}
\label{eq:enbal}
\partial_t (\teta + \nchi) +A\teta = 0 \qquad \text{in $H^{-1}
(\Omega)$}  \quad \forae\, t \in (0,T)\,,
\end{equation}
($A$ denoting the realization of the Laplace operator with
homogeneous Dirichlet boundary conditions), and by the Gibbs-Thomson
condition at the phase interface, which formally reads
\begin{equation}
\label{e:gt-cond} \mathbf{H}(\cdot,t)= \teta(\cdot,t)
\mathbf{\nu}(t) \qquad \text{on $S(t)$,}
\end{equation}
$\mathbf{H}(\cdot,t)$ being the mean curvature vector of $S(t)$ and
$\mathbf{\nu}(t): S(t) \to \mathrm{S}^{d-1}$ the inner measure
theoretic normal to $E^+(t)$.

In the pioneering paper~\cite{Luckhaus90}, \textsc{Luckhaus} has
shown that there exist functions
\begin{equation}
\label{e:reg-teta} \teta \in L^2 (0,T;H_0^1 (\Omega)) \cap L^\infty
(0,T;L^2 (\Omega))
\end{equation}
 and
\begin{equation}
\label{e:reg-chi}
 \nchi : \Omega \times (0,T) \to \{-1,1\}, \  \text{
with}  \ \ \begin{cases} \displaystyle \nchi \in L^\infty
(0,T;\mathrm{BV}(\Omega)), \\
\displaystyle
 u:=\teta + \nchi \in H^1
(0,T;H^{-1}(\Omega)),
\end{cases}
\end{equation}
 fulfilling
equation~\eqref{eq:enbal}   for a.a. $t \in (0,T)$, and complying
with the weak formulation of the Gibbs-Thomson
law~\eqref{e:gt-cond}, i.e. for a.a. $t \in (0,T)$
\begin{equation}
   \label{eq:24}
    \begin{aligned}
     \int_\Omega\big(\dive \zzeta-\mathbf{\nu}^T(t) \,D\zzeta\,
      \mathbf{\nu}(t)\big) \,d|D\nchi(\cdot,t)|  =\int_\Omega &  \dive
      (\vartheta(\cdot,t)\zzeta)\nchi(\cdot,t) \dd x \\ &   \text{for all}\ \  \zzeta\in
      C^2(\overline\Omega;\R^d),\ \
      \zzeta\cdot \mathbf{n}=0\ \ \text{on }\partial\Omega\,,
    \end{aligned}
  \end{equation}
  $|D\nchi(\cdot,t)|$ being the total variation measure associated with the
  distributional gradient $D\nchi(\cdot,t)$.
\textsc{Luckhaus}'s existence proof  (see
also~\cite[Chap.~VIII]{Visintin96}), is based on a
 time-discretization  technique which fits  into the general
Minimizing Movements scheme introduced in Section~\ref{sss:2.3.1}. In
fact, (the initial-boundary value problem for)
system~(\ref{eq:enbal}, \ref{e:gt-cond}) has a natural gradient flow
structure with respect to the functional $\phiste:H^{-1}(\Omega) \to
(-\infty,+\infty]$ given by
\begin{equation}
\label{e:funz-sgt} \phiste(u):=\begin{cases} \displaystyle
\inf_{\nchi \in \mathrm{BV}(\Omega)} \left\{ \int_{\Omega}
\Big(\frac{1}{2}|u -\nchi |^{2} + I_{\{-1,1\}}(\nchi   ) \Big) \dd x
+ \int_{\Omega} |D\nchi| \right\} & \text{if $u \in L^2(\Omega)$,}
\\
\displaystyle + \infty & \text{otherwise,}
\end{cases}
\end{equation}
 $I_{\{-1,1\}}$ denoting the indicator function of the set $\{
 -1,1\}$.

This gradient  structure was later exploited
in~\cite{Rossi-Savare04},   where it was shown  that the Minimizing
Movements scheme in $H^{-1}(\Omega)$, driven by $\phiste$ and
starting from an initial datum $u_0\in D(\phi)=L^2 (\Omega)$
\begin{equation}
\label{incre-sgt} U_\tau^0:=u_0, \qquad  U_\tau^n \in
 \argmin_{v \in H^{-1}(\Omega))}\left\{
 \frac1{2\tau} \| v - U_\tau^{n-1} \|_{H^{-1} (\Omega)}^2  + \phiste(v)
\right\},
\end{equation}
(i.e. the variational approximation
scheme~\eqref{eq:time-incremental}, for $p=2$ and constant time-step
$\tau>0$),
 admits a solution $\{ U_\tau^n \}_{n \in \N}$ and
in fact coincides with the approximation algorithm 
constructed in~\cite{Luckhaus90}
(see~\cite[Rmk.~2.7]{Rossi-Savare04}). Furthermore, it was proved
(see~\cite[Thm.~2.5]{Rossi-Savare04} and Section~\ref{ss:8.1} later
on) that, for every  initial datum $u_0 \in L^2 (\Omega)$,
  every $u
\in \mathrm{GMM} (\phiste, u_0)$ (cf. with Definition~\ref{def:dg})
gives raise to a solution $(\teta, \nchi)$ of~\eqref{eq:enbal} and
\eqref{eq:24}, complying with~\eqref{e:reg-teta}
and~\eqref{e:reg-chi},  and fulfilling  the \emph{Lyapunov
inequality} for all $t \in [0,T]$ and for almost all $s \in (0,t)$
\begin{equation}
\label{e:lyap-ineq}
    \begin{aligned}
 \int_{\Omega} \Big(\frac{1}2|u(x,t)  & -\nchi(x,t)|^2    +
I_{\{-1,1\}}(\nchi(x,t) \Big)\dd x \\ &   + \int_{\Omega}|D\nchi(t)|
     + \int_{s}^t \int_\Omega |\nabla\vartheta(x,r)|^2  \dd x \, \dd
     r \\ &
     \begin{aligned}
\leq \int_{\Omega} \Big(\frac{1}2|u(x,s)  & -\nchi(x,s)|^2   +
I_{\{-1,1\}}(\nchi(x,s) \Big)\dd x + \int_{\Omega}|D\nchi(s)|\,.
\end{aligned}
  \end{aligned}
\end{equation}

 In
Section~\ref{ss:8.1}, we shall recover the  gradient flow approach
of~\cite{Rossi-Savare04}. First of all,   we shall
 make  precise the setting in which the (Generalized) Minimizing
Movement associated with the functional $\phiste$ yields solutions
to the weak formulation (\ref{eq:enbal}, \ref{eq:24}) of the
Stefan-Gibbs-Thomson problem, see Corollary~\ref{prop:nuova}. On the
other hand, on account of  Theorem~\ref{thm:conve} later on, the
\emph{Minimizing Movement solutions} constructed
in~\cite{Rossi-Savare04} (i.e. the solutions arising from the
time-incremental scheme~\eqref{incre-sgt}),  are in fact generalized
solutions (in the sense of Definition~\ref{generalized}), driven by
the functional $\phiste$~\eqref{e:funz-sgt}. Using this fact, in
Theorem~\ref{e:coro-rest-points} we shall obtain the existence of
the global attractor for the Minimizing Movements solutions of  the
Stefan-Gibbs-Thomson problem.
\paragraph{\bf The Mullins-Sekerka flow.}
 The Mullins-Sekerka flow is a variant of the Stefan
problem with the Gibbs-Thomson condition, modeling solid-liquid
phase transitions in thermal systems with a  negligible specific
heat. In this setting,  instead of~\eqref{eq:enbal} the internal
energy balance reads
\begin{equation}
\label{enbal-mu-sek}
\partial_t \nchi +A\teta = 0 \qquad \text{in $H^{-1}
(\Omega)$}  \quad \forae\, t \in (0,T)\,,
\end{equation}
coupled with the Gibbs-Thomson condition in the weak
form~\eqref{eq:24}.

 The global existence of solutions $(\teta,\nchi)$ to the Cauchy problem for
 the weak formulation of the
 Mullins-Sekerka system, with
 $\teta \in L^\infty (0,T;H_0^1 (\Omega)$ and $\nchi \in L^\infty (0,T;\mathrm{BV} (\Omega)) \cap
 H^{1}(0,T;H^{-1}(\Omega))$,  was first obtained in~\cite{Luckhaus95}. The proof was carried out    by passing to
 the limit in the following variational approximation scheme:
  starting from an initial datum $\nchi_0 \in \mathrm{BV}(\Omega;
 \{-1,1\})$, for a fixed time-step $\tau>0$ the
  discrete solutions $\{ \nchi_\tau^n \}_{n \in \N}$ are constructed
  via $\nchi_\tau^0:=\nchi_0$ and, for $n \geq  1$,
\[
\begin{aligned}
 \nchi_\tau^n \in
 &\argmin_{\nchi \in H^{-1}(\Omega))}\left\{
 \frac1{2\tau} \int_{\Omega}\big( A^{-1}\left(\nchi-\nchi_\tau^{n-1}\right)\left(\nchi-\nchi_\tau^{n-1}\right)
 +  I_{\{-1,1\}}(\nchi) \big) \dd x + \int_{\Omega}
|D\nchi| \right\}
 \\ & = \argmin_{\nchi \in H^{-1}(\Omega))}\left\{
 \frac1{2\tau} \| \nchi - \nchi_\tau^{n-1} \|_{H^{-1} (\Omega)}^2 +\int_{\Omega} I_{\{-1,1\}}(\nchi)  \dd x
 + \int_{\Omega}
|D\nchi|
 \right\}\,.
\end{aligned}
\]
Indeed, the above scheme (cf. with~\eqref{eq:time-incremental})
reveals that problem~(\ref{enbal-mu-sek}, \ref{eq:24}) has a
gradient flow structure w.r.t. the functional
$\phims: H^{-1}(\Omega) \to [0,+\infty)$
 \begin{equation}
\label{funz-ms} \phims(\nchi):= \begin{cases} \displaystyle
\int_{\Omega} I_{\{-1,1\}}(\nchi)  \dd x
 + \int_{\Omega}
|D\nchi| & \text{if $\nchi \in \mathrm{BV}(\Omega;\{-1,1\})$,}
\\
\displaystyle +\infty & \text{otherwise.}
\end{cases}
 \end{equation}
In fact, the existence result in~\cite{Luckhaus95} was conditioned
to the validity of the convergence of the energy $\phims$ along the
approximate solutions. Such condition excludes a loss of surface
area for the phase interfaces in the passage to the limit in the
time-discrete problem. However, this convergence requirement, akin
to our continuity assumption~\eqref{cond-cont}, is not, in general,
fulfilled, as observed in~\cite{Roeger06}. Therein, by use of
refined techniques from the theory of integral varifolds, the author
could dispense with the additional condition of~\cite{Luckhaus95},
however at the price of obtaining in the limit the weak formulation
of the Gibbs-Thomson law~\eqref{eq:24} in a generalized, varifold
form.

Since $\phims$ does neither fulfill  the chain rule nor the
conditional continuity, the Minimizing Movements solutions of the
Mullins-Sekerka problem only give raise to generalized solutions of
the gradient flow driven by $\phims$. In
Section~\ref{ss:8.2} we shall initiate some analysis in this
direction. However, proving the existence of the global attractor
for the (Minimizing Movements solutions of the) Mullins-Sekerka
problem remains an open problem.



\section{Main results}
\label{s:4} \setcounter{equation}{0}
\subsection{Statement of the main assumptions}
\label{ss:4.1}
\paragraph{\bf Topological assumptions.}
\begin{equation}
  \label{A1}\tag{A1}
 (\U,d) \quad \text{is a complete metric space.}
\end{equation}
\begin{gather}
  \label{A2}\tag{A2a}
 \sigma \quad \text{is a Hausdorff topology on $ \U $ compatible with $\ d$}
\end{gather}
The latter compatibility means that $ \sigma $ is weaker than
 the topology induced by $ d$ and $ d $ is
 sequentially $ \sigma$-lower semicontinuous,
  namely $$(u_n,v_n)\wweak (u,v) \ \ \Rightarrow \ \ \liminf_{n \to +\infty}d(u_n,v_n) \geq d(u,v).$$
 An example of a choice for $ \sigma$ complying with \eqref{A2}
  is of course that of $ \sigma $ being the topology induced by $  d$.
We shall however keep these two topologies distinct for the sake of
later applications.

We also  require that
\begin{equation}
\label{A2bis} \tag{A2b}
\begin{gathered}
\text{there exists a distance $d_{\sigma}$ on $U$ 
such that  for all $\{u_n \},$ $u \in U$}
\\
u_n \weak u \ \Rightarrow \ d_{\sigma}(u,u_n) \to 0  \ \ \text{as $n
\to \infty$} 
\end{gathered}
\end{equation}
  Namely, the topology induced by $d_{\sigma}$ is globally weaker
than $\sigma$.

\paragraph{\bf Assumptions on the functional $\boldsymbol \phi$.}
Let  $p \in (1,\infty)$. We shall ask for the following
\begin{align}
&\text{(Lower semicontinuity)}\nonumber\\
&\qquad \phi\ \  \text{is sequentially $\sigma$-lower semicontinuous}.\label{A3}\tag{A3}\\
{}\nonumber\\
&\text{(Compactness)}\nonumber\\
\label{A5}\tag{A4} &\qquad\text{the sublevels of $ \phi $ are
relatively $ \sigma$-sequentially compact}.\nonumber\\
{}\nonumber
\end{align}

\begin{remark}
\label{rem:bounded-from-below} \upshape
 It can be easily checked
that~\eqref{A3} and~\eqref{A5} entail that
\begin{equation}
\label{below}
    \phi\ \  \text{is bounded from below on $U$}.
  \end{equation}
\end{remark}

\paragraph{\bf Existence and approximation of generalized solutions.}
\noindent Under the above assumptions,  we  have  the following
crucial statement, which subsumes~\cite[Cor.~3.3.4]{Ambrosio05}
and~\cite[Thm.~2.3.1]{Ambrosio05}.
\begin{proposition}[Generalized Minimizing Movements are generalized
solutions] \label{thm:conve} Assume \eqref{A1}--\eqref{A5} and let a
family $\Lambda$ of partitions of $[0,+\infty)$ be given with
$\inf_{\ttau \in \Lambda}|\ttau| = 0$.
\\
 Then, $\mathrm{GMM}(\phi;u_0) \neq \emptyset$.
Further, every  $u \in \mathrm{GMM}(\phi;u_0) $ is  in $\AC_{\loc}^p
([0,+\infty);U)$ and there exists  a non-increasing function
$\varphi: [0,+\infty) \to \R$ such that, $\{\ttau_k\}$ being a
sequence with $|\ttau_k|\to 0$ fulfilling \eqref{convergenze}, there
holds
\[
\begin{gathered}
 \varphi(t) = \lim_{k \to \infty} \phi(\overline U_{\ttau_{k}}(t)) \geq
 \phi(u(t)) \quad \text{for all} \ \ t \geq 0, \quad \varphi(0)=
 \phi(u(0))=\phi(u_0),
\end{gathered}
\]
 and the pair $(u,\varphi)$ is a {generalized solution} in
the sense of Definition~\emph{\ref{generalized}}.
\end{proposition}
\begin{remark}
\upshape
  Note that no
  restriction on the diameter of the partitions is needed for the above convergence statement.
 Moreover, in \cite{Ambrosio05}
  a slightly more general class of functionals was
   considered
   and the compactness condition~\eqref{A5} was in fact required
   only on $d$-bounded subsets of sublevels of $\phi$. Indeed,  the
   stronger~\eqref{A5} is needed for
   the purposes of the present long-time analysis.
\end{remark}
\subsection{Statement of the main results}
\subsubsection{\bf Global attractor for generalized solutions}
 We refer
the reader to Section~\ref{s:a2} for
 the main
definitions and results of the theory of {global attractors} for
{generalized semiflows},   closely following {\sc J.M. Ball}
\cite{Ball97}.

We shall  apply the theory of generalized semiflows in the framework
of the metric phase space
\begin{equation}
\label{e:phase-space}
\begin{array}{ll}
 \displaystyle
\cx= \left\{(u,\varphi) \in D(\phi) \times \R : \ \varphi\geq
\phi(u)\right\}, \quad \text{endowed with the distance}
\\
 \displaystyle
\dcx ((u,\varphi), (u',\varphi'))= d_{\sigma}(u,u') +
|\varphi-\varphi'| \quad \forall (u,\varphi), (u',\varphi') \in \cx.
\end{array}
\end{equation}

Our candidate generalized semiflow is the set
\begin{align}
{\mathcal S}:=\Big\{ (u,\varphi):[0,+\infty]\to \U\times \Rz \ : \
(u,\varphi)\ \ \text{is a generalized solution} \Big\}.\nonumber
\end{align}
All of  the following results shall be proved in Section~\ref{s:5}.
\begin{theorem}[Generalized solutions  form   a generalized semiflow]\label{semiflow}
Assume \eqref{A1}--\eqref{A5}. Then, $ \mathcal{S}$ is a semiflow on
$ (\cx,\dcx) $ and  complies with   the \emph{continuity
property}~\emph{(C0)}.
\end{theorem}
The following Proposition sheds light on the properties of the
semiflow ${\mathcal S}$.
\begin{proposition}
\label{compactness} Assume \eqref{A1}--\eqref{A5}. Then,
\begin{enumerate}
\item  $ \mathcal{S} $ is asymptotically compact;
\item $ \mathcal{S} $ admits a Lyapunov function;
\item the set
$Z({\mathcal{S}})$ of the rest points for ${\mathcal{S}}$ is given
by
  \begin{equation}
\label{rest-point}
 Z({\mathcal{S}})= \left\{ (\bar{u},\bar{\varphi})\in \cx\
: \  \rls(\bar{u},\bar{\varphi})=0 \right\}.
  \end{equation}
  \end{enumerate}
\end{proposition}
\begin{theorem}[Global attractor for generalized solutions]
\label{teor:1} Under assumptions \eqref{A1}--\eqref{A5}, suppose
further
 that
\begin{equation}
\tag{A5} \label{A12} \text{the set $Z(\mathcal{S})$ of the rest
points of $\mathcal{S}$ is bounded in $ (\cx,\dcx)$}.
\end{equation}
Then, the semiflow $\mathcal{S}$ admits a
 global attractor $A$.
Moreover, $\omega(u,\varphi) \subset Z(\mathcal{S})$ for every
trajectory $(u, \varphi) \in \mathcal{S}$.
\end{theorem}



\subsubsection{\bf Global attractor for energy solutions} Throughout this
section, we further assume that
\begin{equation}
\label{ass-cont} \tag{A6} \text{$\phi$ complies with the conditional
continuity property~\eqref{cond-cont},}
\end{equation}
\begin{equation}
\label{sug} \tag{A7}
 \text{ $\rsls$ is a strong upper
gradient.}
\end{equation}
  Proposition~\ref{comp-sol} and~\eqref{ass-cont} yield that
\[
\begin{aligned}
\mathcal{S}=\Big\{ (u,\varphi) \ \ \text{generalized solution,
with}\ \varphi(t)=\phi(u(t)) \ \text{a.e. on $(0,+\infty)$}\Big\}.
\end{aligned}
\]
Hence, it is not difficult to check that the set of rest points of
$\mathcal S$~\eqref{rest-point} reduces to
  \begin{equation}
\label{rest-point-bis}
 Z({\mathcal{S}})= \left\{ (\bar{u},\bar{\varphi})\in \cx
: \ \ \bar{\varphi} = \phi(\bar u), \ \ \rsls(\bar{u})=0 \right\}.
  \end{equation}
Hereafter, we shall use the following notation
\begin{align}
&\label{def:eni} \Ene:=\left\{ \text{$u\in
{\AC}^p_\loc([0,+\infty);\U)$  :  $u $ \text{is an energy
solution}}\right\}.
\end{align}
 Thanks  to Proposition~\ref{comp-sol}, assumption~\eqref{sug}   gives that $
\pi_1 (\mathcal{S})=\Ene$ ($\pi_1$ denoting the projection on the
first component).

We now aim to study the long-time behavior  of energy solutions  in the phase space
\begin{equation}
\label{phase-space} (D(\phi),\du), \quad \text{with} \quad \du
(u,u'):= d_{\sigma}(u,u') + |\phi(u)-\phi(u')| \qquad \text{for all} \ \
u,u' \in D(\phi).
\end{equation}
We shall denote by $e_{\phi}$ the Hausdorff semidistance associated
with  the metric $\du$, namely, for all non-empty sets $A,\, B
\subset D(\phi) $, we have $e_\phi(A,B)= \sup_{a\in A} \inf_{b\in
B}d_\phi(a,b).$ Similarly, we denote by $e_\mathcal{X} $ the
Hausdorff semidistance associated with  the metric $d_\mathcal{X}$.
Let us introduce the {\em lifting operator}
$$
\mathcal{U}:  D(\phi) \to \mathcal{X} \qquad \mathcal{U}(u):=
(u,\phi(u)) \quad \text{for all} \ \  u \in D(\phi)
$$
and remark that $ \mathcal{U} (\pi_1 (E)) \subset E$ for all $E
\subset \mathcal{X} $ and
\begin{equation}
\label{ohserve}  E \subset \mathcal{U} (D(\phi)) \ \Rightarrow \ E=
\mathcal{U} (\pi_1 (E)).
\end{equation}
Furthermore,  the metric
  $\dcx$ restricted to
the set
  $\mathcal{U}(D(\phi))$ coincides with $\du$, namely
  \begin{equation}
  \label{dist-migliore}
\dcx (\mathcal{U}(u_1),\mathcal{U}(u_2))= \du(u_1, u_2) \qquad
\text{for all} \ \ u_1,u_2 \in D(\phi).
  \end{equation}

The following result (which is in fact a corollary of
Theorems~\ref{semiflow} and~\ref{teor:1})
 states that all information on the long-time behavior of {energy solutions}
 is encoded in the set $\pi_1 (A)$, $A$ being the global
  attractor for {generalized solutions}.
\begin{theorem}[Global attractor for energy solutions]
\label{thm:ene} Assume~\eqref{A1}--\eqref{A5} and
\eqref{A12}--\eqref{sug}. Then,
\begin{enumerate}
\item[\rm 1)]  the set $\Ene$ from~\eqref{def:eni}
is a generalized semiflow in the phase space $(D(\phi),\du)$ defined
by \eqref{phase-space}, and fulfills  the \emph{continuity
properties} \emph{(C0)}--\emph{(C3)} (cf.
Definition~\ref{def:generalized-semiflow}),
\item[\rm 2)]
the set
\begin{equation}
\label{attractor} \text{$\pi_1 (A)$ is the global attractor for
$\Ene$,}
\end{equation}
 while $\pi_1 (Z (\mathcal{S}))$ is the set of its rest points.
 \end{enumerate}
\end{theorem}

\subsubsection{\bf Convergence to equilibrium for energy solutions in the
$(\lambda,p)$-geodesically convex case} The following enhanced
result on the convergence to equilibrium for energy solutions is the
 \emph{doubly nonlinear}   counterpart
to~\cite[Thm.~2.4.14]{Ambrosio05}, which was proved for gradient
flows in the $(\lambda,2)$-geodesically convex case.
\begin{theorem}[Exponential decay to equilibrium]
\label{thm:3} Assume~\eqref{A1}--\eqref{A5}, and suppose further
that
\begin{equation}
\label{lambda-convex} \tag{A8} \text{$\phi$ fulfills the
$(\lambda,p)$-geodesic convexity
condition~\eqref{def:l-p-geod-convex} with $\lambda
>0$.}
\end{equation}
Let $\bar{u} \in D(\phi)$ be the unique minimizer of $\phi$. Then,
every energy solution fulfills for all $t_0>0$ the exponential decay
to equilibrium estimate
\begin{equation}
\label{decay-esstimate} \frac{\lambda}{p}d^p (u(t),\bar{u}) \leq
\phi(u(t)) - \phi(\bar{u}) \leq \left( \phi(u(t_0)) -
\phi(\bar{u})\right)\exp(-\lambda p'(t-t_0)) \quad \text{for all $t
\geq t_0$.}
\end{equation}
Hence, for all $u \in \Ene$
\begin{equation}
\label{e:convergence} d_\phi (u(t),\bar{u}) \to 0 \quad \text{as $t
\to +\infty$,}
\end{equation}
and
 the global attractor of  the generalized  semiflow $\Ene$ is
the singleton $Z(\Ene)=\{ \bar{u}\}$.
 \end{theorem}
 \noindent
 Notice that the $(\lambda,p)$-geodesic
 assumption~\eqref{lambda-convex} has replaced
 \eqref{A12}--\eqref{sug}.
Although the \emph{proof} is an adaptation of the
argument for~\cite[Thm.~2.4.14]{Ambrosio05}, for  the reader's
convenience we shall develop it at the end of Section~\ref{ss:5.2}.
\section{Proofs}
\label{s:5} \setcounter{equation}{0}
\subsection{Proof of Theorem~\ref{teor:1}}
\label{ss:5.1}
\paragraph{\bf Proof of Theorem~\ref{semiflow}.}
In order to check  \eqref{H1}, we fix $(u_0,\varphi_0) \in
(\cx,\dcx)$.
 It follows from Theorem~\ref{thm:conve} that there exists a {\em generalized solution} $(u,\varphi)$
fulfilling
  $u(0)=u_0$  and  ${\varphi}(0)= \phi(u_0)$. We let
$$
\tilde{\varphi}(t):= \begin{cases} \displaystyle \varphi(t) \quad
\text{for $t>0$,}
\\
\displaystyle \varphi_0 \quad \text{for $t=0$.}
\end{cases}
$$
Clearly, the pair $(u, \tilde\varphi)$ still complies
with~\eqref{1}-\eqref{2} and starts from $(u_0,\varphi_0)$ as
desired.
 It  can be easily
checked that $\mathcal{S}$ fulfills the translation and
concatenation properties. What we are left with is the proof of the
upper semicontinuity \eqref{H4}. To this end, we fix a sequence $\{
(u_0^n,\varphi_0^n ) \} \subset \cx$ with $\ds (u_0^n,u_0) +
|\varphi_0^n -\varphi_0|  \to 0$ and we consider a sequence $\{(
u_n,\varphi_n) \} \subset \mathcal{S} $  such that $ u_n(0)=u_0^n $
and $\varphi_n(0)=\varphi_0^n$. Inequality~\eqref{1} reads  for all
$n \in \N$
\begin{equation}
\label{1enne}
  \frac1p \int_s^t |u_n'|^p(r) dr +
  \frac1{p'} \int_s^t
\rls^{p'} (u_n(r),\varphi_n(r))dr +  \varphi_n(t)  \leq  \varphi_n
(s) \quad \text{for all} \ \  0 \leq s \leq t.
\end{equation}
Since
\begin{equation}
\label{varphi-phi}
 \varphi_n(t) \geq
\phi(u_n(t)) \quad \forall t \geq 0 \quad \forall n \in\N,
\end{equation}

using inequality~\eqref{1enne} for $s=0$ and~\eqref{below},
 we deduce that
 \begin{equation}
\begin{gathered}
  \int_0^t |u'_n|^p(r)dr, \ \   \int_0^t \rls^{p'} (u_n(r),\varphi_n(r)) dr,\
   \  \varphi_n(t),  \ \ \phi(u_n(t))  \\
\text{are bounded uniformly with respect to $ n $ and $ t\geq
0$}.\label{stima}
\end{gathered}
\end{equation}
In view of~\eqref{A5},  we conclude that there exists a
$\sigma$-sequentially compact set $\mathcal{K} \subset U$ such that
$u_n (t) \in \mathcal{K}$ for all $n \in \N$ and $t \geq 0$. Then,
 by exploiting a suitably refined version of Ascoli's theorem
 (see \cite[Prop. 3.3.1,~p. 69]{Ambrosio05}), we find
a subsequence    $\{ u_n \}$ (which we do not relabel),
 a curve
$\ u \in \AC^p_\loc([0,+\infty);\U)$, and a function $A \in
L^p_\loc([0,+\infty))$ such that, also recalling \eqref{A3}
and~\eqref{stima}, one has
\begin{align}
&  u_n(t) \weak u(t) \ \ \text{for all}\ \  t \geq 0,  \ \ u(0)=u_0,
\label{conv1}
    \\
&  \liminf_{n \to + \infty} \phi(u_n(t)) \geq \phi(u(t))
   \ \ \text{for all}\ \ t \geq 0, \quad  \phi(u(0))= \phi(u_0), \label{conv12} \\
&|u_n'| \rightharpoonup A \ \ \text{in $ L^p_\loc([0,+\infty))$},
\quad A \geq |u'| \ \ \text{a.e. in $(0,+\infty)$}. \label{conv13}
\end{align}
In particular, \eqref{A2bis} yields that for all $t \geq 0$
$$
\ds(u(t),u_n(t)) \to 0 \qquad \text{ as $n \to \infty$,}
$$
 On the other hand, since, for all $n \in \N$, the function $\varphi_n$ is
non-increasing, thanks to Helly's compactness theorem (see, e.g.,
\cite[Lemma~3.3.3,~p. 70]{Ambrosio05}) and to the a priori
estimate~\eqref{stima}, there  exists a (non-relabeled) subsequence
of $\{ \varphi_n\}$ and a  non-increasing map $\varphi: [0,+\infty)
\to [0,+\infty)$ such that
\[\varphi(t) = \lim_{n \to \infty}\varphi_n (t)
\geq \liminf_{n \to \infty} \phi(u_n (t)) \geq \phi(u(t)) \qquad
\text{for all} \ \  t\geq 0,
\]
where the second inequality is due to~\eqref{varphi-phi} whereas the
latter one to~\eqref{A3}. Hence, $\varphi(0) = \varphi_0$. We are
now in the position of passing to the limit as $n \to \infty$ in the
energy inequality~\eqref{1enne}.  Relying on
\eqref{conv1}-\eqref{conv12} and~\eqref{double-lsc}, and exploiting
Fatou's Lemma,  we conclude that the limit pair $(u,\varphi)$
fulfills~\eqref{1} for all $0 \leq s \leq t$, and \eqref{H4}
follows. \fin

\paragraph{\bf Proof of Proposition~\ref{compactness}.}
 We fix a sequence $\{(u_j,\varphi_j)\} \subset \mathcal{S}$ with
$\{(u_j(0), \varphi_j(0))\}$ $d_\cx$-bounded and a sequence $t_j \to
+\infty$. Inequality~\eqref{1} for $s=0$ and $t=t_j$ yields that
there exists a constant $C>0$ such that
\begin{equation}
\label{314} \phi(u_j(t_j)) \leq \varphi_j (t_j) \leq \varphi_j(0)
\leq C \quad \forall j \in \N.
\end{equation}
Thanks to~\eqref{A5}
 we conclude that there exists a
$\sigma$-sequentially compact set $\mathcal{K}' \subset U$ such that
$u_j (t_j) \in \mathcal{K}'$ for all $j \in \N$. On the other hand,
using~\eqref{314} and recalling that $\phi$ is  bounded from below
(cf. with~\eqref{below}),   we find that $\sup_j | \varphi_j (t_j)|
<+\infty$. Hence, the sequence $\{(u_j(t_j),\varphi_j (t_j))\} $
admits a $\dcx$-converging subsequence.

The projection on the second
 component $\pi_2 : \cx \to \R$ is a {\em Lyapunov function} for the
 semiflow $\mathcal{S}$. Indeed, $\pi_2$ is clearly continuous and
 decreases along the elements of $\mathcal{S}$ since for all $(u,\varphi) \in
 \mathcal{S}$ the map $\varphi$ is non-increasing. Moreover, let $(v,\psi)$
 be a complete orbit such that there exists $\bar{\psi} \in \Rz$ with
 $\pi_2 (v(t),\psi(t))= \psi(t) \equiv \bar{\psi}$ for all  $t \in
 \R$. Then, \eqref{1} yields
 $$
\frac1p \int_s^t |v'|^p(r) dr+\frac1{p'} \int_s^t
\rls(v(r),\bar{\psi}) dr \leq 0 \qquad \forall  s \leq t.
 $$
Hence, by the properties of the metric derivative we easily conclude
that there exists $\bar{v} \in D(\phi)$ such that $v(t) = \bar{v}$
for all  $t \in
 \R$, so that $(v,\psi)$ is a stationary orbit.

Finally,  the check that the set $Z(\mathcal{S})$
in~\eqref{rest-point} is the set of rest points is immediate.
 \fin
 \\
The \emph{proof} of Theorem~\ref{teor:1} follows from
Theorem~\ref{thm:ball2} and~ Proposition~\ref{compactness}.

\subsection{Proof of Theorem~\ref{thm:ene}}
\label{ss:5.2} \noindent \textbf{[Ad $1).$]}\,   Since $\mathcal{S}$
complies with~\eqref{H1}, so does $\Ene=\pi_1 (\mathcal{S})$ thanks
to Proposition~\ref{comp-sol}.
Properties~\eqref{H2} and~\eqref{H3} can be trivially checked too.
We also note that, in view of Definition~\ref{def:ly-en}, any
{energy solution} $u$ is continuous on $[0,+\infty)$ with values in
the space $(D(\phi), \du)$. In particular {(C0)}, {(C1)}, and {(C3)}
hold.

 In order to
prove~\eqref{H4}, we fix a sequence $\{ u_n \} \subset \pi_1(
\mathcal{S})$
 such that   $\du (u_n(0), u_0) \to 0$ for some $u_0 \in
D(\phi)$. This entails that  the lifted sequence $ \{ (u_n (0),
\phi(u_n(0))) \} $ fulfills $\dcx (\mathcal{U}(u_n(0)),
\mathcal{U}(u_0) ) \to 0 $ as $n \to+ \infty$. Thanks
to~Theorem~\ref{semiflow}, there exists $(u,\varphi) \in
\mathcal{S}$ with $u(0)=u_0$ and $\varphi(0)= \phi(u_0)$,
 and a subsequence $\{ n_k\}$
 such that for all $t \geq 0 $ $\dcx (\mathcal{U}(u_{n_k}(t)), (u(t),\varphi(t))) \to
 0$.
In view of Propositions~\ref{comp-sol} and\ref{comp-sol2}, the curve
$u$ is an {energy solution} and $\varphi(t) = \phi(u(t))$ for all
 $t>0$, so that
 \[
 \begin{gathered}
\du (u_{n_k}(t), u(t))= d_{\sigma} (u_{n_k}(t), u(t)) +
|\phi(u_{n_k}(t)) - \phi(u(t))| \to 0 \nonumber\\
 \text{as $k \to \infty$}
\ \text{for all} \ \  t >0,
 \end{gathered}
 \]
and the proof of (H4) is completed. The check of (C2) follows easily
from the same arguments as in the proof of Theorem \ref{semiflow}. In
particular, as long as one restricts to energy solutions, it is easy
to establish estimate \eqref{stima} and Ascoli's Theorem \cite[Prop.
3.3.1, p. 69]{Ambrosio05} in metric spaces entails the desired
convergence. 
\medskip
\\
\textbf{[Ad $2).$]}\, We shall prove that $\pi_1(A)$
 is compact in the phase space $(D(\phi), \du)$ and  that it is
 invariant and attracting for $\Ene$. To this aim,
 for every $t \geq 0$, we denote by $\mathcal{T}(t)$ ($T_1 (t)$,
resp.) the operator associated with the semiflow $\mathcal{S}$ (with
$\Ene$, resp.) by formula~\eqref{eq:operat-T}. It follows
from~Proposition~\ref{comp-sol} that, for all $t >0$,  $
\mathcal{T}(t) \mathcal{X} \subset \mathcal{U} (D(\phi))$. Hence
$\mathcal{U} (D(\phi))$ is positively invariant for the semiflow
$\mathcal{S}$, i.e.
\begin{equation}
\label{pos-inva} \mathcal{T}(t) (\mathcal{U} (D(\phi))) \subset
\mathcal{U} (D(\phi)) \qquad \text{for all} \ \  t \geq 0.
\end{equation}
As a consequence, the global attractor $A$ of $\mathcal{S}$ fulfills
$A \subset \mathcal{U} (D(\phi))$. Hence,
 by~\eqref{ohserve}  we have
\begin{equation}
\label{att-up} A=\mathcal{U} (\pi_1 (A)).
\end{equation}
Moreover, as  the projection operator is continuous from $(
\mathcal{U} (D(\phi)), \dcx)$ to $(D(\phi), \du)$, we conclude that
the set $\pi_1 (A)$ is compact as well. Again
using~Proposition~\ref{comp-sol}, it is not difficult to check that
the operators $\mathcal{T}(t)$
 and $T_1 (t)$ are related in the following way:
 \begin{equation}
\label{4.14} T_1 (t) (B) = \pi_1 (\mathcal{T}(t) (\mathcal{U} (B)))
\qquad \text{for all} \ \ B \subset D(\phi) \quad \text{for all} \ \  t \geq 0.
 \end{equation}
Hence, thanks to~\eqref{att-up} and using that $A$ is invariant for
the semiflow $\mathcal{S}$ we have
$$
T_1 (t) (\pi_1 (A)) =\pi_1 (\mathcal{T}(t) (\mathcal{U} (\pi_1
(A))))= \pi_1 (\mathcal{T}(t) (A)) = \pi_1 (A) \quad \forall t \geq
0,
$$
so that $\pi_1 (A)$ is itself invariant for $\Ene$. Finally, we fix
a bounded set $B \subset (D(\phi), \du)$.
Recalling~\eqref{dist-migliore}, we deduce that  the lifted set
\begin{equation}
\label{up-lifted-bounded} \text{ $\mathcal{U} (B)$ is bounded in
$(\cx, \dcx)$.}
\end{equation}
 Hence,
 \begin{equation}
 \label{4.16}
 \begin{aligned}
\lim_{t \to +\infty} e_{\phi}(T_1 (t) (B), \pi_1 (A)) & = \lim_{t
\to +\infty} e_{\phi}(\pi_1 (\mathcal{T}(t) (\mathcal{U} (B))),
\pi_1 (A))\\ & = \lim_{t \to +\infty} e_{\cx} (\mathcal{U} (\pi_1
(\mathcal{T}(t) (\mathcal{U} (B)))), \mathcal{U} (\pi_1 (A)))\\ &  =
\lim_{t \to +\infty} e_{\cx} (\mathcal{T}(t) (\mathcal{U} (B)), A)=0
\end{aligned}
\end{equation}
where the first identity follows from~\eqref{4.14}, the second one
from~\eqref{dist-migliore}, the third one from~\eqref{ohserve},
\eqref{pos-inva}, and the last one from the fact that $A$ attracts
the bounded sets of $(\cx, \dcx)$ and
from~\eqref{up-lifted-bounded}. By~\eqref{4.16}, $\pi_1 (A) $ has
the same attracting property in $(D(\phi), \du)$,
and~\eqref{attractor} follows. \fin
\paragraph{\bf Proof of Theorem~\ref{thm:3}.}
Notice that \eqref{A3} and \eqref{A5} guarantee that $\phi$ has at
least a minimizer $\bar{u} \in \U$, which is unique by the
$(\lambda,p)$-geodesic convexity~\eqref{lambda-convex}. Furthermore,
the set of the  rest points of the
 semiflow $\Ene$ is given by
\begin{equation}\label{rest-pts}
 Z(\Ene) = \{ \bar{u} \}.
 \end{equation} Indeed,
 there holds
 \[
0 \leq \phi(w) - \phi(\bar{u}) \leq 0 \qquad \text{for all $w \in
Z(\Ene)$,}
 \]
where the first inequality ensues from the fact that $\bar{u}$ is
the minimizer of $\phi$, while the second one  follows from
\eqref{e:key-estimate} and $\rsls(w)=0$, being $w$ a rest point.
Then, $\phi(w) = \phi(\bar{u})$, whence $w=\bar{u}$
by~\eqref{lambda-convex}.

To check \eqref{decay-esstimate} one argues along the very same
lines as in the proof of~\cite[Thm.~2.4.14]{Ambrosio05}. Namely, for
every $u \in \Ene$ and all $t>0$ one sets $\Delta(t):= \phi(u(t)) -
\phi(\bar{u})$, noticing that
\[
\Delta'(t)= - \rsls^{p'}(u(t)) \qquad \forae\, t \in (0,+\infty),
\]
cf. with~\eqref{e:differential-equality}. Combining this
with~\eqref{e:key-estimate}, one obtains the differential inequality
\[
\Delta'(t) \leq -\lambda p \Delta(t) \qquad \forae\, t \in
(0,+\infty),
\]
whence the second inequality in~\eqref{decay-esstimate}. The first
one ensues from the first of~\eqref{e:key-estimate}. Then,
\eqref{e:convergence} is a trivial consequence
of~\eqref{decay-esstimate} via~\eqref{A2}. Clearly, \eqref{rest-pts}
and \eqref{e:convergence} yield that the attractor for $\Ene$  is
given by $\{\bar{u}\}$.
 \fin

\section{Applications to doubly nonlinear  equations in Banach
spaces}  \label{s:6} \noindent  Hereafter, we shall denote by

\[
 \text{$\Banach$  a (separable) reflexive Banach
space, with norm $\| \cdot \|$.}
\]
\subsection{Doubly nonlinear
evolutions  driven  by   nonconvex energies} \label{ss:6.1}
 \noindent Let $\phi: \Banach
\to (-\infty,+\infty]$ be  a proper functional complying
with~\eqref{A3}--\eqref{A5}
  in the ambient
space
\begin{equation}
\label{ambient-banach}
 \text{$ \U=\Banach, $ with $d(u,v)= \| v-u\|$ and  $\sigma $ the strong
 topology}
\end{equation}
(see the  following   Sec.~\ref{ss:6.2} for a different choice). In
this framework, we shall extend the discussion of
Section~\ref{ss:cvx-gflows} to the doubly nonlinear differential
inclusion
\begin{equation}
\label{ex:noncvx-dne} \rmJ_p(u'(t)) + \lmsbd \phi(u(t)) \ni 0 \qquad
\text{in $\Banach'$} \quad \forae\,t \in (0,T)\,,
\end{equation}
which features the \emph{limiting subdifferential} $\lmsbd \phi$ of
$\phi$ (cf. with~\eqref{def:lmbd}), a generalized gradient notion
related to the strong-weak closure of  the  Fr\'echet
subdifferential~\eqref{def:frechet} of $\phi$.
\paragraph{\bf Literature on gradient flows with limiting subdifferentials.}
Being $\phi$ nonconvex, the Fr\'echet subdifferential  is not, in
general, strongly-weakly closed  in $\Banach \times \Banach'$
 in the sense of graphs.  It is hence
  meaningful to consider its strong-weak closure (along sequences
with bounded energy) $\lmsbd \phi$, defined at some $u \in D(\phi)$
by
\begin{equation}
\label{def:lmbd} \xi \in \lmsbd \phi(u) \ \  \Leftrightarrow \quad
\begin{cases}
\displaystyle \exists\, u_n \in \Banach,\xi_n\in\Banach' \
\text{with }
 \xi_n\in
\partial \phi(u_n) \ \text{for all $n\in \N$,}\\
\displaystyle
  u_n\to v,\
  \xi_n \weakto \xi \ \text{as $n\to+\infty$}, \   \ \sup_n\phi(u_n)
  <+\infty\,.
  \end{cases}
\end{equation}
In analogy with~\eqref{e:argmin}, for $u \in D(\phi)$ we shall use
the notation
  \begin{equation}
  \label{e:not-min-sel}
\begin{cases}
\displaystyle
  \|\lmsbd^\circ \phi (u)\|_*=\inf\left\{\| \xi \|_* \, : \ \xi \in \lmsbd^\circ \phi(u)  \right\}
  \\
  \displaystyle
\lmsbd^\circ \phi (u)= \argmin\left\{\| \xi \|_* \, : \ \xi \in
\lmsbd^\circ \phi(u)  \right\}\,.
\end{cases}
\end{equation}
Notice that, in general,  the latter set may be empty.
 The limiting  subdifferential, first introduced
in~\cite{Krueger-Mordukhovich80, Mordukhovich84},   was proposed as
a replacement of the Fr\'echet subdifferential for gradient flow
equations in Hilbert spaces, driven by \emph{nonconvex} energy
functionals, in the paper~\cite{Rossi-Savare04}. Therein, existence
and approximation results for the gradient flow equation
\begin{equation}
\label{ex:noncvx-gflow} u'(t) + \lmsbd \phi(u(t)) \ni 0 \qquad
\text{in $\Hilbert$} \quad \forae\,t \in (0,T)\,,
\end{equation}
(corresponding to $p=2$ and to $\Banach=\Hilbert$, $\Hilbert$  a
 separable Hilbert space, in~\eqref{ex:noncvx-dne}) were obtained, and applications to various PDEs were
developed, under the standing assumptions~\eqref{A3}--\eqref{A5}, the hypothesis
that $\phi$ satisfies the \emph{continuity property }
\begin{equation}
\label{eqn:cont} u_n\rightarrow u, \;\;\;\sup_n\big(\|\lmsbd^\circ
\phi (u_n)\|_*, \phi(u_n)\big)<+\infty \;\;\Rightarrow
\phi(u_n)\rightarrow\phi(u),
\end{equation}
and that   the chain rule w.r.t. $\lmsbd \phi$ holds, which we already state in the general $(p,p')$-case
\begin{equation}
\label{e:chain-rule}
 \begin{gathered}
    \text{if}\ \
    u\in {\AC}^p([0,T];\Banach),\ \xi\in L^{p'}(0,T;\Banach'),\
    \xi(t)\in \lmsbd \phi(u(t))\quad  \forae\, t\in (0,T)\\
    \text{then}\  \  \phi\circ u \in \AC([0,T]),\quad
    \tfrac d{dt}\phi(u(t))=\langle\xi(t),u'(t)\rangle\quad \forae \, t \in (0,T).
  \end{gathered}
\end{equation}
The subsequent paper~\cite{Rossi-Segatti-Stefanelli08} addressed the
long-time behavior  of a slightly less general version of
equation~\eqref{ex:noncvx-gflow} (see~\eqref{ex:noncvx-gflow-strong}
below), in which  $\lmsbd \phi$ was replaced by a strengthened
variant, the \emph{strong limiting subdifferential} $\slmsbd \phi$,
defined at some $u \in D(\phi)$ by
\[\xi \in \slmsbd \phi(u) \ \  \Leftrightarrow \quad
\begin{cases}
\displaystyle \exists\, u_n \in \Banach,\xi_n\in\Banach' \
\text{with } \xi_n\in
\partial \phi(u_n) \ \text{for all $n \in \N,$}
\\
\displaystyle
  u_n\to v,\
  \xi_n \to \xi \ \text{as $n \to \infty$}, \ \   \phi(v_n) \to
  \phi(v)\,.
  \end{cases}
  \]
In fact, it was shown in~\cite[Lemma~1]{Rossi-Segatti-Stefanelli08}
that $\lmsbd \phi$ is the (sequential) strong-weak closure of
$\slmsbd \phi$ along sequences with bounded energy, namely for all
$u \in D(\phi)$
\begin{equation}
\label{e:s-w-closure}
 \xi \in  \lmsbd\phi(u)
\ \Longleftrightarrow \ \exists \, u_k\in \Banach,  \ \ \xi_k\in
\Banach': \
\begin{cases} \displaystyle
u_k \to u, \, \xi_k \weakto \xi, \,  \sup_k \phi(u_k)
<+\infty,\\
\displaystyle
  \xi_k
 \in \slmsbd \phi(u_k) \ \text{for all $k \in \N.$}
\end{cases}
\end{equation}
 Assuming~\eqref{A3}--\eqref{A5} and
\eqref{eqn:cont}--\eqref{e:chain-rule},  the existence of the global
attractor for the semiflow associated with
\begin{equation}
\label{ex:noncvx-gflow-strong} u'(t) + \slmsbd \phi(u(t)) \ni 0
\qquad \text{in $\Hilbert$} \quad \forae\,t \in (0,T)\,,
\end{equation}
 was proved in~\cite{Rossi-Segatti-Stefanelli08} under the further
condition that the set of the rest points
\[
\begin{gathered}
 \left\{u \in D(\phi)\, : \ 0 \in \slmsbd
\phi(u)  \right\} \  \text{is bounded in the phase
space~\eqref{phase-space}.} \end{gathered}
\]
\paragraph{\bf Existence of the global attractor for  \eqref{ex:noncvx-dne} via energy solutions.}
Following the outline of Section~\ref{ss:cvx-gflows}, we shall
analyze~\eqref{ex:noncvx-dne} from a metric point of view. Namely,
in Proposition~\ref{prop:link-between} below we shall prove that the
energy solutions of the metric $p$-gradient flow of $\phi$, in the
ambient space~\eqref{ambient-banach}, yield solutions
to~\eqref{ex:noncvx-dne}. Further, we shall  show that,
 if $\phi$ is in addition $(\lambda,q)$-convex, namely if
\eqref{p-conv-Ban} holds for some $\lambda \in \R$ and  $q \in
(1,\infty)$,
 then energy solutions in fact exhaust the set of solutions
to~\eqref{ex:noncvx-dne}. Hence, we shall deduce from
Theorem~\ref{thm:ene} a result on the long-time behavior of the
solutions of~\eqref{ex:noncvx-dne}, see
Theorem~\ref{thm:appli-gflows} later on.

Our starting point is the following key  proposition.

\begin{proposition}
\label{prop:key1} In the setting of \eqref{ambient-banach}, suppose
that $\phi: \Banach \to (-\infty,+\infty]$ complies
with~\eqref{A3}--\eqref{A5}.
\begin{enumerate}
\item Then,
\begin{equation}
\label{eq:key-ineq} \| \lmsbd^\circ \phi (u) \|_* \leq \rsls (u)
\qquad \text{for all $u \in D(\rsls)$.}
\end{equation}
\item  If, in addition, $\phi$ is $(\lambda,q)$-convex  for some $\lambda \in
\R$ and $q \in (1,\infty)$,  then for all $u \in D(\phi)$
\begin{subequations}
\label{e:lambda}
\begin{equation}
\label{e:lambda-a}
\partial \phi(u) =\slmsbd \phi(u) = \lmsbd \phi(u), \qquad
\rsls (u)= \ls(u) \,,
\end{equation}
\begin{equation}
\label{e:lambda-b}
 \| \partial^\circ \phi (u) \|_*=  \| \lmsbd^\circ \phi (u) \|_*=
 \rsls (u)= \ls(u)\,.
\end{equation}
Furthermore, the local slope $\ls$ is a strong upper gradient.
\end{subequations}
\end{enumerate}
\end{proposition}
\begin{proof}
Preliminarily, we recall that, being reflexive, $\Banach$ has a
renorm which is  Fr\'echet differentiable off the origin. Therefore,
up to switching to an equivalent norm, we may suppose that the map
\begin{equation}
\label{frechet-renorm} v \in  \Banach \mapsto \frac12 \| v \|^2
\quad \text{is Fr\'echet differentiable on $\Banach$}
\end{equation}
(however, in the Appendix we are going to present  a proof which
does not use~\eqref{frechet-renorm}).
\\
 \textbf{Proof
of~\eqref{eq:key-ineq}.}  It is not restrictive to suppose  that
$\rsls (u)<+\infty$: hence, (up to further extractions) we can
select a sequence $\{ u_k \} \subset \Banach$ such that
\begin{equation}
\label{e:by-definition}
  u_k \to u, \qquad \phi(u_k) \to \phi(u), \qquad \ls(u_k) \to
  \rsls(u)\,.
\end{equation}
Now,   we invoke~\cite[Lemma~3.1.5]{Ambrosio05}, which provides a
duality formula for the local slope $\ls$: for every $k \in \N$,
there exists a sequence $\{r_j^k\}_j$, with
 $r_j^k \down 0$ as $j \up \infty$,  and a selection
\begin{equation}
\label{e:minimizzazione} z_j^k \in \argmin_{v \in \Banach} \left\{
\frac{\|v -u_k\|^2}{2r_j^k} + \phi(v) \right\},
\end{equation}
such that
\begin{equation}
\label{convs} \ls^2(u_k) =\lim_{j \up \infty} \frac{\|z_j^k
-u_k\|^2}{(r_j^k)^2}, \quad \text{and} \quad z_j^k \to u_k \
\text{as $j \up \infty$.}
\end{equation}  Notice that
\eqref{e:minimizzazione} yields $ \phi(z_j^k) \leq \phi(u_k)$ for
all $j \in \N$, whereas, by the convergence of $\{ z_j^k \}_j$
in~\eqref{convs} and  the lower semicontinuity of $\phi$, we gather
$\liminf_{j} \phi(z_j^k) \geq \phi(u_k)$, so that
\begin{equation}
\label{e:conv-energies} \phi(z_j^k)\to \phi(u_k) \qquad \text{as $j
\to \infty$.}
 \end{equation} Furthermore, \eqref{frechet-renorm} ensures that  there holds
the following sum rule for the Fr\'echet subdifferential
\[
\frsbd \left( \frac{\| \cdot -u_k\|^2}{2r_j^k} + \phi\right) =
\rmJ_2 \left(\frac{\cdot-u_k}{r_j^k} \right) + \frsbd \phi\,.
\]
Hence, we may conclude that  for all $j \in \N$ $z_j^k$ fulfills the
Euler equation
\[ \rmJ_2 \left(\frac{z_j^k-u_k}{r_j^k}  \right) + \frsbd\phi(z_j^k)
\ni 0\,,
\]
namely
\begin{equation}
\label{e:euler-relaxed}  \exists\, w_j^k \in \rmJ_2
\left(\frac{z_j^k-u_k}{r_j^k} \right) \cap \left(-
\frsbd\phi(z_j^k)\right)
 \,.
\end{equation}
Thus, in view of \eqref{convs} and the definition of $\rmJ_2$,
\begin{equation}
\label{e:crucial-info} \left(\ls(u_k)\right)^2 = \lim_{j \up
\infty}\|w_j^k \|_*^2\,.
\end{equation}
By a diagonalization procedure,  collecting  \eqref{convs},
\eqref{e:conv-energies}, and~\eqref{e:crucial-info},  we may extract
subsequences $\{w_{j_k}^k\}$, $\{z_{j_k}^k\}$
 ($\{w_{k}\}$, $\{z_{k}\}$ for short) such that for all $k \in \N$
 \begin{equation}
 \label{e:diagonalization}
\Big(|\|w_k \|_* -\ls(u_k)|+\| z_k -u_k \| + |\phi(z_k)
-\phi(u_k)|\Big) \leq \frac1k.
\end{equation}
 Thus, in view of
\eqref{e:by-definition} we find that $z_k \to u$, $\phi(z_k) \to
\phi(u)$ and that there exists $w \in \Banach^*$ such that, up the
extraction of a  further non-relabeled
 subsequence, $w_k \weakto w$ in $\Banach^*$. We readily
conclude from~\eqref{e:euler-relaxed} and from the definition of
$\lmsbd\phi$ that
\begin{equation}
\label{limiting-in-the-end} -w \in \lmsbd \phi(u).
\end{equation}
On the other hand,  with \eqref{e:diagonalization}
and~\eqref{e:by-definition} we find
\begin{equation}
\label{e:ineq-lim} \|w\|_{\Banach^*} \leq \liminf_{k \to \infty}
\|w_k \|_*= \liminf_{k \to \infty} \ls(u_k)=  \rsls(u)
\end{equation}
Combining~\eqref{limiting-in-the-end} with~\eqref{e:ineq-lim} we
arrive at~\eqref{eq:key-ineq}.
\\
\textbf{Proof of~\eqref{e:lambda}.} Relations~\eqref{e:lambda} have
been proved in~\cite[Prop.~5.6]{Rossi-Mielke-Savare08} for
$\lambda$-convex functionals. Mimicking the proof
of~\cite[Prop.~5.6]{Rossi-Mielke-Savare08} and    taking into
account Proposition~\ref{prop:l-p-geod}, it is easy to
extend~\eqref{e:lambda} to the $(\lambda,q)$-convex case. In the
same way, the last statement follows from the very same arguments as
in  the proof of~\cite[Prop.~5.11]{Rossi-Mielke-Savare08}.
\end{proof}
\begin{proposition}
\label{prop:link-between} In the setting of \eqref{ambient-banach},
suppose that $\phi: \Banach \to (-\infty,+\infty]$ complies
with~\eqref{A3}--\eqref{A5}.
\begin{enumerate}
\item If  $\phi$ also  fulfills the chain rule \eqref{e:chain-rule}
with respect to $\lmsbd \phi$ and
\begin{equation}
\label{min-non-empty} \lmsbd^\circ \phi(u) \neq \emptyset \qquad
\text{for all $u \in D(\lmsbd \phi)$,}
\end{equation}
  then
\begin{equation}
\label{e:sug} \text{$\rsls$ is a strong upper gradient.}
\end{equation}
Further, any energy solution $u \in \AC^p ([0,T]; \Banach)$ of the
metric $p$-gradient flow of $\phi$ is also a solution of the doubly
nonlinear equation~\eqref{ex:noncvx-dne} and fulfills the minimal
section principle
\begin{equation}
\label{e:min-sec-princ} -\lmsbd^\circ \phi(u(t)) \subset \rmJ_p
(u'(t)) \qquad \forae\, t \in (0,T)\,.
\end{equation}
\item
 In addition,   if $\phi$ fulfills~\eqref{p-conv-Ban}  for some $\lambda \in \R$ and $q \in (1,\infty)$,
   a curve $u \in
\AC^p ([0,T]; \Banach)$ is an energy solution \emph{if and only if}
the map $t \mapsto \phi(u(t))$ is absolutely continuous on $(0,T)$
and
 $u$
fulfills~\eqref{e:min-sec-princ}.
\end{enumerate}
\end{proposition}
\begin{proof}
Property~\eqref{e:sug} is a straightforward consequence of the chain
rule~\eqref{e:chain-rule} and of inequality~\eqref{eq:key-ineq}. The
proof of the subsequent statement relies on the same argument
as~\cite[Prop.~1.4.1]{Ambrosio05}: we shall however sketch it for
the reader's convenience. It follows
from~\eqref{maximal-slope-weak-tris} and~\eqref{eq:key-ineq} that
any energy solution $u \in \AC^p ([0,T]; \Banach)$ fulfills (note
that $u$ is almost everywhere differentiable, hence
$|u'|(\cdot)\equiv \| u'(\cdot)\|$  a.e. in $(0,T)$)
\begin{equation}
\label{e:m-cr}
\begin{cases} \displaystyle
t \mapsto \|\lmsbd^\circ \phi(u(t)) \|_* \in L^{p'} (0,T),
\\
\displaystyle
 (\phi \circ
u)'(t) \leq  -\frac1p \|u'(t)\| - \frac1{p'} \|\lmsbd^\circ
\phi(u(t)) \|_*^{p'} \qquad \forae\, t \in (0,T)\
\end{cases}
\end{equation}
On the other hand, arguing in the same way as in the proof
of~\cite[Lemma~3.4]{Rossi-Savare04} and also exploiting
\eqref{min-non-empty} and  the first of~\eqref{e:m-cr}, we find that
there exists a selection $\xi_{\textrm{min}} \in L^{p'}
(0,T;\Banach^*)$ in the multi-valued map  $t \mapsto \lmsbd^\circ
\phi(u(t))$. The chain rule~\eqref{e:chain-rule} yields $(\phi \circ
u)' = \langle \xi_{\textrm{min}}, u' \rangle $ a.e. in $(0,T)$.
Combining this with the inequality in~\eqref{e:m-cr}, we
conclude~\eqref{e:min-sec-princ}.

  The converse implication may be proved in the
  $(\lambda,q)$-convex
  case by a completely analogous argument, relying on
  identity~\eqref{e:lambda-b}.
\end{proof}
\begin{remark}
\label{rmk:metric-sols} \upshape On behalf of the above result, we
are entitled to refer to the energy solutions of the metric
$p$-gradient flow of $\phi$ as the \emph{metric solutions}
of~\eqref{ex:noncvx-dne}.  It follows from the proof of Proposition~\ref{prop:link-between} that  a curve $u \in \AC^p ([0,T];
\Banach)$ is a   \emph{metric solution} of~\eqref{ex:noncvx-dne} if
and only if
\[ -(\phi \circ u)'(t)  =  \| u'(t) \| \rsls(u(t)) = \| u'(t) \| \|
\lmsbd^\circ \phi (u(t)) \|_* \ \ \forae\, t \in (0,T)\,.
\]
\end{remark}
\noindent The following result is a direct consequence of
Theorem~\ref{thm:ene} and the previous
Proposition~\ref{prop:link-between}.
\begin{theorem}
\label{thm:appli-gflows} In the setting of \eqref{ambient-banach},
suppose that $\phi: \Banach \to (-\infty,+\infty]$ complies
with~\eqref{A3}--\eqref{A5}, with the conditional
continuity~\eqref{cond-cont}, and with~\eqref{A12}. Then,
\begin{enumerate}
\item if $\phi$ also complies with~\eqref{min-non-empty} and
\begin{equation}
\label{e:bound-rest}
\begin{gathered}
\text{the set of the rest points} \  \  \left\{u \in D(\phi)\, : \ 0
\in \lmsbd \phi(u) \right\} \\ \text{is bounded in the phase
space~\eqref{phase-space},}
\end{gathered}
\end{equation}
 then
 the semiflow associated with the
\emph{metric solutions} of equation~\eqref{ex:noncvx-dne} (cf. with
Remark~\ref{rmk:metric-sols}) admits a global attractor;
\item if $\phi$ is
$(\lambda,q)$-convex  for some $\lambda \in \R$ and $q \in
(1,\infty)$   and complies with~\eqref{e:bound-rest},
  the semiflow generated by the \emph{whole} set of solutions
to~\eqref{ex:noncvx-dne} admits a global attractor.
\end{enumerate}
\end{theorem}
\subsection{Outlook to quasivariational doubly nonlinear equations}
\label{ss:6.2} \noindent Finally, we show how our results can be
applied to the study of the long-time behavior of a class of doubly
nonlinear equations of  the form
\begin{equation}
\label{e:quasivar}
\partial \Psi(u(t),u'(t)) + \lmsbd \phi(u(t)) \ni 0  \qquad
\text{in $\Banach'$} \quad \forae\,t \in (0,T)\,,
\end{equation}
where the dissipation functional $\Psi: \Banach \times \Banach
\to [0,+\infty)$ also depends on the state variable $u$ and which
are often referred to as  \emph{quasivariational},
 cf.~\cite{Rossi-Mielke-Savare08} and the references
therein.

In particular, following~\cite{Rossi-Mielke-Savare08}, where the
metric approach was applied to prove  the existence of solutions to
the Cauchy problem for~\eqref{e:quasivar}, we shall focus on the
case the functional $\Psi$ is given by
\begin{equation}
\label{e:psi} \Psi(u,v) = \frac{\eta_u (v)^p}{p} \qquad \text{for
all $u,\, v \in \Banach$, with $1<p<\infty$ and}
\end{equation}
\begin{equation}
\label{def-eta}
\begin{gathered}
 \{ \eta_u\}_{u \in \Banach} \quad \text{a family of
norms on $\Banach$, such that}
\\
\exists\, \mathsf{K}>0 \ \ \forall\, u,\,v \in \Banach\, : \ \
\mathsf{K}^{-1} \| v \| \leq \eta_u(v) \leq  \mathsf{K}^{-1} \| v \|
\end{gathered}
\end{equation}
and the dependence $u \mapsto \eta_u$
 is continuous  in the sense of \textsc{Mosco}-convergence (see, e.g., \cite[Sec.~3.3,~p.~295]{attouch}), namely
 \begin{subequations}
\begin{align}
   \label{eq:Mosco1}
  u_n\to u,\quad v_n\weakto v\quad\text{in }\Banach
  \quad&\Rightarrow\quad
  \liminf_{n\to\infty}\eta_{u_n}(v_n)\ge \eta_u(v),
\\
  \label{eq:Mosco2}
  u_n\to u,\quad v\in \Banach\quad&\Rightarrow\quad
  \exists\, v_n\to v:\quad
  \lim_{n\to\infty}\eta_{u_n}(v_n)=\eta_u(v).
\end{align}
\end{subequations}
For all $u \in \Banach$, we shall denote by $\eta_{u*}$ the related
dual norm on $\Banach'$.

 Within this framework, the metric approach
to~\eqref{e:quasivar} may be developed by endowing the ambient space
\begin{equation}
\label{e:ambient9}
\begin{gathered}
 \U=\Banach \ \  \text{with the Finsler distance
induced by $\{ \eta_u \}_{u \in \Banach}$, i.e.}
\\
d_\eta(v,w):= \inf \left\{ \int_0^1
    \eta_{u(t)}({u'(t)}) \dd t \, : \ u \in \mathrm{AC}([0,1]; \Banach),  \,
    \, u(0)=v, \, u(1)=w \right\}
\end{gathered}
\end{equation}
for all $ v,\,w \in \Banach.$ Then, we have the analogue of
Proposition~\ref{prop:link-between}
(see~\cite[Prop.~8.2]{Rossi-Mielke-Savare08} for the proof).
\begin{proposition}
\label{p:quasiv} In the setting
of~\eqref{def-eta}--\eqref{e:ambient9}, suppose that $\phi: \Banach
\to (-\infty,+\infty]$ complies with~\eqref{A3}--\eqref{A5}, that
 $\phi$  fulfills the chain rule \eqref{e:chain-rule}
with respect to $\lmsbd \phi$, and that \eqref{min-non-empty} holds.

Then,  any energy solution $u \in \AC^p ([0,T]; \Banach)$ of the
metric $p$-gradient flow of $\phi$ is also a solution of the doubly
nonlinear equation~\eqref{e:quasivar} with $\Psi$ given
by~\eqref{e:psi},  and fulfills the minimal section principle
\begin{equation}
\label{e:min-sec-princ-quasi}   \partial \Psi({u(t)},{u'(t)})
\supset
  \argmin\Big\{\eta_{u(t)*}({-\xi}): \xi\in \lmsbd(u(t))\Big\}
 \qquad \forae\, t \in (0,T)\,.
\end{equation}
Conversely,   if $\phi$ fulfills~\eqref{p-conv-Ban} for some
$\lambda \in \R$ and $q \in (1,\infty)$,   a curve $u \in \AC^p
([0,T]; \Banach)$ is an energy solution
 \emph{if and only if}   the map $t \mapsto \phi(u(t))$ is
absolutely continuous on $(0,T)$, and
 $u$
fulfills~\eqref{e:min-sec-princ-quasi}.
\end{proposition}
\noindent Hence,   we derive from  our general Theorem~\ref{thm:ene}
the quasivariational counterpart to Theorem~\ref{thm:appli-gflows}.
To avoid overburdening this paper, we prefer to omit the statement
that, under the assumptions of Proposition~\ref{p:quasiv} and the
boundedness of the set of the rest points~\eqref{e:bound-rest}, the
generalized semiflow associated with the metric solutions
to~\eqref{e:quasivar} admits a global attractor.
\begin{example}
\upshape \label{ex:gen-allen-cahn}
 Our results  apply to the following generalized
Allen-Cahn equation
\begin{equation}
\label{appl}
 \rho(u) \, |u_t|^{p-2}u_t
 -\mathrm{div}(\beta(\nabla u))
+\mathrm{W}'(u) = 0 \quad \text{in $\Omega \times (0,T),$}
\end{equation}
 with $\Omega \subset \R^d$ a bounded domain with sufficiently
 smooth boundary, $\rho: \R \to (0,+\infty)$ a continuous function,
bounded from below and from above by positive constants, $\beta:
\R^d \to \R^d$ the gradient of some smooth function $j$ on $\R^d$,
and $\mathrm{W}: \R \to \R$ a differentiable function. For
simplicity, we supplement  \eqref{appl} with homogeneous Dirichlet
boundary conditions for $u$.

The boundary value problem for~\eqref{appl}  can be recast in the
general form~\eqref{e:quasivar} in the Banach space
\[
\Banach= L^p (\Omega),
\]
 with the choices
\begin{equation}
\label{eta-ex} \eta_{u}(v):= \left(\int_\Omega \rho(u(x)) |v(x)|^p
\dd x\right)^{1/p} \quad \text{for all $u,\,v \in L^p(\Omega)$,}
\end{equation} and
\begin{equation}
\label{phi-ex} \phi(u) := \begin{cases} \displaystyle \int_{\Omega}
\left( j(\nabla u(x)) + \mathrm{W}(u(x)) \right) \dd x  & \text{if
$u \in H_0^1 (\Omega)$, $j(\nabla u) + \mathrm{W}(u) \in L^1
(\Omega)$,}
\\
\displaystyle +\infty  & \text{otherwise.}
\end{cases}
\end{equation}
We refer to~\cite[Sec.~8.2]{Rossi-Mielke-Savare08} for the proof
of the fact that  the family of norms $\{ \eta_u\}$~\eqref{eta-ex}
fulfills~\eqref{def-eta}--\eqref{eq:Mosco2}, and  for the precise
statement of the assumptions on the nonlinearities $j$ and
$\mathrm{W}$, under which the functional $\phi$~\eqref{phi-ex}
complies with the assumptions of Proposition~\ref{p:quasiv}.
 Let us just mention that,  in particular, the classical double-well
potential $\mathrm{W}(r) = \frac{(r^2-1)^2}4$  fits into the frame
of such assumptions.
\end{example}
\section{Application to Curves of Maximal Slope in Wasserstein spaces}
\label{s:7} In this section, we aim to apply  our general
Theorem~\ref{thm:ene} to the  long-time analysis of the class of
diffusion equations in $\R^d$, $d \geq 1$, mentioned in the
Introduction (cf. with~\eqref{e:drift-intro-p-local-nuovo}). We
shall systematically follow
   the \emph{metric
approach} to evolutions in Wasserstein spaces developed
in~\cite{Ambrosio05}. In order to make the paper as self-contained
as possible, we recall some preliminary definitions and results on
gradient flows in Wasserstein spaces in Sec.~\ref{ss:7.1} below,
referring to~\cite{Ambrosio05}  for more details and all the proofs.
\subsection{Setup in Wasserstein spaces}
\label{ss:7.1}
Hereafter,
we shall denote by $\mathscr{L}^d$ the $d$-dimensional Lebesgue measure and by
$\pi^i: \R^d \times \R^d \to \R^d$, $i=1,2$, the projection on the
$i$-component.

In the following, we shall work in the space of probability
measures in $\R^d$ with finite $p$-moment, i.e. we shall take $U$ to
be
\begin{equation}
\label{def-u} U=\mathscr{P}_{p}(\R^d)= \left\{  \mu \in
\mathscr{P}(\R^d)\, : \ \int_{\R^d} |x|^p \dd\mu(x) <+\infty
\right\}\,,
\end{equation}
 endowed with the $p$-Wasserstein distance
\begin{equation}
\label{def-W_p} W_p (\mu_1,\mu_2) = \min \left\{\left( \int_{\R^d
\times \R^d} |x-y|^p \dd\gamma(x,y)\right)^{1/p}\, : \ \gamma \in
\Gamma(\mu_1,\mu_2)\right\}
\end{equation}
for all $\mu_1,\, \mu_2 \in \mathscr{P}_{p}(\R^d)$, where $
\Gamma(\mu_1,\mu_2)$ is the set of probability measures $\gamma \in
\mathscr{P}(\R^d \times \R^d)$ fulfilling the push-forward relations
$\pi^1_{\#} \gamma= \mu_1$ and $\pi^2_{\#} \gamma= \mu_2$ (namely,
for all Borel subset $B \subset \R^d $ and $i=1,2$ $\mu_i (B)=
\gamma ((\pi^{i})^{-1}(B) )$ ). We shall denote by $
\Gamma_{\textrm{o}}(\mu_1,\mu_2)$ the class of the measures $\gamma
\in \Gamma(\mu_1,\mu_2)$ attaining the minimum in the definition of
$W_p$.
 It follows from~\cite[Prop.~7.1.5]{Ambrosio05} that
the metric space $(\mathscr{P}_{p}(\R^d), W_p)$ is complete, so
that~\eqref{A1} is fulfilled. In this setting,
\begin{equation}
\label{def-sigma} \text{$\sigma$ is the
 topology of narrow convergence.}
 \end{equation}
 We recall that a sequence $(\mu_n)
\subset \mathscr{P}(\R^d)$ narrowly converges to some $ \mu \in
\mathscr{P}(\R^d)$ if for all $f \in \rmC_{\textrm{b}}^0 (\R^d)$
(the space of continuous and bounded functions on $\R^d$) there
holds
\[
\lim_{n \to \infty} \int_{\R^d} f(x) \dd \mu_n (x) = \int_{\R^d}
f(x) \dd \mu(x)\,.
\]
The narrow topology complies with~\eqref{A2} thanks
to~\cite[Lemma~7.1.4]{Ambrosio05}, while property~\eqref{A2bis}
ensues from~\cite[Rem.~5.1.1]{Ambrosio05}.

A remarkable property of absolutely continuous curves with values in
$\mathscr{P}_{p}(\R^d)$ is that they can be characterized as
solutions of the \emph{continuity equation}. More precisely
(see~\cite[Thm.~8.3.1]{Ambrosio05}), with any $\mu \in
\mathrm{AC}^p([0,T]; \mathscr{P}_{p}(\R^d))$  we can associate
  a Borel vector field $v :(x,t) \in  \R^d \times (0,T)
\mapsto v_t(x) \in \R^d$
 such that $\forae\, t \in (0,T) $
 \begin{equation}
 \label{field1}
v_t \in L^p (\mu_t; \R^d), \ \  \text{and}\  \  |\mu'|(t)= \|v_t
\|_{ L^p (\mu_t; \R^d)} = \left(\int_{\R^d} |v_t(x)|^p
\dd\mu_t(x)\right)^{1/p} \,,
\end{equation}
and  $\mu_t,\, v_t$
 fulfill the continuity equation
\begin{equation}
 \label{field2}
\partial_t \mu_t + \dive (v_t \mu_t) = 0 \qquad \text{in $\R^d \times (0,T)$}
\end{equation}
in the sense of distributions. In fact, the velocity field $v_t$
turns out to be unique in some suitable sense,
see~\cite[Chap.~8]{Ambrosio05}.

Finally, we recall that it is possible to introduce a
subdifferential notion intrinsic to the Wasserstein framework. In
fact, for the sake of simplicity   we shall not
recall~\cite[Def.~10.3.1]{Ambrosio05} in its general form, but in a
particular case.
\begin{definition}
\label{def:ext-frechet-subdifferential} Let
$\phi:\mathscr{P}_{p}(\R^d) \to (-\infty,+\infty] $ be a proper and
lower semicontinuous functional and let $\mu \in D(\phi)$. Given a
Borel vector field $\xi : \R^d \to \R^d$, we say that $\xi$
 belongs to the \emph{(extended) Fr\'echet subdifferential} of
 $\phi$ at $\mu$, and write $\xi \in \boldpartial \phi(\mu)$, if
 \[
\begin{gathered}
\xi \in L^{p'}(\mu; \R^d), \qquad \text{and, as
$\nu \to \mu$ in $\mathscr{P}_{p}(\R^d) $,}
\\
\phi(\nu) -\phi(\mu) \geq \inf_{\gamma \in \Gamma_{\mathrm{o}}
(\mu,\nu)} \int_{\R^d \times \R^d} \langle \xi(x), y-x
\rangle\dd\gamma(x,y) + \mathrm{o}\left(W_p (\mu,\nu) \right)\,.
\end{gathered}
 \]
 \end{definition}
 \noindent
 In agreement with   notation~\eqref{e:not-min-sel}, we define
 \[
 \begin{cases}
 \displaystyle
\| \boldpartial^\circ \phi(\mu) \|_{L^{p'}(\mu; \R^d)}:=\inf\{\| \xi
\|_{L^{p'}(\mu; \R^d)}\, :   \ \xi \in  \boldpartial \phi(\mu) \},
\\
\displaystyle \boldpartial^\circ \phi(\mu):= \argmin\{\| \xi
\|_{L^{p'}(\mu; \R^d)}\, :   \ \xi \in  \boldpartial \phi(\mu) \}\,.
\end{cases}
 \]

\paragraph{\bf Gradient flows in Wasserstein spaces.}
Hereafter, we shall focus on proper and lower semicontinuous
functionals $\phi: \mathscr{P}_{p}(\R^d) \to (-\infty,+\infty]$,
fulfilling some coercivity condition which we do not specify, for it
is implied by our general lower semicontinuity/compactness
conditions~\eqref{A3}--\eqref{A5}, and such that $\phi$ is
\emph{regular}, i.e. for all $(\mu_n) \subset \mathscr{P}_p (\R^d)$,
with $\varphi_n=\phi(\mu_n)$ and $\xi_n \in
\boldpartial\phi(\mu_n)$, there holds
\begin{equation}
\label{def:regular}
\begin{cases}
\displaystyle \mu_n \to \mu \ \text{in $\mathscr{P}_{p}(\R^d)$,}
\\
\displaystyle \varphi_n \to \varphi,
\\
\displaystyle \sup_n \| \xi_n \|_{L^{p'}(\mu_n;\R^d)} <+\infty,
\\
\displaystyle (i \times \xi_n)_{\#} \mu_n \to (i \times \xi)_{\#}
\mu   \ \ \text{narrowly in $\mathscr{P}(\R^d \times \R^d)$,}
\end{cases}
\ \ \Longrightarrow \ \ \begin{cases} \displaystyle  \xi \in
\boldpartial \phi (\mu),
\\
\displaystyle \varphi=\phi(\mu)\,,
\end{cases}
\end{equation}
($i: \R^d \to \R^d$ denoting the identity map). Notice
that~\eqref{def:regular} is  in the same spirit as the conditional
continuity~\eqref{cond-cont}.

 Using this (extended) Fr\'echet  subdifferential notion, in~\cite[Chap.~11]{Ambrosio05}
a formulation of gradient flows intrinsic to the Wasserstein
framework has been proposed: according
to~\cite[Def.~11.1.1]{Ambrosio05}, we say that a curve $\mu \in
\mathrm{AC}^p([0,T];\mathscr{P}_{p}(\R^d))$ is a solution of the
$p$-gradient flow equation driven by $\phi$ if its velocity field
$v_t$~\eqref{field1}--\eqref{field2}
 satisfies the inclusion
\begin{equation}
\label{eq:wass-gflow} \mathrm{J}_p (v_t) \in -\boldpartial
\phi(\mu_t) \qquad \forae\, t \in (0,T)\,,
\end{equation}
 (where
$\mathrm{J}_p : L^p (\mu_t; \R^d) \to L^{p'} (\mu_t; \R^d)$ denotes
the duality map)  and complies  with the \emph{minimal section
principle}
\begin{equation}
\label{e:min-sel-prin} \mathrm{J}_p (v_t) = -\boldpartial
\phi^\circ(\mu_t) \qquad \forae\, t \in (0,T)\,.
\end{equation}

Under the assumptions that $\phi$ is
 lower semicontinuous, coercive, and regular  in the sense of
 \eqref{def:regular}, \cite[Thm.~11.1.3]{Ambrosio05} provides the
 crucial link between curves of maximal slope and gradient flows in
 Wasserstein spaces: it states that
\[
\begin{gathered}
\text{
 $\mu \in \mathrm{AC}^p([0,T];\mathscr{P}_{p}(\R^d))$ is a $p$-curve of maximal slope
 w.r.t. the
local slope $\ls$}
\\
\text{if and only if $\mu_t$ is a solution of the gradient flow
equation~\eqref{e:min-sel-prin}}
\\
\text{and the map $t \mapsto \phi(\mu_t)$ is a.e. equal to a
function of bounded variation.}
\end{gathered}
 \]
Relying on this result and on our Theorem~\ref{thm:ene}, we shall
prove the existence of the global attractor for the solutions of the
gradient flow equation~\eqref{eq:wass-gflow}, driven by a functional
$\phi$ which is the sum of the internal, potential, and interaction
energy,
cf. with~\eqref{e:funz-sum} below. 
\subsection{The sum of the internal, potential, and interaction
energy} \label{ss:7.2} \noindent As
in~\cite[Sec.~10.4.7]{Ambrosio05} (cf.
also~\cite{Carrillo-McCann-Villani03, Carrilloal}), we consider the
following functions
\begin{equation}
\label{v1} \tag{\textrm{V}1} \begin{gathered} V: \R^d \to
(-\infty,+\infty] \ \ \text{proper, lower semicontinuous, such that}
\\
\text{its proper domain $D(V)$ has a non-empty interior $O \subset
\R^d$,}
\end{gathered}
\end{equation}
\begin{equation}
\label{f1} \tag{\textrm{F}1} \begin{gathered} F: [0,+\infty) \to \R
\ \ \text{convex, differentiable, with  $F(0)=0$,}
\\
\text{
 and satisfying
the \emph{doubling condition}}
\\
\exists\,C_F >0  \ \  \forall\, z,w \in [0,+\infty) \, : \ \ F(z+w)
\leq C_F \left( 1+ F(z) + F(w) \right)\,,
\end{gathered}
\end{equation}
\begin{equation}
\label{w1} \tag{\textrm{W}1} \begin{gathered} W: \R^d \to
[0,+\infty) \ \ \text{convex, G\^{a}teaux differentiable, even,}
\\
\text{
 and
satisfying the \emph{doubling condition}}
\\
\exists\,C_W >0  \ \  \forall\, x, y \in \R^d \, : \ \ W(x+y) \leq
C_W \left( 1+ W(x) + W(y) \right)\,,
\end{gathered}
\end{equation}
which induce the following functionals on $\mathscr{P}_{p}(\R^d)$:
\begin{align}
& \label{funcV} \mathcal{V}(\mu):= \int_{\R^d} V(x) \dd \mu(x)
\qquad \text{(potential energy),}
\\
& \label{funcF} \mathcal{F} (\mu):= \begin{cases} \displaystyle
\int_{\R^d} F(\rho(x)) \dd x  & \text{if $\mu= \rho \mathscr{L}^d$,}
\\
\displaystyle +\infty & \text{otherwise,}
\end{cases}
\qquad \text{(internal energy),}
\\
& \label{funcW} \mathcal{W}(\mu):= \frac{1}2  \int_{\R^d \times
\R^d} W(x,y) \dd (\mu \otimes\mu)(x,y) \qquad \text{(interaction energy).}
\end{align}
Hence, for given constants $c_1 \in (0,+\infty)$ and $ c_2, \, c_3
\in [0,+\infty)$ we define $\phi: \mathscr{P}_{p}(\R^d) \to
(-\infty,+\infty]$ by
\begin{equation}
\label{e:funz-sum} \phi(\mu) := c_1 \mathcal{V}(\mu) + c_2
\mathcal{F}(\mu) + c_3 \mathcal{W}(\mu)\,.
\end{equation}

We further assume that the function $V$ is $p$-coercive, i.e.
\begin{equation}
\label{v3} \tag{\textrm{V}2} \limsup_{|x| \to +\infty}
\frac{V(x)}{|x|^p} = +\infty\,,
\end{equation}
and that (cf. with condition~\cite[(2)]{Agueh2003})
\begin{equation}
\label{v2} \tag{\textrm{V}3} \text{$V$ is $(\lambda,p)$-convex on
$\R^d$ for some $\lambda \in \R$.}
\end{equation}
Notice that \eqref{v2} yields that $V$ is G\^{a}teaux-differentiable
in $O$. As for $F$, we require that
\begin{equation}
\label{f2} \tag{\textrm{F}2}
\begin{gathered}
\text{$F$ has a superlinear growth at infinity, i.e. $\limsup_{s \up
+\infty} \frac{F(s)}s = +\infty$, and}
\\
 \liminf_{s \downarrow 0} \frac{F(s)}{s^\alpha}
>-\infty \ \text{for some $\alpha>\frac{d}{d+p}$}\,,
\end{gathered}
\end{equation}
and finally that
\begin{equation}
\label{f3} \tag{\textrm{F}3} \text{the map $s \mapsto
s^{d}F(s^{-d})$ is convex and non-increasing in $(0,+\infty)$.}
\end{equation}
We associate with $F$ the function $L_F :[0,+\infty) \to
[0,+\infty)$ given by
\[
 L_F(r) = rF'(r) - F(r) \qquad \text{for all $r \geq
0$.}
\]
\begin{remark}
\label{rem:comment-ass} \upshape
 We point out that~\eqref{f2} yields
\begin{equation}
\label{partenegativa} \exists \, C_1, C_2 \geq 0 \ \forall\, r \geq
0\,: \quad  F(r) \geq -  C_1   - C_2 r^{\alpha}\,,
\end{equation}
($F^-$ being  the negative part of $F$), and
 refer to~\cite[Chap.~9]{Ambrosio05} for further
details on the above conditions.
 We shall just mention
(see~\cite[Rem.~9.3.10]{Ambrosio05}) that examples of functionals
complying with~(\ref{f1}, \ref{f2}, \ref{f3}) are
\begin{equation}
\label{e:7.12}
\begin{gathered}
\text{the entropy functional} \ \ F(s)=s \log(s),
\\
\text{the power functional} \  \ F(s) =\frac1{m-1}s^m\,, \ \
\text{for $m >1$.}
\end{gathered}
\end{equation}
(in fact, the case $1-1/d \leq m <1$ can also be treated, see
\cite[Rem.~9.3.10]{Ambrosio05}).
\end{remark}
\begin{remark}[$(\lambda,p)$-geodesic convexity of $\phi$]
\label{rem:impo} \upshape
 It
follows from~\cite[Props.~9.3.5,~9.3.9]{Ambrosio05} that, thanks
to~\eqref{w1} and~\eqref{f3}, the functionals $\mathcal{W}$ and
$\mathcal{F}$ are convex along geodesics in $\mathscr{P}_p (\R^d)$.
On the other hand, arguing as in the proof
of~\cite[Prop.~9.3.2]{Ambrosio05}, one verifies that the potential
energy $\mathcal{V}$ is $(\lambda,p)$-geodesically convex in
$\mathscr{P}_p (\R^d)$. We thus conclude that
\[
\text{$\phi$  from ~\eqref{e:funz-sum} is $(\lambda,p)$-geodesically
convex in $\mathscr{P}_p (\R^d)$.}
\]
\end{remark}

The following proposition collects the main properties of $\phi$.
 We shall just sketch its proof
for the reader's convenience.
\begin{proposition}
\label{prop:wass} In the setting
of~\eqref{def-u}--\eqref{def-sigma}, assume~\eqref{v1}--\eqref{f3}.
Then, the functional  $\phi : \mathscr{P}_{p}(\R^d) \to
(-\infty,+\infty]$  defined by~\eqref{e:funz-sum}
fulfills~\eqref{A3}--\eqref{A5} and \eqref{ass-cont}--\eqref{sug};
$\phi$ is regular  in the sense of~\eqref{def:regular},
 there holds
\begin{equation}
\label{equal-slopes} \ls(\mu) = \rsls(\mu)  \qquad \text{for all
$\mu \in D(\phi)$,}
\end{equation}
and the (extended) Fr\'echet subdifferential $\boldpartial \phi$ of
$ \phi(\mu)$ admits the following characterization at every $\mu \in
D(\phi)$:
 \begin{equation}
 \label{e:charact}
\begin{gathered}
 \xi \in \boldpartial \phi(\mu) \quad \text{if and only if} \quad \xi \in L^{p'}(\mu; \R^d) \ \
   \text{and for all $\nu \in \mathscr{P}_p(\R^d)$}
\\
\phi(\nu) -\phi(\mu) \geq \inf_{\gamma \in \Gamma_{\mathrm{o}}
(\mu,\nu)} \int_{\R^d \times \R^d} \langle \xi(x), y-x
\rangle\dd\gamma(x,y) + \frac{\lambda}{p}W_p^p (\mu,\nu)\,.
\end{gathered}
\end{equation}

 Furthermore, if $\mu
\in \mathrm{AC}_{\loc}^p([0,+\infty);\mathscr{P}_{p}(\R^d))$ is an
energy solution (in the sense of Definition~\emph{\ref{def:ly-en}}),
then there exists   $\rho : t \in [0,+\infty) \mapsto \rho_t \in
L^1(\R^d)$
 such that
\[
\begin{gathered}
 \int_{\R^d} \rho_t(x) \dd x  =1,
\ \  \int_{\R^d} |x|^p \rho_t(x) \dd x<+\infty, \ \ \mu_t=\rho_t
\mathscr{L}^d \ \ \text{for all $t \in [0,+\infty)$},
\\
L_F(\rho) \in L_{\loc}^1 (0,+\infty; W_{\loc}^{1,1}(O)),
\\
\text{the map} \ t \mapsto \left\|c_1\nabla V + c_2\frac{\nabla L_F
(\rho_t)}{\rho_t} + c_3 \nabla W \ast \rho_t\right\|_{L^{p'}
(\mu_t;\R^d)} \in L_{\loc}^p (0,+\infty)\,,
\end{gathered}
\]
(where $\ast$ denotes the convolution product), satisfying the
drift-diffusion equation with nonlocal term
\begin{equation}
\label{gflow-sum-funz}
\partial_t \rho_t - \dive \left(\rho_t j_{p'} \left( c_1\nabla V +
c_2\frac{\nabla L_F(\rho_t)}{\rho_t} + c_3 \nabla W \ast \rho_t
\right) \right)=0
\end{equation}
in $\mathscr{D}'(\R^d \times (0,+\infty))$ ($j_{p'}(r):= |r|^{p' -2} r$),  and the energy identity for all $0 \leq s
\leq t $
\begin{equation}
\label{e:enid-phi} \int_s^t \left\|c_1\nabla V + c_2 \frac{\nabla
L_F (\rho_r)}{\rho_r} + c_3\nabla W \ast \rho_r\right\|_{L^{p'}
(\mu_r;\R^d)}^p \dd r + \phi(\mu_t) = \phi(\mu_s)\,.
\end{equation}
\end{proposition}
\begin{proof}
It follows from~\cite[Lemma~5.1.7, Example~9.3.1,
Rmk.~9.3.8]{Ambrosio05}
 that the functional $\phi$ is sequentially  lower semicontinuous w.r.t. the
 narrow convergence.
 In order to prove~\eqref{A5} in the case $c_2>0$ (the proof for $c_2=0$ being analogous), let us consider a sequence $\{\mu_n=\rho_n
\mathscr{L}_d \} \subset \mathscr{P}_p(\R^d)$
 fulfilling $\sup_n \phi(\mu_n)<+\infty$.
 Using that $W$
takes nonnegative values and the coercivity
  \eqref{v3} of $V$, one finds that
\begin{equation}
\label{e2} \sup_n \left(\int_{\R^d} c_2 F(\rho_n (x)) \, {\rm d}x +
M_1 \int_{\R^d} |x|^{p} \rho_n (x) \, {\rm d}x\right) <+\infty\,,
\end{equation}
for some $M_1>0$.  On the other hand, being $(\alpha
p)/(1-\alpha)>d$ by~\eqref{f2}, one has
$$
K:= \int_{\R^d} (1+ |x|)^{-\frac{\alpha p}{1-\alpha}} \, {\rm d}x
<+\infty,
$$
and H\"{o}lder's inequality yields
$$
\begin{aligned}
 \int_{\R^d}\rho_n^{\alpha} (x) \, {\rm d}x&=
\int_{\R^d}\rho_n^{\alpha} (x)(1+|x|)^{\alpha p}(1+|x|)^{-\alpha p}
\, {\rm d}x
\\
& \leq \left( \int_{\R^d}\rho_n (x) (1+|x|)^{p}\, {\rm
d}x\right)^{\alpha}K^{1-\alpha}.
\end{aligned}
$$

Therefore, from \eqref{partenegativa} we infer that
\begin{equation}
\label{e3}
\begin{aligned}
\int_{\R^d}F(\rho_n (x)) \, {\rm d}x & \geq
 -C_1 -C_2 \int_{\R^d}\rho_n^{\alpha} (x) \, {\rm d}x
\\ & \geq 
-C_1 -C_2 K^{1-\alpha}
2^{\alpha(p-1)} \left( 1+ \int_{\R^d}\rho_n (x) |x|^{p}\, {\rm
d}x\right)^{\alpha}
\\
& \geq
-C_1 -C_{M_1}\left( 1+
\int_{\R^d}\rho_n (x) |x|^{p}\, {\rm d}x\right)-C,
\end{aligned}
\end{equation}
the last passage following
 from a trivial application of the Young
inequality with a suitable constant $C_{M_1}>0$ depending on $M_1$
and  to be specified later. Combining \eqref{e2} and \eqref{e3}, and
choosing $C_{M_1}$ in such a way that $c_2 C_{M_1} \leq M_1/2$,
 we conclude that

\begin{equation}
\label{e4} \sup_n \left(
\frac{M_1}2 \int_{\R^d} |x|^{p} \rho_n (x) \, {\rm d}x\right)
<+\infty.
\end{equation}

Thus, the integral condition for tightness is satisfied and, by the
Prokhorov theorem,  $\{\mu_n\}$ admits a narrowly converging
subsequence. Furthermore, it follows from~\eqref{e4} that for all
$\eps >0$
 $\{\mu_n\}$ has uniformly integrable $(p-\eps)$-moments   so that,
by~\cite[Prop.~7.1.5]{Ambrosio05}, it is
 $W_{p-\eps}$-relatively compact.

It follows from Remark~\ref{rem:impo} and from
Proposition~\ref{prop:l-p-geod} that the local slope $\ls$ is a
strong  upper gradient. Furthermore, arguing as in the proof
of~\cite[Lemma~10.3.8]{Ambrosio05} one easily checks that, by
$(\lambda,p)$-geodesic convexity,  the functional $\phi$ is also
regular, and \eqref{e:charact}  can be shown by repeating the
calculations of~\cite[Thm.~10.3.6]{Ambrosio05}.
 Property~\eqref{equal-slopes} is
proved in~\cite[Prop.~10.4.14]{Ambrosio05}. In view
of~\cite[Thms.~10.3.11,~10.4.13]{Ambrosio05},  we have that for a
given measure $\mu \in \mathscr{P}_p(\R^d)$
\begin{equation}
\label{e:minimal-selection}
\begin{gathered}
 \text{$\ls(\mu)=\rsls(\mu) <+\infty$ if and
 only if $L_F(\rho) \in W_{\loc}^{1,1}(O)$, there exists}
\\
\text{a unique $w \in L^{p'} (\mu;\R^d) $ s.t. $\| w\|_{L^{p'}
(\mu;\R^d)}= \|\boldpartial^\circ \phi(\mu) \|_{L^{p'} (\mu;\R^d)}=
 \ls(\mu)$},
\end{gathered}
\end{equation}
and there holds
\begin{equation}
\label{eq:subdiff}
\begin{gathered}
 \rho w=c_1 \rho \nabla V +
 c_2 \nabla L_F(\rho) + c_3 \rho (\nabla W) \ast
 \rho \qquad \mu-\text{a.e. in } \R^d.
\end{gathered}
\end{equation}
Thus, the
conditional continuity~\eqref{cond-cont} follows from
combining~\eqref{e:minimal-selection} with the regularity
property~\eqref{def:regular}.

Finally, let $\mu \in \mathrm{AC}^p([0,T];\mathscr{P}_{p}(\R^d))$ be
a $p$-curve of maximal slope: collecting~\eqref{e:min-sel-prin},
 \eqref{e:minimal-selection},
and~\eqref{eq:subdiff}, we conclude that its velocity field
satisfies for a.a. $t \in (0,T)$
\[
-j_p(v_t) =c_1\nabla V + c_2\frac{\nabla L_F(\rho_t) }{\rho_t} +
c_3\nabla W \ast \rho_t \qquad \mu_t-\text{a.e. in } \R^d\,.
\]
 Joint with the continuity equation~\eqref{field2}, the latter
relation yields~\eqref{gflow-sum-funz}. Then, the energy
identity~\eqref{e:enid-phi} ensues from~\eqref{field1}
and~\eqref{abstra-enide}.
\end{proof}
\begin{remark}[An alternative convexity assumption]
 \label{rem:alternative}
\upshape Instead of the $(\lambda,p)$-convexity condition~\eqref{v2}
for $V$, one could impose
 the standard
$\lambda$-convexity on $V$,  with $\lambda$ having a different sign
depending on $p$, namely
\begin{equation}
\label{v2-primo} \tag{V3'} \text{$V$ $\lambda$-convex on $\R^d$,
with
 $\lambda \geq 0$ if $p \leq 2 $ or
$\lambda \leq 0$ if $p \geq 2 $.}
\end{equation}
It is proved in~\cite[Prop.~9.3.2]{Ambrosio05} that, under
\eqref{v2-primo}, the functional $\mathcal{V}$ is
$(\lambda,2)$-geodesically convex  on $\mathscr{P}(\R^d)$, and so is
$\phi$ given by~\eqref{e:funz-sum}. It is not difficult to verify
that the proof of Proposition~\ref{prop:wass} goes through upon
replacing~\eqref{v2} with~\eqref{v2-primo}.
\end{remark}
\subsection{Global attractor for drift-diffusion equations with nonlocal terms}
 \label{ss:7.3}
\begin{remark}[Reduction to the $(\lambda,p)$-geodesically convex case, with $\lambda \leq 0$]
\upshape In view of Proposition~\ref{prop:wass}
 and  Theorem~\ref{thm:3}, if $\phi$  from~\eqref{e:funz-sum}   is
$(\lambda,p)$-geodesically convex with $\lambda>0$
 (which is implied by $V$ being $(\lambda,p)$-convex with
 $\lambda>0$), then the global attractor for the semiflow of the
 energy solutions of the gradient flow associated with $\phi$
reduces to the  unique minimum point of $\phi$, to which all
trajectories converge as $t \to +\infty$ with an exponential rate.

Therefore, henceforth we shall focus on  the case in which
\[
\text{$V$ is $(\lambda,p)$-convex on $\R^d$, for some $\lambda \leq
0$.}
\]
\end{remark}
The following result provides the last ingredient for the proof of
the existence of the global attractor for the drift-diffusion
equation~\eqref{gflow-sum-funz}.
\begin{lemma}
\label{lemma:rest-points}
 In the setting of~\eqref{def-u}--\eqref{def-sigma},
assume~\eqref{v1}--\eqref{f3}, and  that \eqref{v2} holds with
$\lambda \leq 0$. Then, the set of the rest points
\[
\begin{aligned}
 \Sigma(\phi) & =
 \left\{\bar{\mu} \in \mathscr{P}_p
(\R^d)\, : \ 0 \in \boldpartial \phi(\bar{\mu}) \right\}
\\ & =
 \left\{\bar{\mu} \in \mathscr{P}_p
(\R^d)\, : \ \begin{array}{l} \displaystyle
\bar{\mu}=\bar{\rho}\mathscr{L}_d, \
L_F(\bar{\rho}) \in W_{\loc}^{1,1}(O), \\
 \displaystyle c_1\nabla V + c_2\frac{\nabla L_F(\bar{\rho})}{\bar{\rho}} + c_3 \nabla W
\ast
 \bar{\rho}=0 \ \  \bar{\mu}-\text{a.e. in $\R^d$}
 \end{array}
 \right\}
 \end{aligned}
\]
is bounded in the phase space $(D(\phi), d_{\phi})$, $d_{\phi}$
being as in~\eqref{phase-space}.
\end{lemma}
\begin{proof}
We suppose that $c_2 >0$, referring to the following
Example~\ref{ex:7.3.1} for  some ideas on   the case $c_2=0$. Due
to~\eqref{e:charact}, for every $\bar{\mu}= \bar{\rho} \mathscr{L}_d
\in \Sigma(\phi)$ there holds in particular
\begin{equation}
\label{e:competitor}
\begin{gathered}
\phi(\tilde{\mu}) - \phi(\bar{\mu}) \geq \frac{\lambda}{p} W_p^p
(\tilde{\mu},\bar{\mu})
\\
\text{for all $\tilde{\mu}=\tilde{\rho}  \mathscr{L}_d \in
\mathscr{P}_p(\R^d)$, with $\tilde{\rho} \in \mathrm{C}^\infty
(\R^d)$ compactly supported.}
\end{gathered}
\end{equation}
Let $\mathrm{t}: \R^d \to \R^d$ be the unique
(by~\cite[Thm.~6.2.4]{Ambrosio05}) optimal transport map between
$\bar{\mu}$ and $\tilde{\mu}$, so that
\[
\begin{aligned}
W_p^p (\tilde{\mu},\bar{\mu}) = \int_{\R^d} |\mathrm{t}(x) - x|^p
\dd \bar{\mu}(x) \leq 2^{p-1} \int_{\R^d} |\mathrm{t}(x)|^p \dd
\bar{\mu}(x) + 2^{p-1} \int_{\R^d} | x|^p \dd \bar{\mu}(x)\,.
\end{aligned}
\]
Hence, it follows from \eqref{e:competitor} (recalling that $\lambda
\leq 0$ and that $\tilde{\mu}$ is compactly supported) that
\begin{equation}
\label{e:boundedness-rest-points}
 \phi(\bar{\mu}) + \frac{\lambda 2^{p-1}}{p}\int_{\R^d} | x|^p \dd
 \bar{\mu}(x) \leq \phi(\tilde{\mu}) -\frac{\lambda 2^{p-1}}{p}  \int_{\R^d} |\mathrm{t}(x)|^p \dd
\bar{\mu}(x) \leq \phi(\tilde{\mu})+ C\,,
\end{equation}
Exploiting now \eqref{v3} and \eqref{f2} and repeating the
calculations  developed throughout~\eqref{e2}--\eqref{e4} in the
proof of Proposition~\ref{prop:wass}, it is not difficult to infer
from \eqref{e:boundedness-rest-points} that
\begin{equation}
\label{e:bdd}
\begin{aligned}
\exists\, \overline{C}>0 \ \ \forall\, \bar{\mu} \in \Sigma(\phi) \,
: \ \ \begin{cases} \displaystyle & \phi(\bar{\mu}) \leq
\overline{C}\,,
\\
\displaystyle & W_p^p  (\delta_0,\bar{\mu})\leq \int_{\R^d} |x|^p
\dd \bar{\mu}(x) \leq \overline{C}\,,
\end{cases}
\end{aligned}
\end{equation}
$\delta_0$ denoting the Dirac mass centered at $0$, which concludes
the proof.
\end{proof}
\noindent On behalf of   Theorem~\ref{thm:ene},
Proposition~\ref{prop:wass}, and Lemma~\ref{lemma:rest-points}, we
thus conclude
\begin{theorem}
\label{th:attrat-wass} In the setting
of~\eqref{def-u}--\eqref{def-sigma}, assume~\eqref{v1}--\eqref{f3},
and  that \eqref{v2} holds with $\lambda \leq 0$.

 Then, the
generalized semiflow generated by the energy solutions of the
gradient flow driven by $\phi$  from~\eqref{e:funz-sum}
  admits a global attractor.
\end{theorem}
\begin{example}
\label{ex:7.3.1} \upshape
 In the case $c_2=c_3=0$,
the set of the rest points $\Sigma(\phi)$ consists of the  measures
$\bar{\mu} \in \mathscr{P}_p (\R^d)$ satisfying
\begin{equation}
\label{res-v} \bar{\mu}(\R^d \setminus S)=0, \quad \text{with}
\qquad S = \left\{x \in \R^d\, : \ \ \nabla V(x) =0 \right\}\,.
\end{equation}
The boundedness of $\Sigma (\phi)$ follows from~\eqref{v1},
\eqref{v3}, and   \eqref{v2} (with $\lambda \leq 0)$.
 Indeed, the latter condition and \eqref{res-v} yield that for
every fixed $\bar{y} \in O$ there holds
\[
V(\bar{y}) \geq V(x) + \frac{\lambda}{p} |\bar{y} - x|^p \qquad
\text{for\,}\bar{\mu}-\text{a.a.}\, x \in \R^d\,,
\]
whence, also using~\eqref{v3}, we find  that   for all $M_2>0$ there
exists $C_{M_2}>0$ with
\[
\begin{aligned}
V(\bar{y}) - \frac{\lambda 2^{p-1}}{p} |\bar{y}|^p \geq V(x) +
\frac{\lambda 2^{p-1}}{p} |x|^p \geq \frac{V(x)}2 + \frac{1}{2}
\left(M_2 + \frac{\lambda 2^p}{p} \right)|x|^p -\frac{C_{M_2}}{2}
\end{aligned}
\]
for $\bar{\mu}$-a.a. $x \in \R^d$. Choosing $M_2 > -\lambda 2^p/p  $
and taking into account that $V$ is bounded from below, one
immediately deduces~\eqref{e:bdd}.
 An analogous
argument can be developed in the case $V$ is $\lambda$-convex in the
standard sense.
\end{example}
\begin{example}
\label{ex:7.2} \upshape
 In the case $c_2 \neq 0$ and  $c_3=0$,  $V$
fulfills~\eqref{v1}, \eqref{v3},
 and~\eqref{v2} (with $\lambda \leq 0)$, and $F$ is either the entropy or the power
functional (cf. with~\eqref{e:7.12}), then $\Sigma(\phi)$ reduces to
a singleton. In fact,  when $F(s)=s \log(s)$ the stationary equation
defining the  set of the rest points of $\phi$  becomes
\[
c_1\nabla V + c_2\frac{\nabla \bar{\rho}}{\bar{\rho}} =0 \qquad
\text{$\bar{\mu}=\bar{\rho} \mathscr{L}^d $-a.e. in $\R^d$,}
\]
leading to
\[
\bar{\mu}= \frac{1}Z e^{-V} \mathscr{L}^d\,,
\]
(where $Z= \int_{\R^d} e^{-V(x)} \dd x$).
 Similar calculations may be developed in the case of the
power functional.
\end{example}
%

\section{Application to phase transition evolutions driven by mean curvature}
\label{s:8} Throughout this section, we shall assume  that
\begin{equation}
\label{e:ambient8} \text{$\Omega$ is a $\rmC^1$, connected,  and
open set,}
\end{equation}
 and  take as ambient space
\begin{equation}
\label{e:ambient7} \U=H^{-1} (\Omega), \ \ \text{$\sigma$ being the
norm topology}
\end{equation}
\subsection{The Stefan-Gibbs-Thomson problem}
\label{ss:8.1} It is straightforward to check that, in the setting
of~\eqref{e:ambient7},    the functional $\phiste:H^{-1}(\Omega) \to
[0,+\infty]$
\[
\phiste(u):=
\begin{cases}
\displaystyle \inf_{\nchi \in \mathrm{BV}(\Omega)} \left\{
\int_{\Omega} \Big(\frac{1}{2}|u -\nchi |^{2} + I_{\{-1,1\}}(\nchi
) \Big) \, dx + \int_{\Omega} |D\nchi| \right\}  & \text{if $u \in
L^2 (\Omega)$,}
\\
\displaystyle +\infty & \text{otherwise}
\end{cases}
\]
  complies
with~\eqref{A3}--\eqref{A5}.
Given $u \in D(\phiste)=L^2 (\Omega)$, we   denote
 by $\mathcal{M}(u)$ the non-empty set of the $\nchi$'s in $\mathrm{BV}(\Omega; \{
 -1,1\})$ attaining the minimum in the  definition of $\phiste$, and by
$\partial \phiste $ its Fr\'echet subdifferential with respect to
the topology of $H^{-1}(\Omega)$   (which we identify with its
dual).   Therefore, the limiting subdifferential
  $\lmsbd \phiste $ is given at every  $u \in L^2 (\Omega)$ by
\[
\lmsbd \phiste(u)= \left\{ \xi \in H^{-1} (\Omega)\, : \
\begin{array}{l} \displaystyle
\exists\, u_n \in H^{-1}(\Omega), \   \xi_n\in  H^{-1}(\Omega)  \
\text{with }
\\
 \displaystyle
 \xi_n\in
\partial \phiste(u_n) \ \text{for all $n\in \N$,}\\
 \displaystyle
  u_n\to u \ \text{in $H^{-1}(\Omega)$,} \
  \xi_n \weakto \xi
\ \text{in $H^{-1}(\Omega)$,}
  \\
   \displaystyle
   \sup_n\phiste(u_n)
  <+\infty
  \end{array}
  \right\}\,.
\]
 It was proved   (cf. with~\cite[Thm.~2.5]{Rossi-Savare04}) that,
for every  initial datum $u_0 \in L^2 (\Omega)$,
  every $u
\in \mathrm{GMM} (\phiste, u_0)$  (which is a non-empty set thanks
to Theorem~\ref{thm:conve}) fulfills the Cauchy problem for the
gradient flow equation
\[ u'(t) + \lmsbd \phiste (u(t)) \ni 0 \qquad
\text{in $H^{-1}(\Omega)$} \quad \forae\, t \in (0,T)\,.
\]

The following result sheds some light on the properties of $\lmsbd \phiste$. 
\begin{proposition}
\label{p:crucial-ste}
 The functional $\phiste$ fulfills
the continuity property~\eqref{cond-cont} and
\begin{equation}
\label{e:crucial-luck}
\begin{gathered}
 \text{for all $u \in L^2 (\Omega)$} \ \text{there
exists a unique  $\bar{\nchi}_u \in \mathrm{BV} (\Omega)$ such that}
 \\   \bar{\nchi}_u \in \mathcal{M}(u) \ \text{and} \
 u-\bar{\nchi}_u \in H_0^1 (\Omega) \,.
 \end{gathered}
\end{equation}
 Further, let $u
\in L^2 (\Omega)$ fulfill $ |\partial^+ \phiste|(u)<+\infty $. Then,
$\lmsbd \phiste(u)$ is non-empty  and, more precisely,
\begin{equation}
\label{rapr-subdif} \lmsbd \phiste(u) = \{ A (u -\bar{\nchi}_u) \},
\end{equation}
($A$ being  the Laplace operator with homogeneous Dirichlet boundary
conditions) with $\teta:= u -\bar{\nchi}_u$ fulfilling the weak form
of the Gibbs-Thomson law, i.e. for all $\zzeta\in
C^2(\overline\Omega;\R^d)$ with $\zzeta\cdot
       \mathbf{n}=0$ on $\partial\Omega$ there holds
   \begin{equation}
     \label{eq:24bis}
     \int_\Omega\big(\dive \zzeta-\mathbf{\nu}^T \,D\zzeta\,
       \mathbf{\nu}\big) \,d|D\bar{\nchi}_u| =\int_\Omega \dive
       (\vartheta\zzeta)\bar{\nchi}_u \dd x,
   \end{equation}
   where the Radon-Nikod\'ym derivative
   $\displaystyle\mathbf{\nu}=\frac{d(D\bar{\nchi}_u)}{d|D\bar{\nchi}_u|}$ is the measure-theoretic
   inner normal to the boundary of the phase $\{ x \in \Omega\, : \bar{\nchi}_u(x) =1 \} $.
 Finally, there holds
\begin{equation}
\label{final-ineq} |\partial^+ \phiste|(u) \geq \|A (u
-\bar{\nchi}_u) \|_{H^{-1}(\Omega)} = \| \nabla (u -\bar{\nchi}_u)
\|_{L^2 (\Omega)}\,.
\end{equation}
\end{proposition}
\begin{remark}
\upshape In view of the conditional continuity~\eqref{cond-cont} and
of Proposition~\ref{comp-sol}, for any generalized solution
$(u,\varphi)$ of the metric gradient flow of $\phiste$ there holds
$\varphi= \phi \circ u$. Hence we shall  simply
 refer to any generalized solution $(u,\varphi)$  as $u$.

\end{remark}
 \begin{proof} Property~\eqref{e:crucial-luck} (which is
somehow underlying some of the arguments in~\cite{Luckhaus90}) is
explicitly proved in~\cite[Sec.~5.2]{Rossi-Savare04}.
 It is  shown
in~\cite[Prop.~5.3]{Rossi-Savare04} that,  if $\frsbd \phiste (u)
\neq \emptyset$, then   $\frsbd \phiste (u)= \{ A(u-\bar{\nchi}_u)\}
$. It is straightforward to prove that there also holds

\begin{equation}
\label{cuore}
 |\partial
\phiste|(u) \leq \| A(u-\bar{\nchi}_u) \|_{H^{-1} (\Omega)}.
\end{equation}

 On the
other hand, we fix $w \in L^2 (\Omega)$ and notice that
\[
\begin{aligned}
\Lambda (u,w) : & = \limsup_{h \downarrow 0} \frac{\left(\phiste(u)
- \phiste(u+hw) \right)^+}{h \| w\|_{H^{-1}(\Omega)}} \\
&\geq \limsup_{h \downarrow 0} \frac{\displaystyle\left(\frac12\| u
-\bar{\nchi}_u \|_{L^2 (\Omega)}^2 +
\int_{\Omega}|\nabla\bar{\nchi}_u| - \frac12 \| u+hw
-\bar{\nchi}_u\|_{L^2 (\Omega)}^2 -
\int_{\Omega}|\nabla\bar{\nchi}_u| \right)^+}{h \|
w\|_{H^{-1}(\Omega)}}
\\ &  = \limsup_{h \downarrow 0} \frac{\displaystyle\left(-\frac{h^2}2 \|w
\|_{L^2 (\Omega)}^2 + h \int_{\Omega} w (u-\bar{\nchi}_u)
\right)^+}{h \| w\|_{H^{-1}(\Omega)}} =
\frac{\displaystyle\int_{\Omega} w (u-\bar{\nchi}_u)}{\|
w\|_{H^{-1}(\Omega)}}\,.
\end{aligned}
\]
Being $|\partial \phiste|(u)  \geq \sup_{w \in L^2 (\Omega)}\Lambda
(u,w) $,  by a density argument we easily conclude that $|\partial
\phiste|(u) \geq \|u-\bar{\nchi}_u \|_{H_0^{1} (\Omega)}$, so that,
 also in view of~\eqref{cuore},  with a slight abuse of notation  we may write
\begin{equation}
\label{e:intermediate-fr} |\partial \phiste|(u)= \| \frsbd
\phiste(u) \|_{H^{-1}(\Omega)} = \| u -\bar{\nchi}_u
\|_{H_0^{1}(\Omega)}\,.
\end{equation}
Using~\eqref{e:intermediate-fr} and arguing as
in~\cite[Prop.~5.3]{Rossi-Savare04} it is not difficult to check the
conditional continuity~\eqref{cond-cont}, and
that~\eqref{rapr-subdif} holds. The latter yields~\eqref{final-ineq}
through the general inequality~\eqref{eq:key-ineq}. Finally, the
weak Gibbs-Thomson law \eqref{eq:24bis} has been proved
in~\cite[Lemma~5.10]{Rossi-Savare04}.
\end{proof}
\noindent
\begin{corollary}
\label{prop:nuova} Under the above assumptions,  for every $u \in
\mathrm{GMM} (\phiste, u_0)$ (which is in particular a generalized
solution of the metric gradient flow of $\phiste$), and every
$\bar{\chi}_u \in \mathcal{M}(u)$ the pair $(\teta= u -\bar{\chi}_u,
\bar{\chi}_u)$ is a solution of the weak
formulation~\textrm{(\ref{eq:enbal}, \ref{eq:24})} of the
Stefan-Gibbs-Thomson Problem, fulfilling the Lyapunov
inequality~\eqref{e:lyap-ineq}.
\end{corollary}
\begin{remark}[Lyapunov inequality]
\label{rem:lyapx}
 \upshape  Indeed,
  through
inequality~\eqref{final-ineq},  the energy inequality~\eqref{1} for
generalized solutions translates into the Lyapunov
inequality~\eqref{e:lyap-ineq}. We may observe that, since in this
case $|\partial^+ \phiste | $ is not an upper gradient (from another
viewpoint, the chain rule~\eqref{e:chain-rule} w.r.t. $\lmsbd
\phiste$ does not hold), inequality~\eqref{e:lyap-ineq} yields the
most exhaustive
 \emph{energetic information}   on solutions  of   the
Stefan-Gibbs-Thomson problem.
\end{remark}
\paragraph{\bf Existence of the global attractor.}
We  deduce  from Theorem~\ref{teor:1} for generalized solutions the
following result, which in particular yields information  on the
long-time behavior of  the \emph{Minimizing Movements  solutions}
 of the Stefan problem with
the (weak) Gibbs-Thomson law. 
\begin{theorem}
\label{e:coro-rest-points} Under the above assumptions,  the
semiflow of the generalized solutions of the metric gradient flow
driven by the functional $\phiste$ from~\eqref{e:funz-sgt} admits
 a global attractor in the phase
space $D(\phiste)=L^2 (\Omega)$, endowed with the distance $
d_{\phiste}$ from~\eqref{e:phase-space}.
\end{theorem}
\noindent
\begin{proof} Since $\phiste$ complies with~\eqref{A3}--\eqref{A5},
in order to apply Theorem~\ref{teor:1} it remains to check that the
set of the rest points is bounded in the space $(L^2 (\Omega),
d_{\phiste})$. Indeed,
  it  follows
from the conditional continuity~\eqref{cond-cont} (cf.
with~\eqref{rest-point-bis}) and from  inequality~\eqref{final-ineq}
that for every  rest point $\bar{u}$ of the semiflow generated by
$\phiste$ there holds
\[
\bar{u} = \bar{\nchi}_{\bar{u}} \in \mathcal{M} (\bar{u}).
\]
Thus, $|\bar{u}|=1$ a.e.  in $\Omega$, and
\[
\phiste(\bar{u}) = \int_{\Omega}|\nabla \bar{\nchi}_{\bar{u}} | \leq
\frac12 \int_{\Omega} |\bar{u} -1|^2  \leq 2|\Omega|\,,
\]
where the second inequality ensues from choosing  $\nchi \equiv 1$
in the minimization which defines $\phiste$. This concludes the
proof.
\end{proof}
\subsection{Some partial results for the Mullins-Sekerka problem}
\label{ss:8.2} The functional $\phims: H^{-1}(\Omega) \to
[0,+\infty]$
\[
 \phims(\nchi):= \begin{cases} \displaystyle
\int_{\Omega} I_{\{-1,1\}}(\nchi)  \dd x
 + \int_{\Omega}
|D\nchi| & \text{if $\nchi \in \mathrm{BV}(\Omega;\{-1,1\})$,}
\\
\displaystyle +\infty & \text{if $\nchi \in  H^{-1}(\Omega)
\setminus \mathrm{BV}(\Omega;\{-1,1\})$}
\end{cases}
\]
complies with \eqref{A3}--\eqref{A5} (cf. with \cite[Lemma
3.1]{Luckhaus95}). On account of Theorem~\ref{teor:1}, the next step
 towards    the existence of the global attractor for the
generalized solutions $(\nchi,\varphi)$ driven by $\phims$ would be
the proof of the boundedness of the rest points set in the phase
space~\eqref{e:phase-space}. In this direction, we present   the
following result on properties of the local and weak relaxed slopes
of $\phims$, which we believe to have  an independent interest.


\begin{proposition}
\label{prop:ms} In the present setting,
 for every
\begin{equation}
\label{e:test-zeta}
 \zzeta\in C^2(\overline\Omega;\R^d) \ \  \text{with} \ \
\text{$\zzeta\cdot  \mathbf{n}=0$ on $\partial\Omega$}
\end{equation}
 there exists a
constant $C_\zeta \geq 0$, depending on $\| \zzeta \|_{W^{1,\infty}
(\Omega;\R^d)} $, such that
 the following
inequalities hold for all  $\nchi \in D(\phims)$ and $\varphi \in
\R$ with $\varphi \geq \phims(\nchi)$:
\begin{align}
\label{ineq-slope1} & |\partial \phims|(\nchi) \geq C_\zeta
\int_\Omega\big(\dive \zzeta-\mathbf{\nu}^T \,D\zzeta\,
       \mathbf{\nu}\big) \dd|D\nchi|\,;
\\
& \label{ineq-slope2} \!\!|\partial^- \phims|(\nchi,\varphi) \geq
C_\zeta \int_\Omega\big(\dive \zzeta-\mathbf{\nu}^T \,D\zzeta\,
       \mathbf{\nu}\big) \dd|D\nchi|\,.
\end{align}
\end{proposition}
\begin{proof}
We start by proving~\eqref{ineq-slope1} for a fixed  field $\zeta$
as in~\eqref{e:test-zeta}.
  Arguing as in the proof
of~\cite[Lemma~5.10]{Rossi-Savare04}, for $s \in \R$ we introduce
the flow $\Flow_s:\Omega\to\Omega$
   associated with $\zzeta$ via the ODE system
\begin{equation}
     \label{eq:62-rsss}
     \left\{
       \begin{aligned}
         \tfrac d{ds}\Flow_s(x)&=\zzeta(\Flow_s(x))\,,\\
         \Flow_0(x)&=x
       \end{aligned}
     \right. \qquad \text{for all $(s,x) \in \R \times \Omega$.}
   \end{equation}
Since $\Omega$ is regular (cf. with~\eqref{e:ambient8}) and
$\zzeta\cdot \nn=0$ on $\partial \Omega$,  $ \Flow_s(\Omega)=\Omega$
for all $s \in \R$. Further,
 the map $x \in \Omega \mapsto \Flow_s(x)$ is
a $\mathrm{C}^2$ diffeomorphism, with inverse $\Flow_{-s}: \Omega
\to \Omega $. Setting
 \[
  \Diff_s(x):=D_x\Flow_s(x),\quad J_s(x):=\det \Diff_s(x)
  \quad \text{for all $x \in \Omega$,}
  \]
it is immediate to check that
  \begin{equation}
  \label{eq:64}
  \left\{
    \begin{aligned}
      \tfrac d{ds}J_s(x)&=\mathrm{div}\left(\zzeta(\Flow_s(x))\right)J_s(x),\\
      J_0(x)&=1
    \end{aligned}
  \right. \qquad \text{for all $(s,x) \in \R \times \Omega$.}
\end{equation}
It follows from~\eqref{eq:62-rsss} that (recall that $i:\R^d \to
\R^d$ denotes the identity map)
\begin{equation}
\label{est1} \| \Flow_s - i\|_{L^\infty (\Omega;\R^d)} \leq
|s|\|\zzeta\|_{L^\infty(\Omega)}\,.
\end{equation}
In the same way, \eqref{eq:64} and the Gronwall lemma yield for all
$s\in [-1,1]$
\begin{equation}
\label{est2}
\begin{aligned}
&
 \| J_s \|_{L^\infty (\Omega)} \leq \exp(
|s|\|\mathrm{div}(\zeta)\|_{L^\infty (\Omega)}), \\
&
 \| J_s - 1\|_{L^\infty (\Omega)} \leq \|
\mathrm{div}(\zeta)\|_{L^\infty (\Omega)}  \exp(
|s|\|\mathrm{div}(\zeta)\|_{L^\infty (\Omega)})\,.
\end{aligned}
\end{equation}
 Then,  for every fixed  $\nchi \in D(\phims)$, we
consider its perturbation $\nchi_s : \Omega \to \{ -1,1\}$,   for $s
\in \R$, defined by
  \[
     \nchi_s(x):=\nchi(\Flow_{-s}(x)) \qquad \forall\, x \in
     \Omega\,.
   \]
   Now, $\nchi_s$ still belongs to $BV(\Omega;\{-1,1\})$ and it can
   be verified that
    $|D\nchi_s|$    coincides with the $(m-1)$-dimensional
   Hausdorff measure restricted to $S_s=\Flow_s(S)$  ($S$ being the essential boundary separating the phases).   Therefore,
   the first variation formula for the area functional
   (see, e.g.,~\cite[Thm. 7.31]{Ambrosio-Fusco-Pallara00})
   yields
   \[
     \frac{d}{ds}\Big(\int_\Omega |D\nchi_s|\Big)_{s=0}=
     \int_\Omega       \big(\dive \zzeta-\nnu^T \,D\zzeta\,
     \nnu\big) \dd|D\nchi|.
   \]
Hence, we have the following chain of inequalities
\[
\begin{aligned}
|\partial \phims|(\nchi) & =\limsup_{\| v -\nchi\|_{H^{-1}(\Omega)}
\to 0} \frac{\left(\int_\Omega |D\nchi|-\int_\Omega |Dv|
\right)^+}{\| v -\nchi\|_{H^{-1}(\Omega)}}
 \\ & \geq \limsup_{s  \to 0}\left( \frac{\left(\int_\Omega
|D\nchi|-\int_\Omega |D\nchi_s| \right)^+}{s} \frac{s}{\| \nchi
-\nchi_s\|_{H^{-1}(\Omega)}}\right) \\  & \geq \limsup_{s  \to
0}\frac{s}{\| \nchi -\nchi_s\|_{H^{-1}(\Omega)}} \left(
 \int_\Omega       \big(-\dive \zzeta+\nnu^T \,D\zzeta\,
     \nnu\big) \dd|D\nchi|\right) \,.
\end{aligned}
\]
To conclude  for~\eqref{ineq-slope1},  we are going to show that
\begin{equation}
\label{e:final-step} \limsup_{s  \to 0}\frac{s}{\| \nchi
-\nchi_s\|_{H^{-1}(\Omega)}} \geq C_\zzeta\,,
\end{equation}
with $C_\zzeta >0$ if $\zzeta$ is not identically zero, depending on
$\| \zzeta \|_{W^{1,\infty} (\Omega;\R^d)} $. Indeed, to
 evaluate $\| \nchi -\nchi_s\|_{H^{-1}(\Omega)}$  we fix $v
\in H_0^1 (\Omega)$, and for later convenience extend it to a
function $\tilde{v} \in H^1 ( \R^d)$. There holds
\begin{equation}
\label{e:stanca}
\begin{aligned}
\pairing{H^{-1}(\Omega)}{H_0^{1}(\Omega)}{\nchi -\nchi_s}{v} &  =
\int_{\Omega} \left(\nchi(x) - \nchi(\Flow_{-s}(x)) \right) v(x) \dd
x\\ & = \int_{\Omega} \nchi(x)\left( v(x)  - v(\Flow_{s}(x))
|J_s(x)| \right) \dd x \\ &
\begin{aligned} =  I_{1} & +  I_{2}:= \int_{\Omega}
\nchi(x)\left( v(x)  - v(\Flow_{s}(x)) \right) \dd x
\\ & + \int_{\Omega}  \nchi(x) v(\Flow_{s}(x))(J_0(x)-|J_s(x)|)\dd x\,,
\end{aligned}
\end{aligned}
\end{equation}  where the second equality follows from the change of variable
formula.  
 Then, supposing $|s| \leq 1$ and setting
$\lambda(\zeta):= \|\zzeta\|_{L^\infty(\Omega)}$,
  we
estimate $I_{1}$  by means  the maximal function
$M_{\lambda(\zeta)}$ of $v$ (see Section~\ref{ss:appendix-c}). In
view of Lemma~\ref{lemma:c-1}, there holds
\begin{equation}
\label{e:serve1} \|M_{\lambda(\zeta)} \nabla v\|_{L^2
(\Omega;\R^d)}^2 \leq  C(d,2) \int_{B_{\bar{r}+ \lambda(\zeta)}(0)}
|\nabla \tilde{v}(x)|^2 \dd x =  C(d,2) \| \nabla v
\|_{L^2(\Omega)}^2
 \end{equation}
where $\bar{r}>0$ is such  that $\Omega \subset B_{\bar{r}}(0) $. In
the same way, changing variables  we have
\begin{equation}
\label{e:serve2}
\begin{aligned}
 \int_{\Omega} \left|M_{\lambda(\zeta)} (\nabla
v(\Flow_{s}(x)))\right|^2 \dd x &  = \int_{\Omega}
\left|M_{\lambda(\zeta)} (\nabla v(y))\right|^2 |J_{-s}(y)|\dd y\\ &
\leq \exp( |s|\|\mathrm{div}(\zeta)\|_{L^\infty (\Omega)})  C(d,2)
\| \nabla v \|_{L^2(\Omega)}^2\,,
\end{aligned}
\end{equation}
the latter inequality from the first of~\eqref{est2}.
 Using~\eqref{est1},
Lemma~\ref{lemma:c-2} and \eqref{e:serve1}--\eqref{e:serve2},
 we  thus have
\begin{equation}
\label{estima1}
\begin{aligned}
 |I_{1} | &  \leq C(d) \int_{\Omega}
|\Flow_{s}(x) -x| \left ( M_{\lambda(\zeta)} (\nabla v(x)) +
M_{\lambda(\zeta)} (\nabla v(\Flow_{s}(x)))  \right) \dd x
\\ &
\leq  C(d,2)  C(d,2)^{1/2} \| \Flow_s - i\|_{L^\infty
(\Omega;\R^d)}\| \nabla v \|_{L^2(\Omega)} \left( 1+ \exp( \frac12
|s|\|\mathrm{div}(\zeta)\|_{L^\infty (\Omega)}) \right)\\ & \leq C_d
 \|\zzeta\|_{L^\infty(\Omega)} \left( 1+ \exp( \frac12
|s|\|\mathrm{div}(\zeta)\|_{L^\infty (\Omega)}) \right)
 |s| \|v \|_{H_0^1 (\Omega)}\,,
\end{aligned}
\end{equation}
the constant $C_d>0$ only depending on the dimension.  As for $I_2$
we have
\begin{equation}
\label{estima2} |{I}_{2} | \leq \|v \|_{L^2 (\Omega)} \|J_s -1
\|_{L^\infty(\Omega)} \leq   \| \mathrm{div}(\zeta)\|_{L^\infty
(\Omega)}  \exp( |s|\|\mathrm{div}(\zeta)\|_{L^\infty (\Omega)})\|v
\|_{L^2 (\Omega)}\,.
\end{equation}
Collecting \eqref{e:stanca} and \eqref{estima1}--\eqref{estima2}, we
immediately infer \eqref{e:final-step}.

In order to prove~\eqref{ineq-slope2} for all $\nchi \in D(\phims)$
and $\varphi \geq \phims(\nchi)$,
 we argue in the following way:
according to the definition of $|\partial^- \phims|$
 we fix a sequence $\{\nchi_k \} \subset D(\phims)$ converging to
 $\chi$ in $H^{-1}(\Omega)$ and such that
 \begin{equation}
 \label{aa}
\left||\partial \phims|(\nchi_k) -  |\partial^-
\phims|(\nchi,\varphi)\right|\leq \frac1k, \qquad \left|\int_\Omega
|D\nchi_k| - \varphi \right| \leq \frac1k\,.
\end{equation}
It follows from the latter inequality
and~\cite[Thm.~1.5.9]{Ambrosio-Fusco-Pallara00} that  $\{
D\nchi_k\}$ has a  weakly$^*$ converging subsequence   in the sense
of measures. Therefore,
   $D\nchi_k\weakto^* D\nchi$ and, thanks to
Reshetnyak Theorem
  \cite[Thm. 2.38]{Ambrosio-Fusco-Pallara00}, there holds
  \[
\liminf_{k \to \infty} \left( \int_\Omega\big(\dive
\zzeta-\mathbf{\nu}^T \,D\zzeta\,
       \mathbf{\nu}\big) \dd|D\nchi_k| \right) \geq \int_\Omega\big(\dive \zzeta-\mathbf{\nu}^T \,D\zzeta\,
       \mathbf{\nu}\big) \dd|D\nchi|
  \]
Thus, using the first of~\eqref{aa}
 we can pass to the limit
  in \eqref{ineq-slope1} as $k\to+\infty$ and
  conclude~\eqref{ineq-slope2}.
\end{proof}
\begin{remark}
\label{rmk:final} \upshape As a consequence of
Proposition~\ref{prop:ms}, for every rest point
$(\bar{\nchi},\bar{\varphi}) $ of the semiflow of the generalized
solutions driven by the functional $\phims$ there holds
\begin{equation}
\label{e:bb}
\begin{gathered}
\int_\Omega\big(\dive \zzeta-\mathbf{\bar{\nu}}^T \,D\zzeta\,
       \mathbf{\bar{\nu}}\big) \dd|D\bar{\nchi}|=0
       \quad
       \text{for all }\zzeta\in C^2(\overline\Omega;\R^d) \ \ \text{with
$\zzeta\cdot  \mathbf{n}=0$ on $\partial\Omega$}
\end{gathered}
\end{equation}
($\displaystyle\mathbf{\bar{\nu}}=\frac{d(D\bar{\nchi})}{d|D\bar{\nchi}|}$
being the measure-theoretic
   inner normal to the boundary of the phase $\{\bar{\nchi} \equiv 1\}$).
    This
 suggests that a way to prove that the generalized semiflow associated with $\phims$ complies with
  condition~\eqref{A12}
 could be to  deduce  from condition~\eqref{e:bb} some universal bound for
 every rest point $\bar{\chi}$. However, this  presently  remains an open
 problem.
\end{remark}
\appendix
\section{Generalized semiflows} \label{s:a2}
\noindent \setcounter{equation}{0}
 In the following, we  shall denote by
 $(\cx,\dcx)$  a (not necessarily complete) metric space. We
recall that the {Hausdorff semidistance} or {\it excess}
$e_{\cx}(A,B)$ of two non-empty subsets $A,  B \subset \cx$ is given
by $e_{\cx}(A,B):= \sup_{a \in A} \inf_{b \in B} \dcx(a,b)$. For all
$\eps>0$, we also denote by $ B(0,\eps)$ the ball $ B(0,\eps) :=\{x
\in X \ : \ \dcx(x,0) < \eps\}$, and by $N_\eps (A):= A + B(0,\eps)$
the $\eps$-neighborhood of  a subset $A$.
\begin{definition}[Generalized semiflow]
\label{def:generalized-semiflow}
 A \emph{generalized semiflow}
$\sfl$ on $\cx$ is a family of maps $g:[0,+\infty) \to \cx$
(referred to as \emph{solutions}), satisfying:
\begin{align}
&   \text{\emph{(Existence)}}\nonumber\\
&\qquad\text{For any $g_0 \in \cx$ there exists at
least one $g \in \sfl$ with $g(0)=g_0.$}\tag{H1}\label{H1}\\
&{}\nonumber\\
&\text{\emph{(Translates of solutions are solutions)}}\nonumber\\
&\qquad\text{For any $g \in \sfl$ and $\tau \geq 0$, the map $g^\tau
(t):=g(t+\tau),$ $t
\in [0,+\infty),$}\nonumber\\
&\qquad\text{belongs to $\sfl$.}\tag{H2}\label{H2}\\
&{}\nonumber\\
&\text{\emph{(Concatenation)}}\nonumber\\
&\qquad\text{For any $ g $, $ h \in \sfl$ and $t \geq 0$  with
$h(0)=g(t)$, then  $z \in \sfl$, $ z$
being}\nonumber\\
&\qquad\text{the map defined by $ z(\tau):=  g(\tau)$ if $0 \leq
\tau
\leq t,$ and $ h(\tau-t)$ if $ t <\tau$.}\tag{H3}\label{H3}\\
&{}\nonumber\\
&\text{\emph{(Upper semicontinuity with respect to initial data)}}\nonumber\\
&\qquad\text{If $\{g_n \} \subset \sfl$ and $g_n (0) \to g_0, $ then
there
exist a subsequence $\{g_{n_k}\}$ of }\nonumber\\
&\qquad\text{$\{g_n \}$ and $g \in \sfl$ such that $ g(0)=g_0$ and
$g_{n_k}(t) \to g(t)$ for all $t \geq 0.$}\tag{H4}\label{H4}
\end{align}

Furthermore, a generalized semiflow $\sfl$ may fulfill the following
\emph{measurability} and \emph{continuity}  properties:
\begin{itemize}
\item[(C0)] Each $g \in \sfl$ is strongly measurable on
$(0,+\infty),$ i.e. there exists a sequence $\{f_j\}$ of measurable
countably valued maps converging  to $g$ a.e. on $(0,+\infty).$
\item[(C1)] Each $g \in \sfl$ is continuous on $(0,+\infty).$
 \item[(C2)]
For any $\{g_n \} \subset \sfl$ with  $g_n (0) \to g_0, $ there
exist a subsequence $\{g_{n_k}\}$ of $\{g_n \}$ and $g \in \sfl$
such that $g(0)=g_0$ and $ g_{n_k} \to g$ uniformly in compact
subsets of $(0,+\infty).$
 \item[(C3)]Each $ g  \in \sfl $ is continuous on $ [0,+\infty).$
\item[(C4)] For any $\{g_n \} \subset \sfl$ with  $g_n (0) \to g_0, $ there exists
a subsequence $\{g_{n_k}\}$ of $\{g_n \}$ and $g \in \sfl$ such that
$g(0)=g_0$ and $ g_{n_k} \to g$ uniformly in compact subsets of
$[0,+\infty).$
\end{itemize}
\end{definition}

\paragraph{\bf 
$\boldsymbol \omega$-limits and attractors.} Given a generalized
semiflow $\sfl$ on $\cx$, we introduce for every $t \geq 0 $ the
operator ${T}(t): 2^\cx \to 2^\cx$ defined by
 \begin{equation}
 \label{eq:operat-T}
  {T}(t)E:=\{ g(t) \ : \
g \in \sfl \ \  \text{with} \ \ g(0) \in E\}, \quad E \subset \cx.
 \end{equation}
 The family of operators $\{{T}(t)\}_{t \geq 0}$ defines a semigroup on $2^\cx$, i.e., it
  fulfills  the following
 property
 \[
   { T}(t+s) B = { T}(t){ T}(s) B \quad \forall s,t \geq 0 \quad\forall  B \subset
   \cx.
 \]
Given  a solution $ g \in \sfl$, we introduce 
its \emph{$\omega$-limit} $\omega(g)$ 
by
$$
\omega(g):= \{ x \in \cx \ : \ \exists \{t_n\}, \ t_n \to +\infty, \
\text{such that} \ \ g(t_n) \to x \}.
$$
We say that $w: \R \to \cx$ is  a \emph{complete orbit} if, for any
$s \in \R$, the translate
 map $w^s \in \sfl$ (cf. \eqref{H2}).
Finally, the \emph{$\omega$-limit} of $E$ is defined as
\begin{align}
\omega(E):=\big\{
 &x \in \cx \ : \ \exists \{g_n\} \subset \sfl \
 \text{such that $\{g_n (0)\} \subset E$,} \nonumber\\
  &\text{$\{g_n (0)\}$ is bounded,
and} \ \  \exists t_n \to +\infty \ \text{with $ g_n (t_n) \to x$}
\big\}.\nonumber
\end{align}
Given subsets $F, E \subset \cx$, we say that $F $ \emph{attracts}
$E$ if
$$e_{\cx}(T(t)E,F) \to 0 \ \ \text{as} \ \  t \to
+\infty.$$
 Moreover, we say that $F $ is \emph{positively
invariant} if $T(t)F\subset F$ for every $t \geq 0$, that $F$ is
\emph{quasi-invariant} if for any $v \in F$ there exists a complete
orbit $w$ with $w(0)=v$ and $w(t) \in F$ for all $t \in \R$, and
finally that $F$ is \emph{invariant} if $T(t)F = F$ for every $t
\geq 0$ (equivalently, if it is both positively and
quasi-invariant).
\begin{definition}[Global Attractor]
\label{def:w-attract}
 Let $ \sfl$ be  a {generalized semiflow}.
  We say that a non-empty set $\att$
is a \emph{global attractor} for $ \sfl $ if it is compact,
invariant, and attracts all bounded sets of $ \cx$.
\end{definition}

\begin{definition}[Point dissipative, asymptotically compact]
Let $ \sfl $ be a generalized semiflow. We say that $
  \sfl $ is \emph{point dissipative} if there exists a bounded
set $B_0 \subset \cx$ such that for any $g \in \sfl$ there exists
$\tau \geq 0$ such that $g(t) \in B_0 $ for all $t \geq \tau$.
Moreover, $\sfl$ is \emph{asymptotically compact} if for any
sequence $\{g_n\} \subset \geneset$ such that $\{g_n (0)\}$ is
bounded
 and for any sequence $t_n \to + \infty $, the
sequence $\{g_n (t_n)\}$ admits a convergent subsequence.
\end{definition}

\paragraph{\bf Lyapunov function.} We say that a complete orbit $g \in \sfl$ is \emph{stationary}
if there exists $x \in \cx$ such that $g(t)=x$ for all $t \in \R$.
Such $x$ is then called a \emph{rest point}.  Note that the set of
rest points of  $\sfl $, denoted by $Z(\sfl)$, is closed in view of
\eqref{H4}.
\begin{definition}[Lyapunov function] $V: \cx \to \R$ is
 a \emph{Lyapunov function} for $\sfl$ if: $V$ is
continuous, $V(g(t)) \leq V(g(s))$ for all $g \in \sfl$ and $0 \leq
s \leq t$ (i.e., $V$ decreases along solutions), and, whenever the
map $t \mapsto V(g(t))$ is constant for some complete orbit $ g$,
then $ g $ is a stationary orbit.
\end{definition}

We say that a global attractor $\att$ for $\sfl$ is \emph{Lyapunov
stable} if for any $\eps>0$ there exists $\delta>0$ such that for
any $E \subset \cx$ with $e_{\cx}(E,\att) \leq\delta, $ then
$e_{\cx}(T(t)E,\att) \leq\eps$ for all $t \geq 0.$

 Finally, we recall the following two results, providing
necessary and sufficient conditions for a semiflow to admit a global
attractor, cf. \cite[Thms.~3.3,~5.1,~6.1]{Ball97}.
\begin{theorem}
\label{thm:ball1} A generalized semiflow $\sfl$ has a global
attractor  if and only if it is point dissipative and asymptotically
compact. Moreover, the attractor $\att$ is unique, it is the maximal
compact invariant subset of $\cx$, and it  can be characterized as
\[
\att= \cup \{\omega(B) \ : \ \text{$B \subset \cx$,
bounded}\}=\omega(\cx).
\]
Besides,  if $\sfl $ complies with \emph{(C1)} and \emph{(C4)}, then
$\att$ is Lyapunov stable.
\end{theorem}

\begin{theorem}
\label{thm:ball2} Assume  that $\sfl $ is asymptotically compact,
admits a Lyapunov function $V$, and that the sets of its rest points
$\rest$ is bounded. Then, $\sfl$ is also point dissipative, and thus
admits a global attractor $\att.$ Moreover,
\[
 \omega(u) \subset \rest \quad \text{for all
trajectories $ u \in \sfl$.}
\]
 \end{theorem}

\section{The proof of Proposition~\ref{prop:key1} revisited}
\noindent Preliminarily, we state and prove the following
\begin{lemma}
\label{l:euler} Under the assumptions of
Proposition~\ref{prop:key1}, let $z_j^k$ fulfill
\begin{equation}
\label{e:minimiz-bis}
 z_j^k \in \argmin_{v \in \Banach}
\left\{ \frac{\|v -u_k\|^2}{2r_j^k} + \phi(v) \right\}\,.
\end{equation}
 Then,
\begin{equation}
\label{e:euler-repeated}
 \rmJ_2 \left(\frac{z_j^k-u_k}{r_j^k}
\right) + \slmsbd\phi(z_j^k) \ni 0.
\end{equation}
\end{lemma}
\begin{proof}
To check~\eqref{e:euler-relaxed} we notice that, thanks
to~\cite[Lemma~2.32, p.~214]{Mordukhovich06},
\[
\forall\, \eta >0 \ \ \exists\, z^1_\eta \in \Banach,  \ z^2_\eta
\in D(\phi)\, : \quad \begin{cases} \displaystyle \|z^1_\eta-z_j^k\|
+ \|z^2_\eta-z_j^k \| &\leq \eta\,,
\\
\displaystyle | \|z^1_\eta- u_k\|^2 -\|z_j^k-u_k \|^2| & \leq  2\eta
r_j^k\,,
\\
\displaystyle | \phi(z^2_\eta) - \phi(z_j^k) | &\leq \eta\,,
\end{cases}.
\]
 and
\[ \exists \, w_\eta \in \rmJ_2 \left(\frac{z^1_\eta-u_k}{r_j^k}
\right), \ \exists\,\xi_\eta \in \frsbd\phi(z^2_\eta), \ \exists\,
\zeta_\eta \in \Banach^*, \ \|\zeta_\eta\|_* \leq \eta, \  \
\text{s.t.} \ \ w_\eta + \xi_\eta + \zeta_\eta=0\,. \]
 Choosing $\eta=1/n$, we find
sequences $\{z^1_n\}$, $\{z^2_n\}$, $\{w_n\}$, $\{\xi_n\}$, and $\{
\zeta_n \}$ fulfilling
\begin{equation}
\label{zeta-conv}
\begin{cases}
& \displaystyle \zeta_n \to 0 \qquad \text{ in $\Banach^*$,}
\\
\displaystyle & z^1_n \to z_j^k \qquad \text{ in $\Banach$,}
\\
\displaystyle & z^2_n \to z_j^k \qquad \text{ in $\Banach$,  with $
\phi(z^2_n) \to \phi(z_j^k ) $.}
\end{cases}
\end{equation}
Furthermore, being $w_n \in \rmJ_2 ((z^1_n-u_k)/r_j^k) $, thanks to
the second of~\eqref{zeta-conv},  we have  $\sup_n \|w_n\|_* \leq
C$,  for a positive constant $C$ only depending on $\|u_k\|$ and
$\|z_j^k\|$. Thus, there exists $\bar{w} \in \Banach^* $ such that,
up to a subsequence,  $w_n \weaksto \bar{w} $ in $\Banach^*$ as $ n
\to \infty$. By the strong-weak$^*$ closedness of $\rmJ_2$, we find
that $\bar{w} \in \rmJ_2 ((z^k_j-u_k)/r_j^k)$.  On the other hand,
by the definition of $\rmJ_2$ and again  by the second
of~\eqref{zeta-conv},
\[
\| w_n \|_*^2 = \left\langle w_n, \frac{z^1_n-u_k}{r_j^k}\right
\rangle \to \left\langle \bar{w}, \frac{z^k_j-u_k}{r_j^k} \right
\rangle
 = \|\bar{w}\|_*^2\,,
\]
so that we ultimately deduce that $w_n \to \bar{w} $ in $\Banach^*$
as $ n \to \infty$. By the first of \eqref{zeta-conv}, we
 conclude that $\xi_n \to -\bar{w}  $ in $\Banach^*$. Combining
 the latter information with the third of \eqref{zeta-conv}, we find
 from the definition of strong limiting subdifferential that $-\bar{w}  \in
 \slmsbd\phi(z_j^k)$, so that \eqref{e:euler-relaxed} ensues.
 \end{proof}
\paragraph{\bf Proof of Proposition~\ref{prop:key1}.} To
prove~\eqref{eq:key-ineq}, like in Section~\ref{s:6} we exhibit a
sequence $\{u_k\}\subset \Banach$ for which~\eqref{e:by-definition}
holds and, correspondingly, sequences $\{r_j^k\}_j$ and $z_j^k $ as
in~\eqref{e:minimiz-bis}, fulfilling \eqref{convs}. The only
difference with respect to the argument developed in
Section~\ref{s:6} is that, without~\eqref{frechet-renorm}, the Euler
equation corresponding to~\eqref{e:minimiz-bis}
is~\eqref{e:euler-repeated}, so that
\begin{equation}
\label{appendix1}
 \exists\, w_j^k \in \rmJ_2
\left(\frac{z_j^k-u_k}{r_j^k} \right) \cap \left(-
\slmsbd\phi(z_j^k)\right)
 \end{equation}
 which also fulfills
$ \ls^2(u_k) = \lim_{j \up \infty}\|w_j^k \|_*^2$. Then, again  with
a diagonalization procedure  we find  a subsequence $w_k$ with  $w_k
\weaksto w$ in $\Banach^*$. Using~\eqref{appendix1} and   the
closure formula~\eqref{e:s-w-closure}, we again arrive at \[ -w \in
\lmsbd \phi(u)\,,
\]
 and conclude the proof of~\eqref{eq:key-ineq} along the very same
 lines as in the proof of Proposition~\ref{prop:key1}.
\fin
\section{Maximal functions}
\label{ss:appendix-c} We recall the definition of the \emph{local
maximal function} of a locally finite measure and two related
results, referring e.g. to~\cite{Stein70} for  all details.
\begin{definition}
\label{def:maximal} Let $\mu$ be a (vector-valued) locally finite
measure. For every $\lambda>0$ the \emph{local maximal function}
$\mathrm{M}_{\lambda}\mu$ of $\mu$ is defined by
\[
\mathrm{M}_{\lambda}\mu(x):= \sup_{0<r<\lambda}
\frac{|\mu|(B_r(x))}{\mathscr{L}^d(B_r(x))} \qquad \text{for all $x
\in \R^d$.}
\]
In particular, when $\mu=f \mathscr{L}^d$ for some $f \in
L_{\mathrm{loc}}^1 (\R^d;\R^d)$, we shall use the notation
$\mathrm{M}_{\lambda} f$.
\end{definition}
\begin{lemma}
\label{lemma:c-1} \label{l:estimate} For all  $1<p<\infty $  there
exists a constant  $C(d,p)>0$  such that for every $f \in
L^p(\R^d;\R^d)$ there holds \[ \int_{B_r(0)} |\mathrm{M}_{\lambda}
f(x)|^p \dd x \leq  C(d,p) \int_{B_{r+\lambda}(0)} |f(x)|^d \dd x
\qquad \text{for all $r>0.$}
\]
\end{lemma}
\begin{lemma}
\label{lemma:c-2}
 There exists a constant  $C(d)>0$   such that
for every $\lambda >0 $ and $v \in \mathrm{BV} (\R^d)$ there holds
\[
|v(x) - v(y)| \leq C(d) |x-y| \left( M_\lambda \nabla v(x) +
M_\lambda \nabla v(y) \right)
\]
for all $x,y \in \R^d \setminus N_v$ (with $N_v \subset \R^d$ a
negligible set) with $|x-y| \leq \lambda$.
\end{lemma}

\bibliographystyle{alpha}

\begin{thebibliography}{GMT80}
\bibitem[Luc90]{Luckhaus90}
S.~Luckhaus.  \newblock Solutions for the two-phase {S}tefan problem
with the
  {G}ibbs-{T}homson {L}aw for the melting temperature. \newblock \emph{European J. Appl. Math.},   1:  101--111, 1990.
\end{thebibliography}
\def\cprime{$'$} \def\cprime{$'$}

\end{document}